\documentclass[a4paper,11pt,reqno]{article}
\usepackage[body=17cm,top=4cm,bottom=4cm]{geometry}
\usepackage[english]{babel}
 \usepackage{amsrefs}
\usepackage{amsmath,amsfonts,amsthm,amscd,amssymb,bbm,mathrsfs}
\usepackage[colorlinks=true]{hyperref}
\usepackage{graphicx}

\newcommand{\cB}{\mathscr{B}}
\newcommand{\cZ}{\mathscr{Z}}
\newcommand{\cR}{\mathscr{R}}

\newcommand{\ott}{[0,T]}

\newcommand{\sst}{\Sigma^*}

\DeclareMathOperator{\img}{Img}

\DeclareMathOperator{\expect}{{\mathbb E}}

\newcommand{\eps}{\varepsilon}

\makeatletter
\newcommand{\eqcolon}{\mathrel{\mathord{=}\raise.2\p@\hbox{:}}}
\newcommand{\coloneq}{\mathrel{\raise.2\p@\hbox{:}\mathord{=}}}
\makeatother
\newcommand{\der}{\delta}
\newcommand{\dd}{\mathrm{d}}

\newcommand{\RR}{\mathbb{R}}

\newcommand{\CC}{\mathscr{C}}
\newcommand{\CCC}{\CC\!\CC}

\newcommand{\norm}[1]{\lVert #1\rVert}

\newcommand{\E}{{\mathbb E}}

\newcommand{\R}{{\mathbb R}}

\newcommand{\cac}{\mathscr C}

\newcommand{\cj}{\mathscr J}

\newcommand{\cn}{\mathscr N}

\newcommand{\laa}{\Lambda}

\newcommand{\vp}{\varphi}

\newcommand{\lp}{\left(}
\newcommand{\rp}{\right)}
\newcommand{\lc}{\left[}
\newcommand{\rc}{\right]}
\newcommand{\lcl}{\left\{}
\newcommand{\rcl}{\right\}}

\newtheorem{theorem}{\protect\theoremname}[section]
\newtheorem{definition}[theorem]{\protect\definitionname}
\newtheorem{proposition}[theorem]{Proposition}
\newtheorem{lemma}[theorem]{Lemma}

\newtheorem{remark}[theorem]{Remark}
\newtheorem{corollary}[theorem]{Corollary}
\newtheorem{hypothesis}[theorem]{Hypothesis}

\newcommand{\theoremname}{}
\newcommand{\definitionname}{}

\addto\captionsenglish{%
  \renewcommand{\theoremname}{Theorem}%
  \renewcommand{\definitionname}{Definition}%
}
\addto\captionsfrench{%
  \renewcommand{\theoremname}{Théoreme}%
  \renewcommand{\definitionname}{Définition}%
}

\begin{document}

\title{Rough Sheets}
\author{K.~Chouk$^{a}$, M.~Gubinelli$^{a,b}$\\
{\small
(a) CEREMADE \& CNRS UMR 7534}\\{\small  Universit\'e Paris--Dauphine, France
}\\{\small
(b) Institut Universitaire de France
}\\{\small \texttt{\{chouk,gubinelli\}@ceremade.dauphine.fr}
}}
\date{\today}
\maketitle

\begin{abstract}
Rough sheets are two-parameter analogs of rough paths. In this work
the theory of integration over functions of two parameters is extended
to cover the case of irregular functions by developing an
appropriate notion of rough sheet. The main
application is to give a path by path construction of the stochastic integral in the plane and obtain a stratonovich change of variables formula.
\end{abstract}
\tableofcontents

\section{Introduction}
For processes indexed by one dimensional parameter Lyons' theory of
rough paths~\cite{LyonsBook,Lyons1998} provide an effective approach to the understanding of
relatively complex maps like the stochastic integral or the It\^o map
which sends the data of an SDE to its solution. These insights were
instrumental in proving new results or simplifying the proofs of some
important results in stochastic analysis. 
Just to mention some application: path-wise solutions of SDEs, the
support theorem for diffusions, differentiation of the It\^o map,  
the existence of path-wise
versions for stochastic currents and the solution of an ODE over
Brownian paths driven by a non-adapted vector-field. The theory, being
independent from notions like martingales, adaptness, etc\dots is
straightforwardly extended to other, more complex, noises like
fractional Brownian motions.   

\medskip
In~\cite{[Gubinelli-2004]} Gubinelli introduces the notion of \emph{controlled path}  in order to abstract the basic structures allowing the integration of rough signals and pointed out the existence of a map (the \emph{sewing map}) which turns out to be at the ``core'' of the integration process.  
\medskip

In the present paper we study a multi parameter extension to the controlled path theory of~\cite{[Gubinelli-2004]} and inspired by the structures which one encounters when trying to solve hyperbolic equations driven by irregular signals.

More precisely the example what we have in mind is the stochastic wave equation driven by a fractional Brownian sheet. This equation has been already studied in~\cite{Wave} when the Hurst parameters of the fractional Brownian sheet are strictly bigger than $\frac{1}{2}$ by using the two dimensional Young integration theory. In the cited paper the authors obtain a global existence result in this case in $(1+1)$-dimension. 

Our work is a (partially successful) attempt to extend this previous work for the case where the Hurst parameters are less than $\frac{1}{2}$ and for that we will give a definition of a \emph{rough sheet}, which is the basic objects underlying multi-parameter integration suitable to build a theory of path-wise integration over the fractional Brownian sheet. Unfortunately the construction of the rough integral in the multiparametrs setting give rise to a very complex algebraic structure which is not well understood at the moment. Due to this difficulties we are not able to construct a full--fledged theory capable of handling solutions to the stochastic wave equation 
 for that reason we will prefer to restrict ourself to the construction of the integral for a restricted family of integrands and to leave the study of the wave equation to further investigations. 
 
 More precisely the main result of this paper can be resumed in the following theorem 
\begin{theorem}
Let $\gamma_1,\gamma_2>1/3$, then there exist a complete metric space $\mathscr R^{\gamma_1,\gamma_2}$ and two continuous application $\mathscr I_a:C_b^8(\mathbb R)\times \mathscr R^{\gamma_1,\gamma_2}\to\CCC_{2,2}^{\gamma_1,\gamma_2}$ (see the equation~\eqref{eq:Holder space} for the exact definition of $\CCC_{2,2}^{\gamma_1,\gamma_2}$), $a=1,2$ such that :
$$
\mathscr I_1(\varphi,\mathbb X)_{s_1s_2;t_1t_2}=\int_{s_1}^{s_2}\int_{t_1}^{t_2}\varphi(x_{st})\partial_s\partial_t x_{st}\dd s\dd t
$$
and 
$$
\mathscr I_2(\varphi,\mathbb X)_{s_1s_2;t_1t_2}=\int_{s_1}^{s_2}\int_{t_1}^{t_2}\varphi(x_{st})\partial_sx_{st}\partial_tx_{st}\dd s\dd t
$$
for all $(s_1,s_2,t_1,t_2)\in[0,1]^4$ and for every smooth sheet $x$ with $\mathbb X$ is the rough sheet associated. Moreover if $x$ is a fractional Brownian sheet with Hust parameter $\gamma_1,\gamma_2>1/3$ then it can be enhanced in a rough sheet $\mathbb X$ and then we obtain in this case the following Stratonovich formula :
$$
\der \varphi(x)=\mathscr I_{1}(\varphi',\mathbb X)+\mathscr I_{2}(\varphi'',\mathbb X).
$$ 
\end{theorem}   

We remark that this result is used in~\cite{Ito plane} to deal with the multi-parametric Skorohod integral defined in~\cite{TVp1,TVp2} and that an explicit formula which give a link between the two notion of integration is obtained.

During the elaboration of this paper two other approaches have been developed to deal 
with singular parabolic SPDE using rough path ideas. 
The first one is due to P.~Imkeller, N.~Perkowsky and M.~Gubinelli~\cite{Para} 
in which the authors use Bony's paraproducts~\cite{bony} to set up a controlled structure which allow to handle product of distributions in enough generality to be useful to solve PDEs driven by singular signals. A second and more powerful approach  has been introduced by M.Hairer in~\cite{hairer_theory_2013} where the author builds a complete non-linear theory for distributions which can be analyzed in terms of a basic set of simple objects called a \emph{model}. 

These two extension of rough path theory are very efficient to deal with parabolic SPDE driven by a very irregular multi parameter noise but at the moment while seems possible to adapt these two approaches to handle hyperbolic problems it seems not very clear what would be the final form of a successful theory. In this respect our study and the objects we introduce seems to be needed, maybe in a slightly different form, in order to build a complete theory of hyperbolic rough equations similar to the non-linear stochastic heat equation. Our efforts then can be seen as a first step in this direction.

\medskip

\noindent

\textbf{Plan.} 
This note is structured as follows. In Sec.~\ref{sec:one-dim} we
recall the basic setup of~\cite{[Gubinelli-2004]} which allows to embed
the theory of rough paths in a theory of integration of
``generalized differentials'', called 
$k$-\emph{increments}.  In Sec.~\ref{sec:two-dim} we
introduce and study a complex of 2d increments (or
\emph{biincrements}) suitable
to analyze 2d integrals and show the existence of a
2d $\Lambda$-map and of an abstract integration theory (in the sense
of convergence of Riemann sums of particular biincrements).

In Sec.~\ref{sec:two-d-young} we use the theory outlined in
Sec.~\ref{sec:two-dim} to generalize Young theory of integration to two
dimensions. Like in the one-dimensional setting this should be seen as
a first (mostly pedagogical)  step towards a full theory of
rough sheets.

In Sec.~\ref{sec:2-d-dissection} we proceed to the dissection of a 2d
integral with the purpose of exposing the constituent elements of the
would-be rough sheet. All the computations will be done in the smooth
setting, emphasizing the algebraic aspects and the respective r\^oles
of the various objects.

In Sec.~\ref{sec:rough} the definition of the rough sheet will
be given, and on~\ref{sec:controlled-sheets} a space of sheets \emph{controlled} by a rough sheet,
will be introduced and we will show how to obtain an integration
theory for them.

In Sec~\ref{sec:stability} we show that for a smooth function $\varphi$ that 
$\varphi(x)$ is controlled by $x$ if $x$ is a rough sheet and then obtain some continuity result
for the integral constructed in this case using the procedure developed in~\ref{sec:rough}.

In Sec~\ref{sec:construction} we show that the fractional Brownian motion can be enhanced in a rough sheet and then we obtain a Stratonovich change of variable formula in 
this case.

\section{Algebraic integration in one dimension}
\label{sec:one-dim}

The integration theory introduced in~\cite{[Gubinelli-2004]} is based on an
algebraic structure, which turns out to be useful for computational
purposes, but has also its own interest. Since this setting is quite
non-standard, compared with the one developed in~\cite{[LionsStFlour]},
and since it lay at the base of our approach to 2d integrals 
 we will recall briefly here its main features.  

\subsection{Increments}

Let $T>0$ be an arbitrary positive real number.  For any
vector space $V$
we introduce a cochain complex $(\CC_*(V),\delta)$ as follows. a
\emph{$k$-increment} with values in $V$ is a function
$g : [0,T]^{k} \to V$, such that $g_{t_1 \cdots t_{k}} = 0$
whenever $t_i = t_{i+1}$ for some $i=1,\dots,k$. Denote with
$\CC_k(V)$ the corresponding set. On $k$-increments, define a
the following coboundary operator $\delta$:
\begin{equation}
  \label{eq:coboundary}
\delta : \CC_k(V) \to \CC_{k+1}(V) \qquad 
(\delta g)_{t_1 \cdots t_{k+1}} = \sum_{i=1}^{k+1} (-1)^{i+1} g_{t_1
  \cdots \hat t_i \cdots t_{k+1}}  
\end{equation}
where $\hat t_i$ means that this particular argument is omitted.  It
is easy to verify that $\delta \delta = 0$. We will denote $\cZ
\CC_k(V) = \CC_k(V) \cap \text{Ker}\delta$ and $\cB \CC_k(V) :=
\CC_k(V) \cap \text{Im}\delta$, respectively the spaces of
\emph{$k$-cocycles} and of \emph{$k$-coboundaries} following standard
conventions of homological
algebra. We will write $\CC_k$  when the underlying vector space is $\RR$. 

Some simple examples of actions of $\der$ are obtained by letting 
$g\in\CC_1(V)$ and $h\in\CC_2(V)$. Then, for any $t,u,s\in\ott$, we have
\begin{equation}
\label{eq:simple_application}
  (\der g)_{ts} = g_t - g_s, 
\quad\mbox{ and }\quad  
(\der h)_{tus} = h_{ts}-h_{tu}-h_{us}.
\end{equation}

The complex $(\CC_*(V),\delta)$ is \emph{acyclic}, i.e. 
$\cZ \CC_{k+1}(V) = \cB \CC_{k}(V)$ for any $k\ge 0$ or otherwise stated, the sequence
\begin{equation}
\label{eq:exact_sequence}
 0 \rightarrow \RR \rightarrow \CC_1(V)
 \stackrel{\der}{\longrightarrow} \CC_2(V) \stackrel{\der}{\longrightarrow} \CC_3(V) \stackrel{\der}{\longrightarrow} \CC_4(V) \rightarrow \cdots 
\end{equation}
is exact.

This exactness implies that all the elements
$h \in\CC_2(V)$ such that $\der h= 0$ can be written as $h = \der f$
for some (non unique) $f \in \CC_1(V)$. Thus we get a heuristic 
interpretation of $\der |_{\CC_1(V)}$:  it measures how much a
given 1-increment  is far from being an {\it exact} increment of a
function (i.e. a finite difference).

For our discussion only $k$-increments with $k \le 3$ will be
relevant. When $V$ is a Banach space with norm $|\cdot|$ we measure the size of these increments by H\"older norms
defined in the following way: for $f \in \CC_2(V)$ let
$$
\norm{f}_{\mu} \equiv
\sup_{s,t\in\ott}\frac{|f_{st}|}{|t-s|^\mu},
\quad\mbox{and}\quad
\CC_2^\mu(V)=\lcl f \in \CC_2(V);\, \norm{f}_{\mu}<\infty  \rcl.
$$
In the same way for $h \in \CC_3(V)$ set
\begin{eqnarray}
  \label{eq:normOCC2}
  \norm{h}_{\gamma,\rho} &=& \sup_{s,u,t\in\ott} 
\frac{|h_{sut}|}{|u-s|^\gamma |t-u|^\rho}\\
\|h\|_\mu &\equiv &
\inf\left \{\sum_i \|h_i\|_{\rho_i,\mu-\rho_i} ;\, h =
 \sum_i h_i,\, 0 < \rho_i < \mu \right\} ,\nonumber
\end{eqnarray}
where the last infimum is taken over all sequences $\{h_i \in \CC_3(V) \}$ such that $h
= \sum_i h_i$ and for all choices of the numbers $\rho_i \in (0,z)$.
Then  $\|\cdot\|_\mu$ is easily seen to be a norm on $\CC_3(V)$, and we set
$$
\CC_3^\mu(V):=\lcl h\in\CC_3(V);\, \|h\|_\mu<\infty \rcl.
$$

Let $\CC_3^{1+}(V) = \cup_{\mu > 1} \CC_3^\mu(V)$. Analogous meaning
should be given to $\cZ \CC_2^\mu(V), \dots$.

From now on $V$ will be a generic Banach space.

The following proposition is a basic result which is at the core of
 our approach to path-wise integration:
\begin{proposition}[The $\Lambda$-map]
\label{prop:Lambda}
There exists a unique linear map $\Lambda: \cZ \CC^{1+}_3(V)
\to \CC_2^{1+}(V)$ such that 
$
\delta \Lambda  = 1_{\cZ \CC_3(V)}.
$
Furthermore, for any $\mu > 1$ this map is continuous from $\cZ \CC^{\mu}_2(V)$
to $\CC_1^{\mu}(V)$ and we have
\begin{equation}\label{ineqla}
\|\Lambda h\|_{\mu} \le \frac{1}{2^\mu-2} \|h\|_{\mu} ,\qquad h \in  \cZ \CC^{1+}_2(V) 
\end{equation}
\end{proposition}

\begin{proof}
See~\cite{[Gubinelli-2004]}.
\end{proof}
\vspace{0.3cm}

Let
$
\cR^\mu(V) = \{ g \in \CC_2(V) : \der g \in \CC_3^{\mu}(V) \}
$. 
When $\mu > 1$
this is the subspace of 1-increments whose coboundary is small enough
to be in the domain of $\Lambda$. $\cR^{1+}(V)$ is defined as the
union of all $\cR^{\mu}(V)$ for $\mu > 1$.

An immediate implication of Prop.~\ref{prop:Lambda} is the following
algorithm for a canonical decomposition of the elements of $\cR^{1+}(V)$:

\begin{corollary}
\label{cor:lambda-inversion}
Take an element $g\in\cR^\mu(V)$ for
$\mu>1$. Then $g$ can be decomposed in a unique way as
$
g=\der f+ \Lambda \delta g
$,
where $f\in\CC_1(V)$. If moreover $g \in \CC_2^{1+}(V)$ then $f=0$ and $g =
\Lambda \delta g$: the coboundary $\delta$ is invertible (as a linear
map) in $\CC_1^{1+}(V)$.
\end{corollary}

\begin{proof}
By assumption there are no ambiguity to define $\Lambda\der g$  and by definition we have that $\der(g-\Lambda\der g)=0$ and then there exist $f\in\CC_1(V)$ such that 
$ \der f=g-\Lambda\der g$. In the case when $g\in\CC^{1+}_2$ we see that $f$ is a $\gamma$-H\"older function  for some $\gamma>1$ and then $f_t=f_0$ which gives our result.

\end{proof}

At this point the relation of the structure we introduced with
the problem of integration of irregular functions can be still quite
obscure to the non-initiated reader. However something interesting is
already going on and the previous corollary has a very nice
consequence which is the subject of the next corollary.

\begin{corollary}[Integration of small increments]
\label{cor:integration}
For any 1-increment $g\in\CC_2 (V)$, such that $\der g\in\CC_1^{1+}$
set
$
\delta f = (1-\Lambda \delta) g 
$,
then
$$
(\delta f)_{ts} = \lim_{|\Pi_{ts}| \to 0} \sum_{i=0}^n g_{t_i\, t_{i+1}}
$$
where the limit is over all partitions $\Pi_{ts} = \{t_0=t,\dots,
t_n=s\}$ of $[t,s]$ as the size of the partition goes to zero. The
1-increment $\delta f$ is the indefinite integral of the 1-increment $g$.
\end{corollary}

Even if the result is already in~\cite{[Gubinelli-2004]} we would like to repeat
the proof since it is quite illuminating and it will be source of
inspiration when proving a similar statement in the 2d setting.

\begin{proof} 

Just consider the equation $g = \delta f + \Lambda \delta g$ and write
\begin{equation*}
  \begin{split}
S_\Pi & = \sum_{i=0}^n g_{t_i\, t_{i+1}} 
=
\sum_{i=0}^n (\delta f)_{t_i\, t_{i+1}} 
+
\sum_{i=0}^n (\Lambda \delta g)_{t_i\, t_{i+1}} 
 =
(\delta f)_{ts}
+
\sum_{i=0}^n (\Lambda \delta g)_{t_i\, t_{i+1}} 
      \end{split}
\end{equation*}
and observe that, due to the fact that $\Lambda \delta g \in
\CC_1^{1+}(V)$ the last sum converges to zero. 
\end{proof}

\subsubsection{Computations in $\CC_*$}\label{cpss}

If $V$ is an associative algebra 
the complex $(\CC_*,\delta)$ is an (associative, non-commutative)
graded algebra once endowed with the following product:
for  $g\in\CC_n(V)$ and $h\in\CC_m(V)$ let  $gh \in \CC_{n+m-1}(V)$
the element defined by
\begin{equation}\label{cvpdt}
(gh)_{t_1,\dots,t_{m+n+1}}=
g_{t_1,\dots,t_{n}} h_{t_{n},\dots,t_{m+n-1}},
\quad
t_1,\dots,t_{m+n}\in\ott.
\end{equation}

The coboundary $\delta$ act as a graded derivation with respect to the
algebra structure. In particular we have the following useful properties.

\begin{proposition}\label{difrul}
The following differentiation rules hold:
\begin{enumerate}
\item
Let $g,h$ be two elements of $\CC_1(V)$. Then
\begin{equation}\label{difrulu}
\der (gh) = \der g\,  h + g\, \der h.
\end{equation}
\item
Let $g \in \CC_1(V)$ and  $h\in \CC_2(V)$. Then
$$
\der (gh) = -\der g\, h + g \,\der h, \qquad
\der (hg) = \der h\, g  +h \,\der g.
$$
\end{enumerate}
\end{proposition}

\begin{proof}
We will just prove (\ref{difrulu}), the other relations being equally trivial:
if $g,h\in\CC_1(V)$, then
$$
\lc \der (gh) \rc_{ts}
= g_th_t-g_sh_s
=g_t\lp h_t-h_s \rp +\lp  g_t-g_s\rp h_s\\
=g_t \lp \der h \rp_{ts}+ \lp \der h \rp_{ts} g_s,
$$
which proves our claim.

\end{proof}

The iterated integrals of smooth functions on $\ott$ are of course
particular cases of elements of $\CC$ which will be of interest for
us. Let us recall  some basic  rules for these objects:
consider $f,g\in\CC_1^\infty$, where $\CC_1^\infty$ is the set of
smooth functions from $\ott$ to $\R$. Then the integral $\int f dg$, 
that we will denote by  
$\cj(f dg )$, can be considered as an element of
$\CC_2^\infty$. That is, for $s,t\in\ott$, we set 
$$
 \cj_{ts}(f \dd g)
=
\left(\int f \dd g  \right)_{ts} = \int_s^t  f_u \dd g_u .
$$
The multiple integrals can also be defined in the following way:
given a smooth element $h \in \CC_2^\infty$ and $s,t\in\ott$, we set
$$
\cj_{ts}(h \dd g )\equiv
\left(\int h \dd g \right)_{ts} = \int_s^t h_{us} \dd g_u  .
$$
In particular, the double integral $\cj_{ts}( f^1 df^2 df^3)$ is defined, for $f^1,f^2,f^3\in\cac_0^\infty$, as
$$
\cj_{ts}( f^1 \dd f^2 \dd f^3)
=\lp \int f^1 \dd f^2 \dd f^3  \rp_{ts}
= \int_s^t \cj_{us}\lp f^1\, \dd f^2  \rp \dd f_u^3  .
$$
Now, suppose that the $n$th order iterated integral of $f^1 df^2\cdots
df^n$, still denoted by the expression $\cj(f^1  df^2\cdots df^n)$, has been defined for 
$f^1,f^2,\cdots, f^n\in\cac_1^\infty$. Then, if $f^{n+1}\in\cac_1^\infty$, we set
\begin{equation}\label{multintg}
\cj_{ts}(f^1 \dd f^{2}\cdots \dd f^n  \dd f^{n+1})
=
\int_s^t   \cj_{us}\lp f^1 \dd f^2\cdots \dd f^n \rp\, \dd f_u^{n+1},
\end{equation}
which defines the iterated integrals of smooth functions recursively.
Observe that a $n$th order integral $\cj(\dd f^1 \dd f^2 \cdots df^n)$ could be defined along the same lines.

\medskip

The following relations between multiple integrals and the operator $\der$ will also be useful in the remainder of the paper:
\begin{proposition}\label{dissec}
Let $f,g$ be two elements of $\CC_1^\infty$. Then, recalling the convention 
(\ref{cvpdt}); it holds that
$$
\der f = \cj( \dd f), \qquad 
\der\lp \cj(f \dd g)\rp = 0, \qquad 
\der\lp \cj (\dd g \dd f)\rp =  (\der g) (\der f) = \cj(\dd g) \cj(\dd f),
$$
and, in general,
\begin{equation}
  \label{eq:dissect-multiple}
 \der \lp \cj( \dd f^n \cdots \dd f^1)\rp  = 
 \sum_{i=1}^{n-1} 
\cj\lp \dd f^n \cdots \dd f^{i+1}\rp \cj\lp \dd f^{i}\cdots \dd f^1\rp.  
\end{equation}
\end{proposition}

\begin{proof}
Here again, the proof is elementary, and we will just show the third of these relations: we have, for $s,t\in\ott,$
$$
\cj_{ts} (\dd g \dd f)
= \int_s^t \dd g_u (f_u-f_s)
= \int_s^t \dd g_u f_u - K_{ts},
$$
with $K_{ts}=(g_t-g_s)f_s$. The first term of the right hand side is easily seen to be in $\ker\der_{\sst_2}$. Thus
$$
\der\lp \cj (\dd g \dd f)\rp_{tus} =
\lp\der K\rp_{tus} = [g_t-g_u][f_u-f_s],
$$
which gives the announced result.

\end{proof}

\subsubsection{Dissection of an integral}
\label{sec:dissection}

To grasp the algorithm underling the rough-path approach to integrals
over irregular functions we will exercise ourselves on the
deconstruction of a ``classic'' integral.  

With the notations of Sec.~\ref{cpss} in mind, we will split the integral
$\int \vp(x) \dd x=\cj(\vp(x) \dd x)$ for a smooth function $x\in\cac_1$
into ``more elementary'' components. This decomposition suggest the
right structure for the 1d rough paths. A similar exercise for 2d
integrals will be very important in understanding the correct
structure of the rough sheets.

\bigskip
The first idea one can have in mind in order to analyze $\cj(\vp(x) \dd x )$ is to
perform an expansion around the increment $\dd x$. By Taylor expansion
we have
\begin{equation}
  \label{eq:dec-1}
\cj\lp  \varphi(x) \dd x \rp = \varphi(x) \cj( \dd x)   + \cj\lp \dd\varphi(x)  \dd x \rp.
\end{equation}
The first term in the r.h.s will be considered ``elementary'' and not
elaborated further. Note that it is defined independently of the
regularity of $x$ since $\varphi(x) \cj( \dd x)  =  \varphi(x)\delta x$. 

Moreover,
as a 1-increment it is easy to see that the second term in the
r.h.s. of eq.~(\ref{eq:dec-1}) is smaller than the first but more
problematic
and we proceed to its \emph{dissection} by
the application of 
$\der$: invoking Proposition~\ref{dissec}, we get that
\begin{equation}
  \label{eq:dec-2}
\der\lp \cj (\dd\varphi(x) \dd x  )\rp =  \der (\varphi(x)) \der x ,
\end{equation}
Now the r.h.s. is well defined independently of the regularity of $x$
since
$$
[\der (\varphi(x)) \der x]_{tus} = (\vp(x_t)-\vp(x_u))(x_u-x_s).
$$

Since $x$ is smooth and assuming that $\vp$ is differentiable, then $ \cj (\dd\varphi(x) \dd x  ) \in \CC_1^{1+}$  (since actually it
belongs to $\CC_1^2$ by easy bounds on the iterated integral). 
Then as a consequence of Corollary~\ref{cor:lambda-inversion} we have that
\begin{equation}
  \label{eq:int-rep1}
 \cj (\dd\varphi(x) \dd x  ) = \Lambda \der \left[ \cj (\dd\varphi(x)
  \dd x  )\right] = - \Lambda[\der (\varphi(x)) \der x].  
\end{equation}
and, as a result, the following expression for the original integral holds
\begin{equation}\label{defyg}
\cj ( \varphi(x) \dd x  )
= \vp(x) \der x  
-\laa\lp \, \der(\varphi(x))\der x \rp = (1-\Lambda \delta)
[\varphi(x) \delta x].
\end{equation}

Eq.~(\ref{defyg}) shows that the ordinary integral on the l.h.s. is
equivalent to an expression in the r.h.s which does not depend any
more on any differentiability assumptions on $x$, indeed the r.h.s
makes sense, for example, when $\vp \in \text{Lip}$ and $x \in \CC^\gamma_0$
for any $\gamma > 1/2$: the only thing we have to check is that $\der
(\varphi(x)) \der x \in \CC_2^{1+}$ but under these assumptions we have
$$
|(\varphi(x_t)-\varphi(x_u))(x_u-x_s)| \le L_\varphi \|x\|_\gamma^2
\|\varphi(x)\|_\gamma |t-u|^\gamma |u-s|^\gamma 
$$
where $\|\cdot\|_\gamma$ is the ordinary $\gamma$-H\"older norm on
functions and $L_\varphi$ is the Lipshitz norm of $\varphi$. 
In this case we can \emph{define} the integral in the l.h.s. as being
equivalent to the well-defined r.h.s. and this new integral is
essentially the integral introduced by Young in~\cite{Young}. 
What is really relevant to our discussion is to note that the integral
can, in this case, be completely recovered from the 1-increment
$\varphi(x) \der x$. 

However, the procedure can be continued further on by the next step in
the Taylor expansion of the integral~(\ref{eq:dec-1}), which reads, for $s,t\in\ott$,
$$
\int_s^t \lc \dd_u \vp(x_u) \rc
=\int_s^t   \vp'(x_u) \dd x_u
= \vp'(x_s) [x_t-x_s] +\int_s^t  \left(\int_s^u \vp'(x_v) \dd x_v\right)  \dd x_u ,
$$
or according to the notations of Section~\ref{cpss},
\begin{equation}
  \label{eq:dec-3}
\der\vp(x)=
\cj \lp \dd\varphi(x)\rp  =  \cj \lp \varphi'(x) \dd x \rp
= \varphi'(x) \cj \lp \dd x \rp  
+ \cj \lp  d\varphi'(x) \dd x \rp.
\end{equation}
Injecting this equality in equation~(\ref{eq:dec-1}), thanks to 
(\ref{multintg}), we obtain
\begin{equation}
  \label{eq:dec-4}
\cj \lp \varphi(x) \dd x \rp = \varphi(x) \cj\lp \dd x\rp  
+ \varphi'(x) \cj\lp \dd x \dd x \rp 
+\cj\lp  \dd \varphi'(x) \dd x  \dd x  \rp.
\end{equation}
In the two first terms in the r.h.s of eq.~(\ref{eq:dec-4}) the
function $\varphi(x)$ has ``pop out'' form the integral, so we
consider them elementary (in a sense we will discuss below). Again, the last
term in the r.h.s. can be seen to belong to $\CC_1^{1+}$ (since
actually, in this smooth setting, it belongs to $\CC_1^3$). Then
in analogy with eq.~(\ref{eq:int-rep1}) we can represent it in terms
of its image under $\delta$ as
\begin{equation}
  \label{eq:int-rep2}
\cj\lp  \dd \varphi'(x) dx  dx  \rp = -\Lambda \delta \cj\lp  \dd \varphi'(x)
\dd x  \dd x  \rp =- \Lambda [\cj(  \dd \varphi'(x)) \cj(
\dd x  \dd x  ) + \cj( \dd \varphi'(x)
\dd x) \cj(  \dd x )]
\end{equation}
where we acted with $\delta$ upon the triple iterated integral
according to Prop.~\ref{dissec}. Concerning the argument of $\Lambda$
in this last equation, we note the following two facts: $\cj(
\dd \varphi'(x)) = \der \varphi'(x)$ while the double iterated integral
$\cj(\dd \varphi'(x) dx)$  appears in the Taylor expansion for
$\der \varphi(x)$:
$$
\der \varphi(x) = \cj(\dd \varphi(x)) = \cj(\varphi'(x)\dd x) = \varphi'(x)
\cj(\dd x) + \cj(\dd \varphi'(x) \dd x)
$$
so
\begin{equation}
  \label{eq:taylor-1}
\cj(\dd\varphi'(x) \dd x) = \der \varphi(x) - \varphi'(x) \der x.  
\end{equation}
Then we can rewrite eq.~(\ref{eq:int-rep2}) as
\begin{equation}
  \label{eq:int-rep3}
\cj\lp  \dd \varphi'(x) \dd x  \dd x  \rp = -\Lambda [\der \varphi'(x) \cj(
\dd x  \dd x  ) + (\der \varphi(x) - \varphi'(x) \der x) \cj(  \dd x )]
\end{equation}
and finally we have obtained another expression for the integral
$\cj(\varphi(x) \dd x)$:
\begin{equation}
  \label{eq:integral-step2}
  \begin{split}
\cj(\varphi(x) \dd x) & = \varphi(x) \der x + \varphi'(x) \cj(\dd x \dd x) 
\\ & \qquad   -\Lambda [\der \varphi'(x) \cj(
\dd x  \dd x  ) + (\der \varphi(x) - \varphi'(x) \der x) \cj(  \dd x )]
\\ & = (1-\Lambda \der) [\varphi(x) \der x + \varphi'(x) \cj(\dd x \dd x)]
  \end{split}
\end{equation}
where to go from the first equation to the second we used the
algebraic relation
\begin{equation}
\label{eq:area-rel}
\der \cj(\dd x \dd x) =\cj(\dd x) \cj(\dd x).
\end{equation}
Up to this point all we got are another equivalent expression for the
classic integral in the r.h.s. of eq.~(\ref{eq:int-rep3}). 

It is a very remarkable basic result of rough path theory that the
r.h.s. of eq.~(\ref{eq:int-rep3}) makes sense for paths $x$ which are
very irregular like the sample paths of Brownian motion (which
a.s. are not H\"older continuous for any index greater than $1/2$), 
once we have
at our disposal also a 1-increment  $\cj(dx dx)$ which is sufficiently
small and satisfy eq.~(\ref{eq:area-rel}). Heuristically,
in this situation the formula says that the 1-increment $\varphi(x)
\der x$ can be ``corrected'' or ``renormalized'' by adding the
correction $\varphi'(x) \cj(dx dx)$ so that it becomes integrable (in the sense of Corollary~\ref{cor:integration}).

In the cases where this
correction belongs to $\CC_1^{1+}$ we have
$
(1-\Lambda \delta)[\varphi'(x) \cj(\dd x \dd x)] = 0
$
so eq.~(\ref{eq:int-rep3}) becomes again eq.~(\ref{eq:int-rep1}) and
we reobtain the Young integral.
 
It is worth noticing at that point that the  integral, as
\emph{defined} by eq.~(\ref{eq:int-rep3}), has now to
be understood as an integral over the (step-2) rough path $(x, \cj(\dd x\,
dx))$~\cite{[Gubinelli-2004]} and it coincide with the notion of  integral
over a rough path given by Lyons in~\cite{Lyons1998}. 

This algorithm has an obvious extension to higher orders if we
assume that a reasonable definition of the iterated integrals 
$\cj(\dd x\, \dd x\cdots \dd x)$ can be given. To proceed further however we
need the notion of \emph{geometric} rough path (for more details
on this notion see~\cite{Lyons1998}) which must be  exploited
crucially to show that some terms are small and belongs to the domain
of $\Lambda$.

Note that we have worked in the scalar setting (i.e. all the object we
considered are real-valued). Willing to add some notational burden it
is easy to see that this section has an equivalent formulation in the
vector case (when $x$ takes values on $\RR^n$ and $\varphi$ is a
(smooth) differential from on $\RR^n$). Indeed all the theory is
interesting and useful especially in the vector case.
This explain the fact that we
do not considered techniques like the Doss-Sussmann approach to define
one dimensional integrals since they are essentially limited to the
scalar setting (where every reasonable differential form $\varphi$ is
exact) and do not have a vectorial counterpart.

\section{The increment complex in two dimensions}
\label{sec:two-dim}

In this paper we are interested in particular two-dimensional integrals which can take two basic forms which in general are not equivalent.
If $f,g : \RR^2 \to \RR$ are regular enough we can define the two
dimensional integral of $f$ wrt. $g$ as
\begin{equation}
\label{eq:2d-integral}
  \iint_{(s_1,t_1)}^{(s_2,t_2)} f \dd g := \int_{s_1}^{s_2} \dd s
  \int_{t_1}^{t_2} \dd t f(s,t) \partial_1 \partial_2 g(s,t) 
\end{equation}
where $\partial_1$ and $\partial_2$ are the partial derivatives
wrt. the first and the second coordinate, respectively.
Another possible and nonequivalent basic integral in two-dimension is
given, for a triple of functions $f,g,h$, by
\begin{equation*}
  \iint_{(s_1,t_1)}^{(s_2,t_2)} f \dd_1 g \dd_2 h :=
 \iint_{(s_1,t_1)}^{(s_2,t_2)} f(s,t) \partial_1 g(s,t) \partial_2
 h(s,t)  \, \dd s\dd t
\end{equation*}

Then
\begin{equation}
\label{eq:2d-increment}
  \iint_{(s_1,t_1)}^{(s_2,t_2)} \dd g =: (\der g)(s_1,t_1,s_2,t_2)
\end{equation}
defines the coboundary map $\der$ for
functions of two parameters:
\begin{equation}
  \label{eq:2d-delta}
  \begin{split}
(\der g)(s_1,t_1,s_2,t_2) & =
g(s_2,t_2)-g(s_1,t_2)-g(s_2,t_1)+g(s_1,t_1) 
\\ &= [g(s_2,t_2)-g(s_1,t_2)]-[g(s_2,t_1)-g(s_1,t_1)]    
  \end{split}
\end{equation}
which is just the composition of two finite-difference operator in the
two directions. These integrals are to be considered as continuous functions of two points $(s_1,t_1)$ and $(s_2,t_2)$ on the plane which vanishes whenever $s_1=s_2$ or $t_1=t_2$. This preliminary observation leads to the following general
construction for a 2d cochain complex suitable for the analysis of these two-parameter integrals.

\bigskip
Fix a positive real $T$ and
let $\CCC_{k,l}(V)$ the space of continuous functions from $\ott^k \times
\ott^l \to V$, $V$ some vector space such that 
$$
g_{(s_1,\dots,s_k) (t_1,\dots,t_l)} = 0
$$ 
whenever $s_i = s_{i+1}$ or $t_i = t_{i+1}$. We will write $\CCC_{k,l}
= \CCC_{k,l}(\RR)$.

For $\CCC_{k}(V) =
\CCC_{k,k}(V)$ we will use the natural identification with the space
of continuous functions from $(\ott^2)^k \to V$. These will play the
rôle of 2d $k$-increments: they are functions of $k$ points in the
square $\ott^2$ such that they become zero whenever two contiguous
arguments have one coordinate in common. 

Note that $\CCC_{k,l} = \CC_k \otimes \CC_l$ and in
general
\begin{equation}
  \label{eq:isomorphism}
\CCC_{k,l}(V) = \CC_k \otimes \CC_l \otimes V  
\end{equation}
We will call the elements of $\CCC_{k,l}(V)$ $(k,l)$-biincrements and
the elements of $\CCC_k(V)$ $k$-biincrements. 
Moreover we introduce one-dimensional coboundaries $\der_1, \der_2$
which acts as described in Sec.~\ref{sec:one-dim} on the biincrements
view as functions of the first set, or of the second set of arguments,
i.e. they acts on the first or second $\CC_*$ factor according to
factorization of eq.~(\ref{eq:isomorphism}).
To be concrete
$$
\der_1 : \CCC_{k,l}(V) \to \CCC_{k+1,l}(V)
$$
$$
\der_2 : \CCC_{k,l}(V) \to \CCC_{k,l+1}(V)
$$
and for example, if $g \in \CCC_{k,l}(V)$ then
$$
(\der_1 g)_{(s_1,\cdots,s_{k+1}),(t_1,\cdots,t_l)} = \sum_{i=1}^{k+1} (-1)^{i+1}
g_{(s_1,\cdots,\hat s_i,\cdots,s_{k+1}),(t_1,\cdots,t_l)}
$$
where, as usual, the notation $\hat s_i$ means that the corresponding
argument is omitted. It is easy to see that $\der_1$ and $\der_2$
commute and that 
$$
\der = \der_1 \der_2 : \CCC_k(V) \to \CCC_{k+1}(V)
$$
is a coboundary, i.e. satisfy the equation $\der \der = 0$. Moreover,
if $g \in \CCC_{k,l}(V)$ we have
$$
(\der g)_{(s_1,\cdots,s_{k+1}),(t_1,\cdots,t_{l+1})}  =
\sum_{i=1}^{k+1} \sum_{j=1}^{l+1} (-1)^{i+j}
g_{(s_1,\cdots,\hat s_i,\cdots,s_{k+1}),(t_1,\cdots,\hat t_j,\cdots,t_{j+1})} 
$$

Then $(\CCC_*(V),\der)$ is a cochain complex. It will be important to
note that its cohomology is not trivial and that it will play a rôle
in our subsequent results.

\subsection{Cohomology of $(\CCC_*,\der)$}

The complex $(\CCC_*(V),\der)$ is the diagonal of the following
commutative diagram
\begin{equation}
  \label{eq:cd}
\begin{CD}
\CCC_{1,1}(V) @>\der_1>> \CCC_{2,1}(V) @>\der_1>> \CCC_{3,1}(V)
@>\der_1>> \cdots\\
@V\der_2VV @V\der_2VV @V\der_2VV \\
\CCC_{1,2}(V) @>\der_1>> \CCC_{2,2}(V) @>\der_1>> \CCC_{3,2}(V)
@>\der_1>> \cdots\\
@V\der_2VV @V\der_2VV @V\der_2VV \\
\CCC_{1,3}(V) @>\der_1>> \CCC_{2,3}(V) @>\der_1>> \CCC_{3,3}(V)
@>\der_1>> \cdots\\
@V\der_2VV @V\der_2VV @V\der_2VV \\
\end{CD}  
\end{equation}

We are mainly interested in the first cohomology group
$$
H_1(\CCC,\der) = \frac{\cZ \CCC_1(V)}{\cB \CCC_1(V)}
$$ 
where as before we denote $\cZ \CCC_k(V) = \text{Ker}
\der|_{\CCC_k(V)}$ and $\cB \CCC_k(V) = \text{Im}
\der|_{\CCC_{k-1}(V)}$, the spaces of \emph{$k$-bicocycles} and
\emph{$k-1$-bicoboundaries}, respectively.

To compute the cohomology consider applications $\sigma_1 : \CCC_{k,l}(V)
\to \CCC_{k-1,l}(V)$ (for $k \ge 1$) and $\sigma_2
: \CCC_{k,l}(V) \to \CCC_{k,l-1}(V)$ (for $l \ge 1$) which fix the first argument on
each direction to the (arbitrary) value $0$. For example:
$$
(\sigma_1 g)_{(s_1,\dots,s_{k-1})(t_1,\dots,t_l)} = g_{(0,s_1,\dots,s_{k-1})(t_1,\dots,t_l)}
$$ 
and a similar equation for $\sigma_2$. Then we have the homotopy formulas
$$
\sigma_i \der_i - \der_i \sigma_i = 1, \qquad \text{for i=1,2}
$$
which are at the origin of the exactness of the one-dimensional
complexes forming the rows and the columns of the diagram~(\ref{eq:cd}).
Let $k \ge 1$. Take $a \in \cZ \CCC_{k}(V)$ and let
$$
b = a - \sigma_1 \der_1 a - \sigma_2 \der_2 a
$$
and, using the homotopy formulas,  verify that
$$
\der_1 b = \der_1 a - \der_1 a - \der_1 \sigma_2 \der_2 a = - \sigma_2 \der a = 0
$$
since $\der_1$ commutes with $\sigma_2$. Similarly $\der_2 b = 0$. Then $b
\in \text{Ker} \der_1 \cap \text{Ker} \der_2$ which means that we can
write
$$
 b = \der_1 s_1 b = \der_1 \sigma_1 \der_2 \sigma_2 b 
$$
but since operators with different indexes commutes we can always
rewrite this as
$$
 b = \der_1  \der_2 \sigma_1 \sigma_2 b = \der \sigma_1 \sigma_2 b 
$$
so that $b \in \cB \CCC_k(V)$. Next, note that $\sigma_1 \der_1 a \in
\text{Ker} \der_2$, since
$
\der_2 \sigma_1 \der_1 a = \sigma_1 \der a = 0
$
so $\sigma_1 \der_1 a = \der_2 \sigma_2 \sigma_1 \der_1 a$. Then for any $a \in \cZ
\CCC_k(V)$ we have the decomposition
\begin{equation}
  \label{eq:decomposition-of-bicocycles}
a = \der q + \der_1 \sigma \der_2 a + \der_2 \sigma \der_1 a  
\end{equation}
for some $q \in \CCC_{k-1}(V)$ where we let $\sigma = \sigma_1 \sigma_2 = \sigma_2 \sigma_1$.

\subsection{Computations in $\CCC_{*,*}$}
For $a \in \CCC_{n,m}$, $b \in \CCC_{k,l}$ we can define
the (noncommutative, associative) product $ab \in \CCC_{n+k-1,m+l-1}$ as
\begin{equation*}
ab_{(s_1,\dots,s_{n+k-1})(t_1,\dots,t_{m+l-1})}
= 
a_{(s_1,\dots,s_{n})(t_1,\dots,t_{m})} b_{
(s_{n},\dots,s_{n+k-1})(t_{m},\dots,t_{l+m-1})}.
\end{equation*}
For example for $a \in \CCC_{2,1}$, $b \in \CCC_{1,2}$
we have
$
  (ab)_{(s_1,s_2)(t_1,t_2)} = a_{(s_1,s_2) t_1} b_{s_2 (t_1,t_2)}
$.
This definition is suited to work well with the action of
$\der_1,\der_2$, for example we have that, if $a,b \in \CCC_{1,*}$:
\begin{equation*}
  \der_1 (a \der_1 b) =  \der_1 a \der_1 b, \qquad \der_1 ( \der_1 b a)
  = - \der_1 b \der_1 a
\end{equation*}
If $a,b \in \CCC_{1}$:
\begin{equation*}
  \der_2 (a \der_1 b) = \der_2 a \der_1 b + a \der b,
\end{equation*}
\begin{equation*}
  \der(\der_1 a \der_2 b) = \der_1 ((\der_2 \der_1 a) \der_2 b) +
  \der_1 a (\der_2 \der_2 b)) = \der_1 (\der a \der_2 b) =  -\der a \der b
\end{equation*}
and
\begin{equation*}
  \der(\der_2 b \der_1 a) = - \der b \der a
,\qquad \der (a\der b) = \der (\der a b) = \der a \der b
\end{equation*}
as can be easily checked by a direct computation.
In the two parameters setting  we will improve our algebraic structure by adding a new type of product for $f\in\CCC_{m,n}$ and $g\in\CCC_{m,l}$ we define $f\circ_1g\in\CCC_{m,n+l}$ by
$$
f\circ_1g_{(s_1s_2...s_m)(t_1t_2...t_{n+l-1})}=f_{(s_1s_2...s_m)(t_1t_2...t_n)}g_{(s_1s_2...s_n)(t_n,t_{n+1}...t_{n+l-1})}
$$
and an analogue definition in the second direction .
\bigskip

For a two-dimensional quantity like the basic integral in
eq.(\ref{eq:2d-integral}) we can write down the  following relation
\begin{equation}
\label{eq:int-decomp}
  \begin{split}
  \iint_{(s_1,t_1)}^{(s_2,t_2)} & \left[\iint_{(s_1,t_1)}^{(s,t)}
    \dd f(u,v)\right] \dd g(s,t) 
\\ & =
      \iint_{(s_1,t_1)}^{(s_2,t_2)} [f(s,t)
      -f(s_1,t)-f(s,t_1)+f(s_1,t_1)]\dd g(s,t)
\\ & = \iint_{(s_1,t_1)}^{(s_2,t_2)} f(s,t)\dd g(s,t)
 + f(s_1,t_1) (\der g)(s_1,t_1,s_2,t_2)
\\ & -
\int_{t_1}^{t_2} f(s_1,t) \dd_2 [g(s_2,t)-g(s_1,t)]
\\ & -
\int_{s_1}^{s_2} f(s,t_1) \dd_1 [g(s,t_2)-g(s,t_1)]
  \end{split}
\end{equation}
which will play the same rôle as eq.~\eqref{eq:dec-1} in the
one-dimensional setting.

\bigskip
As an example of the formalism set up up to now we can consider the
decomposition of eq.~\eqref{eq:int-decomp} for the (two-dimensional) iterated integral
$\iint df dg$ of two smooth elements $f,g \in \CCC_1$. In
compact notation it reads:
\begin{equation*}
\iint f \dd g = - f \der g   + \int f \dd_1 \der_2 g + \int f \dd_2
 \der_1 g  + \iint \dd f \dd g 
\end{equation*}
where we understand the integrals as functions of both the extremes of integration.
Note that 
$$
\der \iint \dd f \dd g = \der (f \der g) = \der f \der g
$$ 
so
\begin{equation}
  \label{eq:an-decomp}
\iint f \dd g - \int f \dd_1 \der_2 g - \int f \dd_2
 \der_1 g \in \ker \der.  
\end{equation}
Actually the three factors in eq.~(\ref{eq:an-decomp}) corresponds
exactly to the decomposition~\eqref{eq:decomposition-of-bicocycles}
 since
\begin{equation*}
  \iint f \dd g = \der \iint_{*,*} f \dd g, \qquad
\int f \dd_1 \der_2 g = \der_1 \int_{*} f \dd_1 \der_2 g
,
\qquad
\int f \dd_2 \der_1 g = \der_2 \int_{*} f \dd_2 \der_1 g
\end{equation*}
where the star in the integral sign denote that the lower integration
point has been fixed arbitrarily (e.g. to $0 \in \ott$).

Another relevant remark is to note that the antisymmetric element $\omega^a = \der_1 f
\der_2 g - \der_2 f \der_1 g$ satisfy $\der \omega^a = 0$. Indeed
\begin{equation*}
\der_1 \omega^a =   - \der_1 f
\der g - \der f \der_1 g
,\qquad
\der \omega^a = -\der f \der g + \der f \der g = 0
\end{equation*}
according to the rules established above. For the symmetric
counterpart $\omega^s = \der_1 f
\der_2 g + \der_2 f \der_1 g$ we have
\begin{equation*}
\der_1 \omega^s =   - \der_1 f
\der g + \der f \der_1 g
,\qquad
\der \omega^s = -\der f \der g - \der f \der g = - 2 \der f \der g  
\end{equation*}

\subsection{Splitting and other operations}

Each of the vector spaces $\CCC_{k,m}(V)$ is naturally isomorphic  to
either $\CC_k(\CC_m(V))$ or to $\CC_m(\CC_k(V))$: consider each
$k,m$-biincrement either as a $k$-increment in the first direction
with values in $m$-increments in the other direction or vice-versa. The
multiplication in $\CCC_{*,*}$ is compatible with these isomorphism.

In what follows it will be useful to introduce a one-dimensional
\emph{splitting} map $S$  which sends
products $ab \in \CC_2(V)$ for $a \in \CC_1(V)$ and $b \in \CC_1(V)$ to the elementary
tensor $S(ab) = a\otimes b \in \CC_1(V) \stackrel{V}{\otimes}
\CC_1(V)$ where the tensor product is over the algebra $V$. 
The map is the extended by linearity to
the subspace $\mathscr{M}_2$ of $\CC_2$ generated by the linear
combinations of products of two elements of $\CC_1$.
Elements
of $\CC_1(V) \stackrel{V}{\otimes}
\CC_1(V)$ are just functions $(t,u,v,s) \mapsto c_{tuvs}$ of four
arguments which are 1-increments in the couple $(t,u)$ and in the
couple $(v,s)$ but which may be non-zero for $u=v$. The multiplication
map $\mu: \CC_1(V) \stackrel{V}{\otimes}
\CC_1(V) \to \CC_2(V)$ just sends each $c$ to the 2-increment $(t,u,s)
\mapsto a_{tuus}$ and is the inverse of $S$: $\mu \circ S (a) = a$ for
any $a \in \mathscr{M}_2$.
 
 We will denote
$S_1: \CCC_{2,k}(V) \to \CC_{1}(\CC_k(V))
\stackrel{\CC_k(V)}{\otimes_1} \CC_{1}(\CC_k(V))$ and $S_2:
\CCC_{k,2}(V) \to \CC_{1}(\CC_k(V)) \stackrel{\CC_k(V)}{\otimes_2}
\CC_{1}(\CC_k(V))$
the splitting maps according to the first or the second
direction. These are understood according to the above isomorphism
$\CC_{k,m}(V)  \simeq \CC_{k}(\CC_m(V))  \simeq \CC_{m}(\CC_k(V))$ and
the index $1,2$ on the tensor product remember in which of the two
directions the splitting has taken place.

\subsection{Abstract integration in $\CCC_*$}

From now on we assume that $V$ is a Banach space with norm
$|\cdot|$. When they appears tensor product will be understood
according to the projective topology.

Let us introduce the following norms, for any $g \in \CCC_2(V)$
\begin{equation}\label{eq:Holder space}
 \|g\|_{z_1,z_2} \coloneq \sup_{s,t \in \ott^2} \frac{|g_{(s_1,t_1)(s_2,t_2)}|}{|s_1-t_1|^{z_1} |s_2-t_2|^{z_2}} 
\end{equation}
 and for $h \in \CCC_3(V)$
\begin{equation*}
 \|h\|_{\gamma_1,\gamma_2,\rho_1,\rho_2} \coloneq \sup_{s,u,t \in \ott^2}
 \frac{|h_{(s_1,u_1,t_1)(s_2,u_2,t_2)}|}{\prod_{i=1,2} |s_i-u_i|^{\gamma_i} |u_i-t_i|^{\rho_i}} 
\end{equation*}
and
\begin{equation*}
\|h\|_{z_1,z_2} \coloneq \inf \left\{\sum_i
\|h_i\|_{\gamma_{1,i},\gamma_{2,i},z_1-\gamma_{1,i},z_2-\gamma_{2,i}}
\;\Bigg|\; h = \sum_i h_i, \; \gamma_{j,i} \in (0,z_i), \; j=1,2  \right\}  
\end{equation*}
and the corresponding subspaces $\CCC^{z_1,z_2}_2(V)$,
$\CCC^{\gamma_1,\gamma_2,\rho_1,\rho_2}_3(V)$ and
$\CCC^{z_1,z_2}_3(V)$. Moreover we say that $f\in\CCC_1^{\rho_1,\rho_2}$ if
\begin{equation}\label{eq:norm-def}
\mathscr N_{\rho_1,\rho_2}(f)=||\der f||_{\rho_1,\rho_2}+||\der_1f||_{\rho_1,0}+||\der_2f||_{0,\rho_2}+||f||_{\infty}<+\infty
\end{equation}
with 
$$
||\der_1f||_{\rho_1,0}=\sup_{(s_1,s_2,t_1)\in[0,T]^3}\frac{|\der_1f_{s_1s_2t_1}|}{|s_2-s_1|^{\rho_1}}
$$
and similar definition in the second direction.
The main feature of the space $\CCC^{z_1,z_2}_2$ is that
$\CCC^{z_1,z_2}_2  \cap \ker \der_1 = \{0 \}$ if $z_1 > 1$ and
$\CCC^{z_1,z_2}_2  \cap \ker \der_2 = \{0 \}$ if $z_2 > 1$.  This
implies that the equation $\der a = 0$ has only a trivial solution $a
= 0$ if we require $a \in \CCC^{z_1,z_2}_1$ with $z_1,z_2 > 1$.

Let $\CCC^{1+}_i(V) = \cup_{z_1>1,z_2>1} \CCC^{z_1,z_2}_i(V)$, $i=1,2$.

Note that we have the isomorphism $\CCC_{a,b}^{z_1,z_2}(V) \simeq
\CC_a^{z_1}(\CC_b^{z_2}(V))  \simeq \CC_b^{z_2}(\CC_a^{z_1}(V))$ for
$a,b = 0,1,2$ and $z_1,z_2 \ge 0$.

Before stating the main result of this section we introduce two
versions of the one-dimensional $\Lambda$ map of
Prop.~\ref{prop:Lambda}, acting on the two different coordinates.  

\begin{lemma}
$\label{lem2}$  $\Lambda_1 : \cB_2 \CCC^{w_1,z_2}_{3,a}(V) \to
\CCC^{w_1,z_2}_{2,a}(V)$ for $a=1,2,3$ with $w_1>1$ such that $\der_1 \Lambda_1 = 1$ and
\begin{equation}
  \label{eq:bound-lambda1}
\|\Lambda_1 h\|_{z_1,w_2} \le C_{z_1} \|h\|_{z_1,w_2}  
\end{equation}
and an analogous bound for $\Lambda_2$.
\end{lemma}
\begin{proof}
If we fix $s_2,u_2,t_2 \in \ott$ we can consider $h^{s_2 u_2 t_2} \in \CC_2(V)$ such that
$$
(s_1,u_1,t_1) \mapsto h^{s_2,u_2,t_2}_{s_1,u_1,t_1} = h_{(s_1,u_1,t_1)}^{(s_2,u_2,t_2)}
$$
and note that $h^{s_2u_2t_2} \in \cB \CC^{z_1}_2(V)$ since $\der h^{s_1 u_1
  t_1} = 0$ so that it is in the
range of the one-dimensional $\Lambda$ of Prop.~\ref{prop:Lambda}  and
\begin{equation}
\label{eq:bound-lambda1-proof}
  |(\Lambda h^{s_2,u_2,t_2})_{s_1 t_1}| \le C_{z_1}
   \|h^{s_2,u_2,t_2}\|_{z_1} |s_1-t_1|^{z_1}.
\end{equation}
Then define $\Lambda_1$ as
$$
(\Lambda_1 h)_{(t_1,s_1),(t_2,u_2,s_2)} \coloneq (\Lambda h^{s_2,u_2,t_2})_{s_1 t_1} 
$$
and note that the bound~(\ref{eq:bound-lambda1-proof}) implies
eq.~(\ref{eq:bound-lambda1}).
Proceeding similarly one can prove a similar statement about
$\Lambda_2$.
\end{proof}

Then it holds the analogous of the one-dimensional result:
\begin{proposition}
There exists a unique map $\Lambda : \cB \CCC^{1+}_2(V) \to
\CCC^{1+}_1(V)$ such that $\der \Lambda = 1$. Moreover
if $z_1,z_2 > 1$, $h \in \cB \CCC_2^{z_1,z_2}(V)$
then $$
\|\Lambda h\|_{z_1,z_2} \le \frac{1}{2^{z_1}-2} \frac{1}{2^{z_2}-2} \|h\|_{z_1,z_2}.
$$  
\end{proposition}
\begin{proof} 
Since $h \in \text{Img} \der$ we have $\der_1 h = \der_2 h = 0$. Then
let $\Lambda h = \Lambda_1 \Lambda_2 h$ which is meaningful since the
$\der_1 \Lambda_2 h=\Lambda_2 \der_1 h=0$ (by linearity) and the
requirement on the regularity is satisfied.
Then
$$
\der \Lambda h = \der_2 \der_1 \Lambda_1 \Lambda_2 h = \der_2
\Lambda_2 h = h
$$
and
$$
\|\Lambda h\|_{z_1,z_2} = \|\Lambda_1 \Lambda_2 h\|_{z_1,z_2} \le
C_{z_1} \|\Lambda_2 h\|_{z_1,z_2} \le C_{z_1} C_{z_2} \|h\|_{z_1,z_2}.
$$
Uniqueness depends on the fact that $z_1,z_2 > 1$. Using the
uniqueness it is easy to deduce that $\Lambda = \Lambda_2 \Lambda_1$,
i.e. the one-dimensional maps commute (when they can be both applied).
\end{proof}

We can already state an interesting result about integration of
``small'' biincrements.

\begin{corollary}[2d integration]
\label{cor:2d-integration}
Let $a \in \CCC_{2,2}(V)$ such that $\der_1 a \in \CCC_{3,2}^{z_1,*}$,
$\der_2 a \in \CCC_{2,3}^{*,z_2}$, $\der a \in \CCC_{2,2}^{z_1,z_2}$
with $z_1,z_1 > 1$. There exists $f \in \CCC_{1,1}(V)$ such that 
$$
\der f = (1-\Lambda_1 \der_1) (1-\Lambda_2 \der_2) a
$$  
and
$$
\lim_{|\Pi|\to 0} \sum_{i,j}
a_{(t^1_{i+1},t^1_{i}){(t^2_{j+1},t^2_{j})}} = (\der f)_{(t^1,s^1),(t^2,s^2)}
$$
where the limit is taken over partitions $\Pi = \{ (t^1_i, t^2_j)_{i,j} \}$  of the square
$[t^1,t^2]\times [s^1,s^2]$ into boxes whose maximum size $|\Pi|$ goes
to zero.
\end{corollary}
\begin{proof}
The required conditions on $a$ ensure that the 1-biincrement 
$$
h = (1-\Lambda_1 \der_1) (1-\Lambda_2 \der_2) a = a -\Lambda_1 \der_1
a - \Lambda_2 \der_2 a + \Lambda \der a
$$  
is well defined. By direct computation we have that
$$
\der_1 h = \der_1 (1-\Lambda_1 \der_1) (1-\Lambda_2 \der_2) a =
(\der_1 - \der_1) (1-\Lambda_2 \der_2) a = 0
$$
and $\der_2 h = 0$. So $h$ must be in the image of $\der$, i.e. there
exists $f$ such that $h = \der f$. This proves the first claim.

To prove the convergence of the sums consider the above
decomposition
$$
a = \der f + \Lambda_2 \der_2 a + \Lambda_1 \der_1 a -\Lambda \der a
$$
written as
$
a = \der f + r_1 + r_2 + r
$
where
$$
r_1 = (1-\Lambda_1\der_1)\Lambda_2 \der_2 a,
\qquad r_2 =  
(1-\Lambda_2\der_2) \Lambda_1 \der_1 a 
\qquad r =  \Lambda \der a.
$$
Note that $r_1 \in \CCC_{2,2}^{*,z_2}$, $\der_1 r_1 = 0$ and $r_2 \in
\CCC_{2,2}^{z_1,*}$ and $\der_2 r_2 = 0$. 

Then let
\begin{equation*}
  \begin{split}
S_{\Pi} & = \sum_{i,j} a_{(t^1_{i+1},t^1_{i}){(t^2_{j+1},t^2_{j})}}  
\\
 & = \sum_{i,j} (\der f)_{(t^1_{i+1},t^1_{i}){(t^2_{j+1},t^2_{j})}} 
+ \sum_{i,j} (r_1)_{(t^1_{i+1},t^1_{i}){(t^2_{j+1},t^2_{j})}} 
\\ & \qquad + \sum_{i,j} (r_2)_{(t^1_{i+1},t^1_{i}){(t^2_{j+1},t^2_{j})}} 
+ \sum_{i,j} (r)_{(t^1_{i+1},t^1_{i}){(t^2_{j+1},t^2_{j})}}  
  \end{split}
\end{equation*}
and note that, using the fact that $\der f$ is an exact biincrement
$$
\sum_{i,j} (\der f)_{(t^1_{i+1},t^1_{i}){(t^2_{j+1},t^2_{j})}} = 
 (\der f)_{(t^1,s^1){(t^2,s^2)}}
$$
and
$$
 \sum_{i,j} (r_1)_{(t^1_{i+1},t^1_{i}){(t^2_{j+1},t^2_{j})}} = \sum_{j} (r_1)_{(t^1,s^1){(t^2_{j+1},t^2_{j})}} 
$$
since $\der_1 r_1 = 0$ (i.e. $r_1$ is an exact increment in the
direction $1$).
In the same way
$$
\sum_{i,j} (r_2)_{(t^1_{i+1},t^1_{i}){(t^2_{j+1},t^2_{j})}}  = 
\sum_{i} (r_2)_{(t^1_{i+1},t^1_{i}){(t^2,s^2)}} 
$$
Then
$$
 \left|\sum_{j} (r_1)_{(t^1,s^1){(t^2_{j+1},t^2_{j})}}\right| \le
 \|r_1\| \sum_j |t^2_{j+1}-t^2_j|^{z_2} \to 0
$$
and
$$
\left|\sum_{i} (r_2)_{(t^1_{i+1},t^1_{i}){(t^2,s^2)}}\right| \le
\|r_2\| \sum_i |t^1_{i+1}-t^1_i|^{z_1} \to 0
$$
and finally
$$
\left|  \sum_{i,j} (r)_{(t^1_{i+1},t^1_{i}){(t^2_{j+1},t^2_{j})}}  \right|
\le \|r\|  \sum_{i,j} |t^1_{i+1}-t^1_i|^{z_1} |t^2_{j+1}-t^2_j|^{z_2} \to 0
$$
as $|\Pi| \to 0$ which proves our claim.
\end{proof}

\section{Two-dimensional Young theory}
\label{sec:two-d-young}
\begin{proposition}
\label{prop:young2d}
Let $f \in \CCC^{\gamma_1,\gamma_2}_{1,1}$, $g \in \CCC^{\rho_1,\rho_2}_{1,1}$
with $\gamma_1 + \rho_1 = z_1 > 1$, $\gamma_2 + \rho_2 = z_2 >
1$. Then $\der f \der g \in \CCC^{z_1,z_2}_{2} \cap \img \der$ and the
integral $\iint f \dd g$ can be defined as
\begin{equation}
  \label{eq:Young2d}
  \begin{split}
  \iint f \dd g  & = (1-\Lambda_1 \der_1)(1-\Lambda_2 \der_2) (f \der
  g)
\\
& = - f \der g + \Lambda (\der f \der g) + \int f \dd_1 \der_2
  g +  \int f \dd_2 \der_1 g     
  \end{split}
\end{equation}
where $\int f \dd_1 \der_2
  g ,  \int f \dd_2 \der_1 g$ are standard Young integrals. 
Moreover we can define also the integral $\iint \dd_1 f \dd_2 g$ as
\begin{equation}
  \label{eq:young2d-mixed}
  \iint \dd_1 f \dd_2 g = (1-\Lambda_1 \der_1) (1-\Lambda_2 \der_2) (\der_1 f
  \der_2 g) .
\end{equation}
and for $\rho_1,\rho_2>1/2$ we can define 
\begin{equation}
\iint f\dd_1g\dd_2g= (1-\Lambda_1 \der_1) (1-\Lambda_2 \der_2)(f\der_1g\der_2g) 
\end{equation}
\end{proposition}
\begin{corollary}
\label{cor:young-sums}
Under the hypothesis of the previous proposition, the two-dimensional
sums of the increments $f \der g$ converge:
  \begin{equation}
    \label{eq:young-sums}
\lim_{|\Pi_{z,w}|\to 0} \sum_{i,j} (f\der g)_{x_i
  x_{i+1};y_{j} y_{j+1}} = \left(\iint f \dd g\right)_{z,w}.    
  \end{equation}
where the partition $\Pi_{z,w}$ is taken on the square $(z,w)$, $z_i \in
\RR^2$.
\end{corollary}
\begin{proposition}
Under the assumption of the Proposition \eqref{prop:young2d} we have that:
\begin{equation}
\begin{split}
\iint f\dd_1g\dd_2g&= (1-\Lambda_1 \der_1) (1-\Lambda_2 \der_2)(f\der_2g\der_1g)\\&= (1-\Lambda_1 \der_1) (1-\Lambda_2 \der_2)(f\der_1g\bullet\der_2g)
\end{split}
\end{equation}
where $(\der_1g\bullet\der_2g)_{s_1s_2t_1t_2}=\der_1g_{s_1s_2t_1}\der_2g_{s_1t_1t_2}$
\end{proposition}
\begin{proof}
Let $a=f\der_2g\bullet\der_1g$. By a simple computation we have that $\der_1a\in\CCC_{3,2}^{\gamma_1+\rho_1,*}$ , $\der_2a\in\CCC_{2,3}^{*,\gamma_2,\rho_2}$  and $\der a\in\CCC_{3,3}^{\gamma_1+\rho_1,\gamma_2+\rho_2}$ then the Corollary~\eqref{cor:2d-integration} gives 
\begin{equation}
  \begin{split}
  (1-\Lambda_1 \der_1) (1-\Lambda_2 \der_2)(f\der_2g\bullet\der_1g)_{s_1s_2t_1t_2}&=\lim_{|\Pi|\to0}\sum_{\Pi}f_{s_it_j}\der_2g_{s_it_jt_{j+1}}\der_1g_{s_is_{i+1}t_j} 
\\&= \lim_{|\Pi|\to0}\sum_{\Pi}f_{s_it_j}\der_1g_{s_is_{i+1}t_j}\der_2g_{s_{i+1}t_jt_{j+1}}-f_{s_it_j}\der_1g_{s_is_{i+1}t_j}\der g_{s_is_{i+1}t_jt_{j+1}}
 \end{split}
\end{equation}
where $\Pi :=\{(s_i,t_j)\}_{i,j}$ is a partition of $[s_1,s_2]\times[t_1,t_2]$ and we denote also by $\Pi_1=(s_i)_i$ and $\Pi_2=(t_j)_j$ a respectively partition of $[s_1,s_2]$ and $[t_1,t_2]$. Once again by the Corollary ~\eqref{cor:2d-integration} we have that 
$$
\lim_{|\Pi|\to0}\sum_{\Pi}f_{s_it_j}\der_1g_{s_is_{i+1}t_j}\der_2g_{s_{i+1}t_jt_{j+1}}=(\iint f\dd_1g\dd_2g)_{s_1s_2t_1t_2}
$$
and then the one dimensional Young theory of integration gives
$$
\lim_{|\Pi_2|\to0} \sum_{\Pi_2}f_{s_it_j}\der_1g_{s_is_{i+1}t_j}\der g_{s_is_{i+1}t_jt_{j+1}}=\int_{t_1}^{t_2}f_{s_it}\der_1g_{s_is_{i+1}t}\dd_t\der_1g_{s_is_{i+1}t}
$$
where $\left|\int_{t_1}^{t_2}f_{s_it}\der_1g_{s_is_{i+1}t}\dd_t\der_1g_{s_is_{i+1}t}\right| \lesssim  (s_{i+1}-s_{i})^{2\rho_1}(t_2-t_1)^{\rho_2}$. 
Finally using the fact that $\rho_1,\rho_2>1/2$ gives:
$$
\lim_{|\Pi_{1}|\to0}\lim_{|\Pi_2|\to0}\sum_{|\Pi|}f_{s_it_j}\der_1g_{s_is_{i+1}t_j}\der g_{s_is_{i+1}t_jt_{j+1}}=0
$$
Putting all these last equation together give us the second line of our proposition, the first part is given by the same argument. 
\end{proof}
\begin{proposition}\label{proposition:Stratonovich-Young-formula}
Let $x\in\CCC^{\alpha,\beta}_{1,1}$ and $\varphi\in C^{4}(\mathbb R)$ then for $\alpha,\beta>1/2$ the following  change of variable formula :
$$
\der\varphi(x)=\iint\varphi'(x)\dd x+\iint\varphi''(x)\dd_1x\dd_2x
$$
hold.
\end{proposition}

\begin{proof}
Let $\Pi_1=(s_i)_i$  a partition of $[s_1,s_2]$ and  $\Pi_2=(t_j)_j$ of $[t_1,t_2]$  then
\begin{equation*}
 \begin{split}
\der\varphi(x)_{s_1s_2t_1t_2}&=\sum_{\Pi_1\Pi_2}\der\varphi(x)_{s_is_{i+1}t_jt_{j+1}}\\&
=\sum_{\Pi_1\Pi_2}\left(\der_2\int_1\varphi'(x)\dd_1x\right)_{s_is_{i+1}t_jt_{j+1}}\\&
=\sum_{\Pi_1\Pi_2}\int_{s_i}^{s_{i+1}}\der_2\varphi'(x)_{st_jt_{j+1}}\dd_sx_{st_{j+1}}+\int_{s_i}^{s_{i+1}}\varphi'(x_{st_j})\dd_s\der_2 x_{st_jt_{j+1}}\\&
=\sum_{\Pi_1\Pi_2}a_{ij}+b_{ij}
\end{split}
 \end{equation*}
where we have used the one dimensional change of variable formula for the Young integral.
Now let us treat the first term of our sum
 \begin{equation*}
  \begin{split}
a_{ij}&=\int_{s_i}^{s_{i+1}}\der_2\varphi'(x)_{st_jt_{j+1}}\dd_sx_{st_{j+1}}\\&
=\der_2\varphi'(x)_{s_it_jt_{j+1}}\der_1x_{s_is_{i+1}t_{j+1}}+\Lambda_1(\der\varphi'(x)\der_1x)_{s_is_{i+1}t_{j}t_{j+1}}
\\&=\left(\int_{t_j}^{t_{j+1}}\varphi''(x_{s_{i}t})\dd_tx_{{s_i}t}\right)\der_1x_{s_is_{i+1}t_{j+1}}\\&+\Lambda_2\der_2(1-\Lambda_1\der_1)(\der_2\varphi'(x)\der_1x)_{s_is_{i+1}t_jt_{j+1}}
-\Lambda\der(\der_2\varphi'(x)\der_1x)_{s_is_{i+1}t_jt_{j+1}}
  \end{split}
 \end{equation*}
Is not difficult to see that 
\begin{equation}\label{eq:young-form-1}
\lim_{|\Pi|\to0}\sum_{\Pi_1\Pi_2}\Lambda_2\der_2(1-\Lambda_1\der_1)(\der_2\varphi'(x)\der_1x)_{s_is_{i+1}t_jt_{j+1}}=0
\end{equation}
and
\begin{equation}\label{eq:young-form-2}
\lim_{|\Pi|\to0}\sum_{\Pi_1\Pi_2}\Lambda\der(\der_2\varphi'(x)\der_1x)_{s_is_{i+1}t_jt_{j+1}}=0
\end{equation}
then is remind to treat the term $\left(\int_{t_j}^{t_{j+1}}\varphi''(x_{s_{i}t})\dd_tx_{{s_i}t}\right)\der_1x_{s_is_{i+1}t_{j+1}}$ and for that we have the following expansion :
\begin{equation}\label{eq:young-form-3}
\left(\int_{t_j}^{t_{j+1}}\varphi''(x_{s_{i}t})\dd_tx_{{s_i}t}\right)\der_1x_{s_is_{i+1}t_{j+1}}=\varphi''(x_{s_it_j})\der_2x_{s_it_jt_{j+1}}\der_1x_{s_is_{i+1}t_j}+r^{\flat}_{s_it_jt_{j+1}}\der_1x_{s_is_{i+1}t_j}
\end{equation}
where $r^{\flat}=\der_2\varphi'(x)-\varphi''(x)\der_2x\in\CCC_{1,2}^{*,2\beta}$ such that $\der_1r^{\flat}\in\CCC_{2,2}^{\alpha,2\beta}$ which gives : 
\begin{equation*}
\lim_{|\Pi_1|\to0}\sum_{\Pi_1}r^{\flat}_{s_it_jt_{j+1}}\der_1x_{s_is_{i+1}t_j}=\int_{s_1}^{s_2}r^{\flat}_{st_jt_{j+1}}\dd_sx_{st_{j+1}}
\end{equation*}
where the limit satisfy the upper bound $|\int_{s_1}^{s_2}r^{\flat}_{st_jt_{j+1}}\dd_sx_{st_{j+1}}|\lesssim_{\alpha,\beta}(s_2-s_1)^{\alpha}(t_2-t_1)^{2\beta}$
these allow us to conclude that :
\begin{equation}\label{eq:young-form-4}
\lim_{|\Pi_2|\to0}\lim_{|\Pi_1|\to0}\sum_{\Pi_1\Pi_2}r^{\flat}_{s_it_jt_{j+1}}\der_1x_{s_is_{i+1}t_j}=0.
\end{equation}
Then putting together equation \eqref{eq:young-form-1} , \eqref{eq:young-form-2} , \eqref{eq:young-form-3} and \eqref{eq:young-form-4} we obtain that :
$$
\left(\iint\varphi''(x)\dd_1x\dd_2x\right)_{s_1s_2t_1t_2}=\lim_{|\Pi_2|\to0}\lim_{|\Pi_1|\to0}\sum_{\Pi_1\Pi_2}a_{ij}
$$
the second sum $\sum_{\Pi_1\Pi_2}b_{ij}$ is more simple to compute indeed :
\begin{equation*}
\begin{split}
\sum_{\Pi_1\Pi_2}b_{ij}&=\sum_{\Pi_1\Pi_2}\int_{s_i}^{s_{i+1}}\varphi'(x_{st_j})\dd_s\der_2 x_{st_jt_{j+1}}\\&
=\sum_{\Pi_1\Pi_2}\varphi'(x_{s_it_j})\der x_{s_is_{i+1}t_jt_{j+1}}+\Lambda_2\der_2(1-\Lambda_1\der1)(\varphi'(x)\der x)-\Lambda\der(\varphi'(x)\der x)
\end{split}
\end{equation*}
and then 
\begin{equation*}
\begin{split}
\lim_{|\Pi|\to0} \sum_{\Pi_1\Pi_2}b_{ij}&=\lim_{|\Pi|\to0} \sum_{\Pi_1\Pi_2}\varphi'(x_{s_it_j})\der x_{s_is_{i+1}t_jt_{j+1}}\\&
=\left(\iint\varphi'(x)\dd x\right)_{s_1s_2t_1t_2}
\end{split}
\end{equation*}
and this gives the needed result.
\end{proof}
We have also this immediate generalization for the Young integral.
\begin{proposition}
\label{mixed-Young-2D}
Let $y\in\CCC_{1,1}^{\gamma_1,\gamma_2}$,$x\in\CCC_{1,1}^{\rho_1,\rho_2}$ and $z\in\CCC_{1,1}^{\beta_1,\beta_2}$ such that $\gamma_i+\rho_i>1$,$\beta_i+\gamma_i>1$ and $\beta_i+\rho_i>1$ for $i=1,2$ then we can define
$$
\iint y\dd_1x\dd_2z=(1-\Lambda_1\der_1)(1-\Lambda_2\der_2)(y\der_1x\der_2z)
$$
$$
\iint y\dd_2z\dd_1x=(1-\Lambda_1\der_1)(1-\Lambda_2\der_2)(y\der_2z\der_1x)
$$
and
$$
\iint y\dd_1x\bullet \dd_2z=(1-\Lambda_1\der_1)(1-\Lambda_2\der_2)(y\der_1x\bullet\der_2z)
$$
and we have the following identity 
$$
\iint y\dd_1x\dd_2z=\iint y\dd_2z\dd_1x=\iint y\dd_1x\bullet \dd_2z
$$
Moreover we have the following Riemann sums representation for our integrals
\begin{equation*}
\begin{split}
\int_{s_1}^{s_2}\int_{t_1}^{t_2}y_{st}\dd_sx_{st}\dd_tz_{st}&=\lim_{|\pi|\to0}\sum_{i,j}y_{s_it_j}\der_1x_{s_is_{i+1}t_j}\der_2z_{s_{i+1}t_jt_{j+1}}
\\&=\lim_{|\pi|\to0}\sum_{i,j}y_{s_it_j}\der_1x_{s_is_{i+1}t_j}\der_2z_{s_it_jt_{j+1}}
\\&=\lim_{|\pi|\to0}\sum_{i,j}y_{s_it_j}\der_2z_{s_it_jt_{j+1}}\der_1x_{s_is_{i+1}t_{j+1}}
\end{split}
\end{equation*}
\end{proposition}
\begin{proof}
Let $a^1=y\der_1x\der_2z$, $a^2=y\der_2z\der_1x$ and $a^3=y\der_1x\bullet\der_2z$ then by a simple computation we have that 
$$
\der_1a^1=-\der_1y\der_1x\der_2z-y\der_1x\der z\in\CCC_{3,2}^{\min(\gamma_1+\rho_1,\rho_1+\beta_1),\beta_2}
$$
$$
\der_2a^1=-\der_2y\der_1x\der_2z-y\der x\der_2z\in\CCC_{2,3}^{\rho_1,\min(\gamma_2+\beta_2,\rho_2+\beta_2)}
$$
and
$$
\der a^1=\der y\der_1x\der_2z+\der_1y\der x\der_2z+\der_2y\der_1x\der z+y\der x\der z\in\CCC_{3,3}^{\min(\gamma_1+\rho_1,\rho_1+\beta_1),\min(\gamma_2+\beta_2,\rho_2+\beta_2)}
$$
 then $a^1$ satisfies the assumption of the Corollary~\eqref{cor:2d-integration} which also true for $a^2$ and $a^3$ by a similar computation. Then the integral $\iint y\dd_1x\dd_2z$, $\iint y\dd_2z\dd_1x$ and $\iint y\dd_1x\bullet\dd_2z$ are well defined and we have that
$$
(\iint y\dd_2z\dd_1x)_{s_1s_2t_1t_2}=\lim_{|\pi|\to0}\sum_{i,j}y_{s_it_j}\der_1x_{s_is_{i+1}t_j}\der_2z_{s_{i+1}t_jt_{j+1}}
$$
$$
(\iint y\dd_2z\dd_1x)_{s_1s_2t_1t_2}=\lim_{|\pi|\to0}\sum_{i,j}y_{s_it_j}\der_2z_{s_it_jt_{j+1}}\der_1x_{s_is_{i+1}t_{j+1}}
$$
and
$$
(\iint y\dd_1x\bullet \dd_2z)_{s_1s_2t_1t_2}=\lim_{|\pi|\to0}\sum_{i,j}y_{s_it_j}\der_1x_{s_is_{i+1}t_j}\der_2z_{s_it_jt_{j+1}}
$$
To prove that this last integral coincide it suffices to show that the difference between the Riemann sum vanish when the mesh of the partition go to zero. In fact we have that 
\begin{equation}
\begin{split}
\lim_{|\pi_2|\to0}\sum_{i,j}(a^1-a^3)_{s_is_{i+1}t_jt_{j+1}}&=\lim_{|\pi_2|\to0}\sum_{i,j}y_{s_it_j}\der_1x_{s_is_{i+1}t_j}\der z_{s_is_{i+1}t_jt_{j+1}}
\\&=\sum_i\int_{t_1}^{t_2}y_{s_it}\der_1x_{s_is_{i+1}t}\dd_t\der_1z_{s_is_{i+1}t}
\end{split}
\end{equation}
Using the fact that $|\int_{t_1}^{t_2}y_{s_it}\der_1x_{s_is_{i+1}t}\dd_tz_{s_is_{i+1}t}|\lesssim(s_{i+1}-s_i)^{\rho_1+\beta_1}$ we obtain that : 
$$
\lim_{|\pi_1|\to0}\lim_{|\pi_2|\to0}\sum_{i,j}(a^1-a^3)_{s_is_{i+1}t_jt_{j+1}}=0
$$ 
which gives the following equality $\iint y\dd_1x\dd_2z=\iint y\dd_1x\bullet \dd_2z$. The other identity is obtained by exactly the same computations.
\end{proof}

\section{Analysis of a two-parameter integral}
\label{sec:2-d-dissection}
Following the informal approach of sec.~\ref{sec:dissection} we would
like to get inspired for a general construction from the
\emph{analysis} of a concrete two-dimensional integral. Then consider
$\iint \varphi(x) \dd x$ for a smooth surface $x : \RR^2 \to \RR$ and a
smooth function $\varphi : \RR \to \RR$. The decomposition of
eq.(\ref{eq:int-decomp}) gives
$$
\iint \varphi(x) \dd x = - \varphi(x) \int \dd x 
+ \int_1 \varphi(x) \int_2 \dd x
+ \int_2 \varphi(x) \int_1 \dd x
+ \iint \dd\varphi(x) \dd x
$$
where, for example, we used the notation $\int_1 \varphi(x) \int_2
\dd x$ to mean
$$
\left(\int_1 \varphi(x) \int_2 \dd x\right)_{(s_1,s_2;t_1,t_2)} :=
\int_{s_1}^{t_1} \varphi(x_{u_1,s_2}) \int_{s_2}^{t_2} \partial_1
\partial_2 x(u_1,u_2) \dd u_1 \dd u_2 
$$
and the others analogous expressions.

This decomposition is at the origin of Prop.\ref{prop:young2d} when
the demanded regularity is satisfied since the iterated integral is
given by the formula
$$
\iint \dd\varphi(x) dx = - \Lambda \left[\der \varphi(x) \der x \right].
$$

To proceed further we note that
\begin{equation}
  \label{eq:dec2-2}
\iint \dd\varphi(x) = \iint \varphi'(x) \dd x +  \iint \varphi''(x) (\dd_1 x
\dd_2 x)    
\end{equation}
(recall that $\dd = \dd_1 \dd_2$ is not a  derivation). Take the first term
in the r.h.s. and use again eq.(\ref{eq:int-decomp}) to write
\begin{equation}
  \label{eq:dec2-3}
 \iint \varphi'(x) \dd x = - \varphi'(x) \int \dd x 
+ \int_1 \varphi'(x) \int_2 \dd x 
+  \int_2 \varphi'(x) \int_1 \dd x
+ \iint \dd\varphi'(x) \dd x
\end{equation}
and a similar equation for the term $\iint \varphi''(x) \dd_1x \dd_2 x$.
Then
\begin{equation}
  \label{eq:dec2-4}
  \begin{split}
\iint \dd\varphi(x) \dd x = \iint \varphi'(x) \dd x \dd x + \iint \varphi''(x)
(\dd_1 x \dd_2 x) \dd x.    
  \end{split}
\end{equation}
For simplicity we treat explicitly the first term in the r.h.s, the
analysis  of the second being similar.

Using eq.(\ref{eq:dec2-3}) and the definition of iterated integral, we have
\begin{equation}
  \label{eq:dec2-5}
 \iint \varphi'(x) \dd x \dd x = - \varphi'(x) \int \dd x \dd x 
+ \int_1 \varphi'(x) \int_2 \dd x \dd x 
+  \int_2 \varphi'(x) \int_1 \dd x \dd x
+ \iint \dd\varphi'(x) \dd x \dd x
\end{equation}
This expression seems complicated, however it shows that, in
order to control the l.h.s. we need two ingredients:
\begin{itemize}
\item[1)] Being able to define \emph{essentially one-dimensional
    integrals} like  
\begin{equation}
    \label{eq:boundary-integrals}
\int_1 \varphi(x) \int_2 \dd x, \int_1 \varphi'(x) \int_2 \dd x \dd x,
\int_1 \varphi''(x) \int_2 \dd_1x \dd_2 \dd x, \dots
  \end{equation}
\item[2)] Control the \emph{remainders} given by the three-fold
  iterated integrals
$$
\mathscr{R} := \iint \dd\varphi'(x) \dd x \dd x
\qquad
\mathscr{\widetilde R} := \iint \dd\varphi''(x) \dd_1 x \dd_2 x \dd x.
$$
\end{itemize}

\subsubsection{The boundary integrals}
We call integrals like those appearing in
eq.(\ref{eq:boundary-integrals}) \emph{boundary integrals} to
emphasize their one-dimensional nature (which, we hope,  will be clear
from what follows). Their appearance is characteristic of the
multidimensional setting and it is  linked to the cohomological
structure of the complex $(\CCC_*,\der)$ studied in
Sec.~\ref{sec:two-dim}. 
Take the first of them and expand it according to the (one-dimensional) eq.(\ref{eq:dec-1}):
\begin{equation}
  \label{eq:dec2-6}
\int_1 \varphi'(x) \int_2 \dd x \dd x = \varphi'(x) \iint \dd x \dd x + \int_1 \dd_1
\varphi'(x) \int_2 \dd x \dd x  
\end{equation}
Next, apply $\der_1$ to the second term in the r.h.s.:
\begin{equation}
  \label{eq:dec2-6-bis}
- \der_1 \int_1 \dd_1
\varphi'(x) \int_2 \dd x \dd x = \int_1 \dd_1
\varphi'(x) \int_1 \int_2 \dd x \dd x + \int_1 \dd_1
\varphi'(x) \int_2 \dd x \int_1 \dd x
\end{equation}
The first term in the r.h.s. is \emph{simple} since it is equivalent
to $\der_1 \varphi'(x) \iint \dd x \dd x$, and can be controlled with some
regularity of $\varphi'$ and the a-priori knowledge of $\iint \dd x \dd x$.
The second needs to be further expanded as
\begin{equation}
  \label{eq:dec2-7}
  \begin{split}
\int_1  \dd_1
\varphi'(x)  \int_2 \dd x \int_1 \dd x
& =\int_1 
\varphi''(x) \dd_1x \int_2 \dd x \int_1 \dd x
\\ &= 
\varphi''(x) \int_1  \dd_1x \int_2 \dd x \int_1 \dd x
+ \int_1 \dd_1 \varphi''(x)  \dd_1x \int_2 \dd x \int_1 \dd x
 \end{split}
\end{equation}
Again, the first term in the r.h.s do not pose any further problem so
we continue to study only the last one. Set
$$
\mathscr{A}_1 :=  \int_1 \dd_1 \varphi''(x)   \dd_1x \int_2 \dd x \int_1 \dd x
$$
We apply to $\mathscr{A}_1$ the splitting
operator $S_1$ obtaining
\begin{equation}
  \label{eq:split}
 S_1 \mathscr{A}_1 = 
\int_1 \dd_1 \varphi''(x) \dd_1 x \int_2 \dd x \otimes_1 \int_1 \dd x  
\end{equation}
Let us clarify the meaning of this last term: in the first direction
is the tensor product of two 1-increments, while in the second
direction is a 1-increment. Taking four points $(u_1,u_2,u_3,u_4)$ in
direction 1 and two $(v_1,v_2)$ in direction 2, its value is given by
the expression
$$
\int_{u_3}^{u_4}\int_{v_1}^{v_2}\left(\int_{u_1}^{u_2}\int_{v_1}^{v}\left(\int_{u_1}^{u}\int_{u_1}^{a}\dd_b\varphi"(x_{bv_1})\dd_ax_{av_1}\right)\dd_{uv}x_{uv}\right)\dd_{st}x_{st}
$$
As the reader can easily check, by setting $u_2=u_3$ we reobtain
$\mathscr{A}_1$.

Denote with $\der_1 \otimes_1 1$ the action of the $\der_1$ operator
on the first factor of an element of $\CC_1 \otimes \CC_1$, where
again the $1$ as index of the tensor operation denote that it acts on
tensor products according to the first direction.
Apply  $\der_1 \otimes_1 1$ to the r.h.s of
eq.~(\ref{eq:split}) last term in the r.h.s to obtain
\begin{equation*}
  \begin{split}
- (\der_1 \otimes_1 1) &  \int_1 \dd_1 \varphi''(x)   \dd_1x \int_2 \dd x \otimes_1 \int_1
\dd x
  =
- \left[\der_1   \int_1 \dd_1 \varphi''(x)   \dd_1x \int_2 \dd x\right] \otimes_1 \int_1
\dd x
\\ &  = \int_1 \dd_1 \varphi''(x)  \int_1 \dd_1x \int_2 \dd x \otimes_1 \int_1
\dd x
 +
\int_1\dd_1 \varphi''(x)   \dd_1x \int_1 \int_2 \dd x \otimes_1 \int_1
\dd x
\end{split}    
\end{equation*}
Hence let
$$
\mathscr{AA}_1 := \int_1 \dd_1 \varphi''(x)  \int_1 \dd_1x \int_2 \dd x
\otimes_1 \int_1
\dd x = \der_1 \varphi''(x) \int_1 \dd_1x \int_2 \dd x\otimes_1 \int_1 \dd x
$$
and (cfr. eq.~(\ref{eq:taylor-1}))
$$
\mathscr{AB}_1 := \int_1 \dd_1 \varphi''(x)   \dd_1x \int_1 \int_2 \dd x
\otimes_1 \int_1 \dd x
= \left( \der_1 \varphi''(x) - \varphi''(x) \der_1 x \right)
\int_1 \int_2 \dd x \otimes_1 \int_1
dx
$$
which depend only on the function $\varphi''(x)$ and on the
``splitted'' iterated
integrals
$$
\int_1 \dd_1x \int_2 \dd x\otimes_1 \int_1 \dd x, \qquad
\int_1 \int_2 \dd x\otimes_1 \int_1
\dd x
$$
Since we are working under smoothness conditions is possible to
recover  $\mathscr{A}_1$ by applying $\Lambda_1 \otimes_1 1$:
$
\mathscr{A}_1 =- (\Lambda_1 \otimes_1 1)(\mathscr{AA}_1+\mathscr{AB}_1)
$.
Moreover 
$$
\mathscr{C}_1 := \int_1
\dd_1 \varphi''(x)\dd_1x = 
\varphi''(x) \int_1 \dd_1x \int_2\dd x \int_1 \dd x
+ \mu_1 \mathscr{A}_1
$$
where recall that $\mu_1$ is the multiplication in the first direction
which is the inverse operation of the splitting $S_1$.
Then we can recover also the term $\int_1  \dd_1
\varphi'(x)  \int_2 \dd x \int_1\dd x$ by eq.(\ref{eq:dec2-6}) and an
application of $\Lambda_1$. Finally we have obtained the following
expression for the boundary term:
$$
\int_1 \varphi'(x) \int_2\dd x \dd x = \varphi'(x) \iint \dd x \dd x -\Lambda_1 
\left[  
\der_1 \varphi'(x) \int_1 \int_2 \dd x \dd x + \mathscr{C}_1
\right]
$$
where the r.h.s. depends only on a finite number of iterated integrals
of $x$.

\subsubsection{The remainders $\mathscr{R}$,$\mathscr{\widetilde R}$}
The three-fold iterated integral $\iint d\varphi'(x) dx dx$ will be
analyzed in terms of its image under $\der$:
\begin{equation}
  \label{eq:dec2-r1}
  \begin{split}
\der   \iint \dd\varphi'(x) \dd x \dd x   
& = \iint \dd\varphi'(x) \iint \dd x \dd x + \iint \dd\varphi'(x) \dd x  \iint \dd x
\\ &\qquad  + \iint \dd\varphi'(x) \int_1 \dd x \int_2 \dd x
+ \iint \dd\varphi'(x) \int_2 \dd x \int_1 \dd x
  \end{split}
\end{equation}
The first term is readily controlled in the same spirit as
above. Consider the fourth (the third is similar) which we will write
as
\begin{equation}
  \label{eq:dec2-r2}
  \begin{split}
\iint \dd\varphi'(x) \int_2 \dd x \int_1 \dd x 
& = \int_1 \dd_1 [\der_2 \varphi'(x)] \int_2 \dd x \int_1 \dd x 
 \\ &    
= \int_1 \der_2 [\varphi''(x) \dd_1 x] \int_2 \dd x \int_1 \dd x 
 \\ &    
= \int_1 \der_2 [\varphi''(x)] \dd_1 x \int_2 \dd x \int_1 \dd x 
+ \int_1 \varphi''(x) \dd_1 \der_2 x \int_2 \dd x \int_1 \dd x 
 \\ &    
= \mathscr{D}_1 + \mathscr{E}_1
  \end{split}
\end{equation}
Expanding the integral of $\der_2 [\varphi''(x)] \dd_1 x$ we get
\begin{equation}
  \label{eq:dec2-r3}
  \begin{split}
 \mathscr{D}_1 & =  \int_1 \der_2 [\varphi''(x)] \dd_1 x \int_2 \dd x
 \int_1 \dd x
\\ & =  \der_2 [\varphi''(x)] \int_1  \dd_1 x \int_2 \dd x
 \int_1 \dd x +  \int_1 \dd_1 \der_2 [\varphi''(x)] \dd_1 x \int_2 \dd x
 \int_1 \dd x
  \end{split}
\end{equation}
Again we apply $(\der_1 \otimes_1 1)S_1$ to the second term:
\begin{equation}
  \label{eq:dec2-r4}
  \begin{split}
-(\der_1 & \otimes_1 1) S_1 \int_1  \dd_1 \der_2 [\varphi''(x)] \dd_1 x \int_2 \dd x
 \int_1 \dd x 
\\ &  =
\int_1 \dd_1 \der_2 [\varphi''(x)] \int_1\dd_1 x \int_2 \dd x
\otimes_1 \int_1 \dd x +
\int_1 \dd_1 \der_2 [\varphi''(x)]  \dd_1 x \int_1 \int_2 \dd x
\otimes_1 \int_1 \dd x     
\\ &  =: \mathscr{DA}_1 + \mathscr{DB}_1
  \end{split}
\end{equation}
From which we get:
\begin{equation}
  \label{eq:dec2-r5}
 \mathscr{D}_1 = \der_2 [\varphi''(x)] \int_1  \dd_1 x \int_2 \dd x
 \int_1 \dd x -\mu_1 [\Lambda_1 \otimes_1 1 ] (\mathscr{DA}_1 + \mathscr{DB}_1)  
\end{equation}
And in the same way:
\begin{equation}
  \label{eq:dec2-r6}
 \mathscr{E}_1 = \varphi''(x) \int_1 \dd_1 \der_2 x \int_2 \dd x
 \int_1 \dd x -\mu_1 [\Lambda_1 \otimes_1 1 ] (\mathscr{EA}_1 + \mathscr{EB}_1)  
\end{equation}
with
$$
\mathscr{EA}_1 := \der_1 [\varphi''(x)] \int_1 \dd_1 \der_2 x \int_2 \dd x
\otimes_1 \int_1 \dd x 
$$
and
$$
\mathscr{EB}_1 :=
\int_1 d_1 \varphi''(x)  \dd_1 \der_2 x \int_1 \int_2 \dd x
\otimes_1 \int_1 \dd x.  
$$
Note that
\begin{equation}
  \label{eq:dec2-r5}
  \begin{split}
\mathscr{F}_1 & := \mathscr{DB}_1 + \mathscr{EB}_1  
 \\ & =
\der_2 \left[\int_1 \dd_1 \varphi''(x)  \dd_1  x\right] \int_1 \int_2 \dd x
 \otimes_1 \int_1 \dd x  
 \\ & =
\der_2 \left[\der_1 \varphi'(x) - \varphi''(x) \der_1 x\right] \int_1 \int_2 \dd x
 \otimes_1 \int_1 \dd x.  
  \end{split}
\end{equation}
Together
eq.(\ref{eq:dec2-r2}),~(\ref{eq:dec2-r5})~(\ref{eq:dec2-r6})~(\ref{eq:dec2-r5})
imply the representation:
\begin{equation}
  \label{eq:dec2-r6}
  \begin{split}
\iint \dd\varphi'(x) \int_2 \dd x \int_1 \dd x & =
 \der_2 [\varphi''(x)] \int_1  \dd_1 x \int_2 \dd x
  \int_1 \dd x 
+ \varphi''(x) \int_1  \dd_1 \der_2 x \int_2 \dd x
  \int_1 \dd x
\\ & \qquad - \mu_1 [\Lambda_1 \otimes_1 1 ] (\mathscr{DA}_1 + \mathscr{EA}_1 + \mathscr{F}_1).       
  \end{split}
\end{equation}
The only term, appearing in $\mathscr{R}$ which is left is the second
term in eq.(\ref{eq:dec2-r1}): $\iint d\varphi'(x) dx \iint dx$. This
term can be handled together a similar term appearing in the
expansion of $\mathscr{\widetilde R}$ which reads:
$$
\iint \dd\varphi''(x) \dd_1x \dd_2 x \iint \dd x
$$
(all the other terms in  $\mathscr{\widetilde R}$ are handled as
above).
Indeed we have that the sum of $\iint d\varphi'(x) dx$ and $\iint
d\varphi''(x) d_1x d_2 x$ which are nontrivial two-dimensional
iterated integrals, appear in the expansion for $\der \varphi(x)$:
\begin{equation}
  \label{eq:dec2-rr1}
  \begin{split}
\der \varphi(x) & =  
 - \varphi'(x) \iint \dd x 
+ \int_1 \varphi'(x) \int_2 \dd x +  \int_2 \varphi'(x) \int_1 \dd x
\\ & \quad - \varphi''(x) \iint \dd_1 x \dd_2 x   
+ \int_1 \varphi''(x) \int_2 \dd_1x \dd_2x +  \int_2 \varphi''(x) \int_1 \dd_1x \dd_2x
\\ & \quad + \iint \dd\varphi'(x) \dd x
+ \iint d\varphi''(x) \dd_1x \dd_2x
  \end{split}
\end{equation}
In this expression all the terms, except the last two can be expressed,
following the approach we used for the boundary integrals above, as
functional of a small number of integrals of $x$.

\subsection{Rough sheet}\label{sec:rough}
We have shown in this section how the 2-dimensional integral $\iint
\varphi(x) \dd x$ admits to be expressed as a well behaved functional
$F$ of
a family $\mathbb{X}$ of iterated integrals of $x$ and  of $\varphi(x)$. This
functional can then be extended to more irregular
functions $x$, not necessarily smooth, in two ways:
\begin{itemize}
\item[a)] \emph{Algebraic approach.} We are given a family $\mathbb{X}$ of biincrements which
  satisfy algebraic conditions analogous to that which allowed us to
  perform the computations in this section. In this case, the integral
  can be defined by the same functional $F$. The algebraic relations
  are then needed to show that such a definition is consistent with
  our notion of integral (e.g. that this integral is in the kernel of
  both $\der_1$ and $\der_2$).

\item[b)] \emph{Geometric approach.} We are able to show that, there exists a sequence of
  families $\mathbb{X}^n$ obtained by iterated integrals over smooth
  2-dimensional functions $x^n$, which converges, under
  suitable topologies on the biincrements, to a limiting family
  $\mathbb{X}$. Then by the continuity of the map $F$ we are able to
  identify the limit of $\iint \varphi(x^n) \dd x^n$ and to consider it
  as an extension of the integral over smooth 2-parameter
  functions. This is analogous to the \emph{geometric} theory of rough paths.
\end{itemize}

Already in the one-dimensional theory the
two approaches can give different notions of integrals. 

\subsubsection{Algebraic assumption and Boundary integrals}
\begin{hypothesis}
\label{hyp:chen-2d}
If $a=1,2$ then $\hat a = 2,1$.
Assume that $\alpha,\beta > 1/3$ .
Now as in the one parameter case we assume the existence of some algebraic object
\begin{equation*}
\begin{split}
&A^x,A^\omega  \in \CCC_{2,2}^{\alpha,\beta}; \quad
B^{xx}_{1},B_{1}^{x\omega} \in \CCC_{2,2}^{2\alpha,\beta};\quad B_2^{xx},B_{2}^{x\omega}\in\CCC_{2,2}^{\alpha,2\beta}  
\\&C^{xx},C^{\omega x},C^{x\omega},C^{\omega\omega}\in \CCC_{1,1}^{2\alpha,2\beta}; \quad
D^{xx}_{1},D_{1}^{\omega x},D_1^{\omega\omega}\in (\CC_{2}^{\alpha}\otimes_1 \CC_{2}^{\alpha})(\CC_1^{2\beta}),\quad D^{xx}_{2},D_{2}^{\omega x},D_2^{\omega\omega}\in (\CC_{2}^{\beta}\otimes_2 \CC_{2}^{\beta})(\CC_1^{2\alpha})\end{split}
\end{equation*}
$$
E^{xxx}_{1},E_{1}^{x\omega x},E_{1}^{xx\omega},E_1^{x\omega\omega}\in (\CC_{2}^{2\alpha}\otimes_1 \CC_{1}^{\alpha})(\CC_1^{2\beta}); \quad
F^{xxx}_{1},F_{1}^{x\omega x},F_1^{x\omega\omega},F^{xx\omega}_1 \in (\CC_{2}^{2\alpha}\otimes_1 \CC_{2}^{\alpha})(\CC_1^{3\beta})
$$
$$
E^{xxx}_{2},E_{2}^{x\omega x},E_{2}^{xx\omega},E_2^{x\omega\omega}\in (\CC_{2}^{2\beta}\otimes_2 \CC_{1}^{\beta})(\CC_1^{2\alpha}); \quad
F^{xxx}_{2},F_{2}^{x\omega x},F_2^{x\omega\omega},F^{xx\omega}_2 \in (\CC_{2}^{2\beta}\otimes_2 \CC_{2}^{\beta})(\CC_1^{3\alpha})
$$
satisfying the following equations
\begin{enumerate}
\item $A^{x} = \der x$ ,
\item $\der_{a} B^{xx}_{a} = (\der_{a} x) A^{x}$, $\der_{\hat a} B^{xx}_{a} =- \mu_{\hat a} D_{\hat a}^{xx}$ 
\item $\der_{a} C^{xx} = \mu_{a} D^{xx}_{a}$ 
\item $(1\otimes_{a}\der_{a}) D^{xx}_{a} = 0$, $\der_{\hat a} D^{xx}_{a} = A^{x} \otimes_{a} A^{x}$, 
\item$(1\otimes_{a}\der_{a})E_{a}^{xxx}=0$, $(\der_{a}\otimes_{a} 1)E^{xxx}=\der_{a} x\otimes D^{xx}_{a} $ , $\der_{\hat a} E^{xxx}_{a} = F^{xxx}_{a} + B^{xx}_{a} \otimes_{a} A^x$,
\item$\der_{a} A^{\omega}=0$
\item$\der_{a} B^{x \omega}_{a}=(\der_{a} x)A^{\omega}$, $\der_{\hat a} B_{a}^{x\omega}=\mu_{\hat a} D^{x\omega}_{a}$
\item$\der_{a} C^{\omega x} =\mu_{a} D^{\omega x}_{a}$
\item$(1\otimes_{a}\der_{a})D^{\omega x}_{a}=0$, $\der_{\hat a}D^{\omega x}_{a} = A^{\omega}\otimes_{a}A^{x}$
\item$(1\otimes_{a}\der_{a})E_{a}^{x\omega x}=0$, $(\der_{a}\otimes 1)E^{x\omega x} = \der_{a} x\otimes_{a} D_{a}^{\omega x}$, $\der_{\hat a} E_{a}^{x\omega x} =   F_{a}^{x\omega x}+B_{a}^{x\omega}\otimes_{a} A^{\omega}$
\item $\der_{a} C^{x\omega} = \mu_{a} D^{x\omega}_{a}$ 
\item $(1\otimes_{a}\der_{a}) D^{x\omega}_{a} = 0$, $\der_{\hat a} D^{x\omega}_{a} = A^{x} \otimes_{a} A^{\omega}$, 
\item$(1\otimes_{a}\der_{a})E_{a}^{xx\omega}=0$, $(\der_{a}\otimes_{a} 1)E^{xx\omega}=\der_{a} x\otimes D^{x\omega}_{a} $ , $\der_{\hat a} E^{xx\omega}_{a} = F^{xx\omega}_{a} + B^{xx}_{a} \otimes_{a} A^\omega$,
\item $\der_{a} C^{\omega\omega} = \mu_{a} D^{\omega\omega}_{a}$ 
\item $(1\otimes_{a}\der_{a}) D^{\omega\omega}_{a} = 0$, $\der_{\hat a} D^{\omega\omega}_{a} = A^{\omega} \otimes_{a} A^{\omega}$, 
\item$(1\otimes_{a}\der_{a})E_{a}^{x\omega\omega}=0$, $(\der_{a}\otimes_{a} 1)E^{x\omega\omega}=\der_{a} x\otimes D^{\omega\omega}_{a} $ , $\der_{\hat a} E^{x\omega\omega}_{a} = F^{x\omega\omega}_{a} + B^{x\omega}_{a} \otimes_{a} A^\omega$,
\end{enumerate}
\begin{remark}
When $x$ is a smooth sheet we can choose this algebraic object as the following iterated integral :
 \begin{enumerate}
\item $(A^{x})_{s_{1}s_{2}t_{1}t_{2}}=(\iint \dd x)_{s_{1}s_{2}t_{1}t_{2}}= \der x_{s_{1}s_{2}t_{1}t_{2}}$
\item$(A^{\omega})_{s_{1}s_{2}t_{1}t_{2}}= (\iint \dd \omega)_{s_{1}s_{2}t_{1}t_{2}}=\int_{s_{1}}^{s_{2}}\int_{t_{1}}^{t_{2}}\dd_{s}x_{st}\dd_{t}x_{st}$
\item$(B_{1}^{xx})_{s_{1}s_{2}t_{1}t_{2}}= (\int_1 \dd_1 x\int_2 \dd x)_{s_{1}s_{2}t_{1}t_{2}}=\int_{s_{1}}^{s_{2}}\int_{t_{1}}^{t_{2}}\der_{1}x_{s_{1}st_{1}}\dd_{st}x_{st}$
\item$(B_{1}^{x\omega})_{s_{1}s_{2}t_{1}t_{2}}= (\int_1 \dd_1 x\int_2 \dd \omega)_{s_{1}s_{2}t_{1}t_{2}}=\int_{s_{1}}^{s_{2}}\int_{t_{1}}^{t_{2}}\der_{1}x_{s_{1}st_{1}}\dd_{s}x_{st}\dd_{t}x_{st}$
\item$(C^{xx})_{s_{1}s_{2}t_{1}t_{2}}= (\iint  \dd x \dd x)_{s_{1}s_{2}t_{1}t_{2}}=\iint_{(s_{1},t_{1})}^{(s_{2},t_{2})}\left(\iint_{(s_{1},t_{1})}^{(s,t)}\dd_{rr'}x_{rr'}\right)\dd_{st}x_{st}$
\item$(C^{\omega x})_{s_{1}s_{2}t_{1}t_{2}}= (\iint \dd \omega  \dd x)_{s_{1}s_{2}t_{1}t_{2}}=\iint_{(s_{1},t_{1})}^{(s_{2},t_{2})}\left(\iint_{(s_{1},t_{1})}^{(s,t)}\dd_{r}x_{rr'}\dd_{r'}x_{rr'}\right)\dd_{st}x_{st}$
\item$(D_{1}^{xx})_{s_{1}s_{2}s_{3}s_{4}t_{1}t_{2}}= (\int_2 \int_1 \dd x \otimes_1 \int_1 \dd x)_{s_{1}s_{2}s_{3}s_{4}t_{1}t_{2}}=\iint_{(s_{3},t_{1})}^{(s_{4},t_{2})}\left(\iint_{(s_{1},t_{1})}^{(s_{2},t)}d_{rr'}x_{rr'}\right)\dd_{st}x_{st}$
\item$(E_{1}^{xxx})_{s_{1}s_{2}s_{3}s_{4}t_{1}t_{2}}= (\int_1 \dd_1 x\int_2 \dd x\otimes_1 \int_1 \dd x)_{s_{1}s_{2}s_{3}s_{4}t_{1}t_{2}}=\iint_{(s_{3},t_{1})}^{(s_{4},t_{2})}\left(\iint_{(s_{1},t_{1})}^{(s_{2},t)}\der_{1}x_{s_{1}rt_{1}}d_{rr'}x_{rr'}\right)\dd_{st}x_{st}$
\item 
\begin{equation*}
\begin{split}
&(F_{1}^{xxx})_{s_{1}s_{2}s_{3}s_{4}t_{1}t_{2}t_{3}}=(\iint\dd x \int_2  \dd x \otimes_1 \int_1 \dd x)_{s_{1}s_{2}s_{3}s_{4}t_{1}t_{2}t_{3}}\\&=-\iint_{(s_{3},t_{2})}^{(s_{4},t_{3})}(\iint_{(s_{1},t_{2})}^{(s_{2},t)}(\iint_{(s_{1},t_{1})}^{(r,t_{2})}\dd_{ab}x_{ab})\dd_{rr'}x_{rr'})\dd_{st}x_{st}
\end{split}
\end{equation*}
\end{enumerate}
\end{remark}
\end{hypothesis}
In this section we assume the previous hypothesis to be  true and we will give a "reasonable" construction  of the following boundary integrals :
$$
\int_a y\int_{\hat a} \dd x,\quad
\int_a y\int_{\hat a} \dd x \dd x,
$$
$$
\int_a y\int_{\hat a} \dd{\omega} ,\quad
\int_a y\int_{\hat a} \dd{\omega} \dd x,
$$
which allow us to construct the space of two  parameters controlled sheet and integrate them.
We begin by recalling the notion of a one dimensional  controlled path which the space of the sheet $y$ satisfying the following assumption :
$$
 \der _{a}y=y^{x_{a}}\der_{a}x+y^{\sharp a} , \quad y^{x_{a}}\in \CCC_{1,1}^{\alpha,\beta}, \quad (y^{\sharp 1},y^{\sharp2})\in\CCC_{2,1}^{2\alpha,\beta}\times\CCC_{1,2}^{\alpha,2\beta}
$$
where $x\in \mathscr C^{\alpha,\beta}_{1,1}$,  $a\in\{1,2\}$ and we denoted by $\mathscr Q_x^{\alpha,\beta}$ this space. Now we will set out a permutation lemma that is useful to conduct the computation in the following. 
\begin {lemma}
\label{lemma:mu-lambda-commute}
We have for $h\in (\CC_{2} \otimes_{a} \CC_{2})(\CC_{2})$ the following identity:
\begin{equation}
\der_{a}\mu_a h = \mu_{a}(\der_{a}\otimes_{a}1)h-\mu_a(1\otimes_{a}\der_a)h .
\end{equation}
\end {lemma}
And then we have the following construction for the Boundary integral :
\begin{proposition}\label{boundry}
Assume that the hypothesis~\ref{hyp:chen-2d} to be true, and let $y\in\mathscr Q^{\alpha,\beta}_{x}$. Then we define the boundary integral by :
\begin{enumerate}
\item $\int_{a} y\int_{\hat a} \dd x := yA^x + y^{x_a}B_a^{xx}+\Lambda_a[y^{\sharp a}A^x + \der_a y^{x_a}B_{a}^{xx}]$,
\item$\int_{a} y\int_{\hat a} \dd x\dd x := yC^{xx} + \Lambda_{a} [\der_{a} yC^{xx} + y^{x}\mu_{a}E_{a}^{xxx} + \mu_{a}(\Lambda_{a}\otimes_{a} 1)(y^{\sharp a}D_{a}^{xx} + \der_{a} y^{x_{a}} E_{a}^{xxx})]$,
\item $\int_{a} y\int_{\hat a} \dd\omega := yA^\omega + y^{x_a}B_a^{x\omega}+\Lambda_a[y^{\sharp a }A^\omega + y^{x}B_{a}^{x\omega}]$
\item $\int_{a} y\int_{\hat a} \dd\omega \dd x := yC^{\omega x} + \Lambda_{a}[\der_{a} yC^{\omega x} + y^{x}\mu_a E_{a}^{x\omega x} + \mu_{a}(\Lambda_{a}\otimes_{a}1)(y^{\sharp a}D_{a}^{\omega x} + \der_{a}y^{x_a}E_{a}^{x\omega x})]$
\item $\int_{a} y\int_{\hat a} \dd x \dd\omega := yC^{x\omega} + \Lambda_{a}[\der_{a} yC^{x\omega} + y^{x}\mu_a E_{a}^{xx\omega} + \mu_{a}(\Lambda_{a}\otimes_{a}1)(y^{\sharp a}D_{a}^{x\omega} + \der_{a}y^{x_a}E_{a}^{xx\omega})]$
\item $\int_{a} y\int_{\hat a} \dd \omega \dd\omega := yC^{\omega\omega} + \Lambda_{a}[\der_{a} yC^{\omega\omega} + y^{x}\mu_a E_{a}^{x\omega\omega} + \mu_{a}(\Lambda_{a}\otimes_{a}1)(y^{\sharp a}D_{a}^{\omega\omega} + \der_{a}y^{x_a}E_{a}^{x\omega\omega})]$
\end{enumerate}
Moreover all these formulas have meaning and when $x$ is differentiable we can choose the rough sheet so that they coincide well with their definition in the  Riemann-Stieltjes case, which justifies the notation.
\end{proposition}
\begin{proof}
We will only prove the first two formula, for the others we have identical proofs. Let now $x$ a smooth sheet and $y\in\mathscr Q_x^{\alpha,\beta}$, then we have easily the following expansion 
\begin{equation*} 
\begin{split}
\left(\int_1 y\int_{2} \dd x\right)_{s_{1}s_{2}t_{1}t_{2}} &\coloneq \iint_{(s_{1},t_{1})}^{(s_{2},t_{2})} y_{s{t_1}} \dd_{st}x_{st}
= \int_{s_1}^{s_2}  y_{s{t_1}}\dd_1\der_{2}  x_{st_{1}t_{2}}\\& = y_{s_{1}t_1}(A^{x})_{s_{1}s_{2}t_{1}t_2} +y^{x}_{s_{1}t_1} (B_{1}^{xx})_{s_{1}s_{2}t_{1}t_{2}} +\int_{s_{1}}^{s_2}  y^{\sharp1}_{s_{1}st_{1}}  \dd_ {s}\der_{2} x_{st_{1}t_2}
\end{split}
\end{equation*} 
where $A^x=\der x$ and $(B_1^{xx})_{s_1s_2t_1t_2}=\int_{s_1}^{s_2}\int_{t_1}^{t_2}\der_1x_{s_1st_1}\dd_{st}x_{st}$. Now is easy to check that  $r_{s_{1}s_{2}t_{1}t_{2}}\coloneq\int_{s_{1}}^{s_{2}}y^{\sharp1}_{s_{1}st_{1}}d_{s}\der_{2}x_{st_{1}t_{2}}\in\CC_{2,2}^{3\alpha,\beta}$ and $\der_{1}r=y^{\sharp1}\der x+\der_{1}y^{x_{1}}B^{xx}_{1}\in\CCC_{3,2}^{3\alpha,\beta}$. So finally we obtain :
$$
r=\Lambda_{1}[y^{\sharp1}\der x+\der_{1}y^{x_{1}}B^{xx}_{1}]
$$
and thus 
$$
\int_{1} y\int_{2} \dd x := yA^x + y^{x_1}B_{1}^{xx}+\Lambda_{1}[y^{\sharp 1}A^x + \der_{1} y^{x_{1}}B_{1}^{xx}]
$$
Now we remark that this formula is still valid when $x$ satisfy only the assumption~\eqref{hyp:chen-2d}. Now we will focus on the second equation which requires a different argument. A quick computation gives :
\begin{equation*}
\begin{split}
\left(\int_{1} y\int_{2}\dd x\dd x\right)_ {s_{1}s_{2}t_{1}t_{2}}&\coloneq \iint_{(s_{1},t_{1})}^{(s_{2},t_{2})}\iint_{(s_{1},t_{1})}^{(s,t)} y_{rt_{1}} \dd_{rr'}x_{rr'} \dd_{st}x_{st}
\\&=y_{s_{1}t_{1}}(C^{xx})_{s_{1}s_{2}t_{1}t_{2}}+\iint_{(s_{1},t_{1})}^{(s_{2},t_{2})}\iint_{(s_{1},t_{1})}^{(s,t)} \der_{1} y_{s_1rt_{1}} \dd_{rr'}x_{rr'} \dd_{st}x_{st}
\end{split}
\end{equation*}
The first term of this last equation is well understood, we will focus on the second term denoted by $ \mathscr H $ in the sequel. We observed that $\mathscr H\in\CCC_{2,2}^{3\alpha,2\beta}$, then:  
\begin{equation*}
\begin{split}
\der_{1}\mathscr H_{s_{1}s_{2}s_{3}t_{1}t_{2}}&= \der_{1}y_{s_{1}s_{2}t_{1}}C^{xx}_{s_{2}t_{1}s_{3}t_{2}}+y^{x}_{s_{1}t_{1}}\iint_{(s_{2},t_{1})}^{(s_{3},t_{2})}\iint_{(s_{1},t_{1})}^{(s_{2},t)}\der_{1} x_{s_{1}rt_{1}} d_{rr'}x_{rr'} d_{st}x_{st} \\&+\iint_{(s_{2},t_{1})}^{(s_{3},t)}\iint_{(s_{1},t_{1})}^{(s_{2},t)} y^{\sharp1}_{s_{1}rt_{1}}  d_{rr'}x_{rr'} d_{st}x_{st} . 
\end{split}
\end{equation*}
Now, if we put 
\begin{equation*}
\mathscr A_{s_{1}s_{2}s_{3}t_{1}t_{2}}:= \iint_{(s_{2},t_{1})}^{(s_{3},t)}\iint_{(s_{1},t_{1})}^{(s_{2},t)} y^{\sharp1}_{s_{1}rt_{1}}  d_{rr'}x_{rr'} d_{st}x_{st}
\end{equation*}
we obtain
\begin{equation*}
\begin{split}
(\der_{1}\otimes_{1} 1)\mathscr S_1 (\mathscr A_1)_{s_{1}s_{2}vs_{3}s_{4}t_{1}t_{2}} &=\iint_{(s_{3},t_{1})}^{(s_{4},t_{2})} \iint_{(s_{2},t_{1})}^{(v,t)}  y^{\sharp}_{s_{1}rt_{1}}-y^{\sharp1}_{s_{2}rt_{1}}\dd_{rr'}x_{rr'}\dd_{st}x_{st}\\
&=y_{s_{1}s_{2}t_{1}}^{\sharp1}\iint_{(s_{3},t_{1})}^{(s_{4},t_{2})}\iint_{(s_{2},t_{1})}^{(v,t)}\dd_{rr'}x_{rr'}\dd_{st}x_{st} \\&+\iint_{(s_{3},t_{1})}^{(s_{4},t_{2})}\iint_{(s_{2},t_{1})}^{(v,t)}\der_{1}y^{\sharp1}_{s_{1}s_{2}rt_{1}}\dd_{rr'}x_{rr'}\dd_{st}x_{st} 
\\&=y_{s_{1}s_{2}t_{1}}^{\sharp1}\iint_{(s_{3},t_{1})}^{(s_{4},t_{2})}\iint_{(s_{2},t_{1})}^{(v,t)}\dd_{rr'}x_{rr'}\dd_{st}x_{st}
\\&+\der_{1}y^{x}_{s_{1}s_{2}t_{1}}\iint_{(s_{3},t_{1})}^{(s_{4},t_{2})}\iint_{(s_{2},t_{1})}^{(v,t)}\der_{1}x_{s_{2}rt_{1}}\dd_{rr'}x_{rr'}\dd_{st}x_{st}
\end{split}
\end{equation*}
Then if we recall that the two last iterated integrals are denoted respectively by $D_{1}^{xx}$ and $E_{1}^{xxx}$ we obtain:
\begin{equation*}
\mathscr S_{1}\mathscr A_{1}=(\Lambda_{1}\otimes_{1}1)[y^{\sharp1}D_{1}^{xx}+\der_{1}y^{x_{1}}E_{1}^{xxx}]
\end{equation*}
then
\begin{equation*}
\mathscr H=\Lambda_{1} [\der_{1} yC^{xx} + y^{x_{1}}\mu_{1}E_{1}^{xxx} + \mu_{1}(\Lambda_{1}\otimes_{1} 1)(y^{\sharp 1}D_{1}^{xx} + \der_{1} y^x_{1} E_{1}^{xxx})]
\end{equation*}
This last equation gives us the second formula when $x$ is a smooth sheet, in the general case when $x$ satisfies only the Hypothesis~\ref{hyp:chen-2d}. It's easy to see that all terms or we apply $\Lambda_{1}$ enjoy the regularity and thanks to the Lemma~\ref{lemma:mu-lambda-commute} it also satisfy the required algebraic conditions (i.e. : $\der_{1} [yC^{xx} + y^{x_{1}}\mu_{1}E_{1}^{xxx} + \mu_{1}(\Lambda_{1}\otimes_{1} 1)(y^{\sharp 1}D_{1}^{xx} + \der_{1} y^x_{1} E_{1}^{xxx})]=0$) and this concludes the proof.
\end{proof}

\subsubsection{Controlled Sheet}
\label{sec:controlled-sheets}

\begin{definition}
Let $x\in\CCC_{1,1}$ such that $\der x\in\CCC_{2,2}^{\alpha,\beta}$, and assume the algebraic hypothesis~\ref{hyp:chen-2d} to be true then we define the space of the two parameter controlled sheet $\mathscr K_{x}^{\alpha,\beta}$  by $y\in\mathscr K_{x}^{\alpha,\beta}$ if :
\begin{enumerate}
\item$\der y = -y^{x}A^x-y^{\omega}A^{\omega}+\sum_{a=1,2}(\int_{a} y^{x}\int_{\hat a}\dd x +\int_{a} y^{\omega}\int_{\hat a}\dd\omega)+y^\sharp$
\item$y,y^{x} ,y^{\omega}\in\mathscr Q_{x}^{\alpha,\beta}$
\end{enumerate}
\end{definition}

\begin{theorem}\label{th:integral}
For $\alpha ,\beta>1/3, x\in\CCC_{1,1}$ such that $\der x\in\CCC_{2,2}^{\alpha,\beta}$ and assume that the hypothesis~\eqref{hyp:chen-2d} is true then for $y\in\mathscr K_{x}^{\alpha,\beta}$ we define the increment  $\iint y \dd x\in \CCC_{2,2}^{\alpha,\beta}$ and $\iint y\dd\omega\in\CCC_{2,2}^{\alpha,\beta}$ by:
$$
\iint y dx := -yA^{x}+ y^{x}C^{xx}+y^{\omega}C^{\omega x}+\sum_{a= 1,2}\left(\int_{a} y \int_{\hat a}\dd x +\int_{a} y^{x}\int_{\hat a}\dd x\dd x +\int_{a} y^{\omega}\int_{\hat a} d\omega \dd x\right)+r^{\flat}
$$
where
$$
r^{\flat,\omega} =\Lambda \left[\der y A^{x} -\der \left(-y^{x}C^{xx}-y^{\omega} C^{\omega x} +\sum_{a = 1,2}(\int_{a} y^{x}\int_{\hat a}\dd x\dd x+\int_{a} y^{\omega} \int_{\hat a} d\omega \dd x\right)\right]
$$
and
$$
\iint y d\omega := -yA^{\omega}+ y^{x}C^{x\omega}+y^{\omega}C^{\omega \omega}+\sum_{a= 1,2}\left(\int_{a} y \int_{\hat a}\dd \omega+\int_{a} y^{x}\int_{\hat a}\dd x\dd \omega +\int_{a} y^{\omega}\int_{\hat a} d\omega \dd \omega\right)+r^{\flat,\omega}
$$
with
$$
r^{\flat} =\Lambda \left[\der y A^{\omega} -\der \left(-y^{x}C^{x\omega}-y^{\omega} C^{\omega \omega} +\sum_{a = 1,2}(\int_{a} y^{x}\int_{\hat a}\dd x\dd \omega+\int_{a} y^{\omega} \int_{\hat a} d\omega \dd \omega\right)\right]
$$
Theses two formula are well defined moreover if $x$ is a smooth sheet then this two definition coincide with that given by the Riemann-Stieltjes theory of integration.
\end{theorem}
\begin{proof}
Let $x$ a differentiable sheet then 
\begin{equation*}
\iint y \dd x=-y\der x+\sum_{a=1,2}\int_{a}y\int_{\hat a}\dd x+\iint \dd y \dd x .
\end{equation*}
Now using the fact that $y\in\mathscr K_{x}^{\alpha,\beta}$ we have 
\begin{equation*}
\iint \dd y\dd x=-y^{x}C^{xx}-y^{\omega}C^{\omega x}\sum_{a=1,2}\left(\int_{a}y^{x}\int_{\hat a}\dd x\dd x+\int_{a}y^{\omega}\int_{\hat a}\dd\omega \dd x\right)+\iint y^{\sharp}\dd x .
\end{equation*}
Finally if we remark that $\iint y^{\sharp}dx\in\CCC_{2,2}^{3\alpha,3\beta}$ we obtain 
\begin{equation*}
\iint y^{\sharp}dx:=\Lambda\left[\der\iint \dd y\dd x-\der\left(-y^{x}C^{xx}-y^{\omega}C^{\omega x}\sum_{a=1,2}(\int_{a}y^{x}\int_{\hat a}\dd x\dd x+\int_{a}y^{\omega}\int_{\hat a}\dd\omega \dd x)\right)\right]
\end{equation*}
This give us the formula when $x$ is smooth. Now we have to check that last formula have meaning in general case in other word we must show that we can apply $\Lambda$ for $r:=\der\iint \dd y\dd x-\der(-y^{x}C^{xx}-y^{\omega}C^{\omega x}\sum_{a=1,2}(\int_{a}y^{x}\int_{\hat a}\dd x\dd x+\int_{a}y^{\omega}\int_{\hat a}\dd\omega \dd x))$ , for this we will make some preliminary computation.
$$
\der\iint \dd y\dd x=\der y\der x
$$
$$
\der\left(y^{x}C^{xx}\right)=\der_{2}(-\der_{1}y^{x}C^{xx}+y^{x}\der_{1}C^{xx})=\der y^{x}C^{xx}+y^{x}\der x\der x-\der_{1}y^{x}\der_{2}C^{xx}-\der_{2}y^{x}\der_{1}C^{xx}
$$
\begin{equation*}
\begin{split}
\der\int_{a}d_{a}y^{x}\int_{\hat a}\dd x\dd x&=\der_{\hat a}\left(\der_{a}y^{x}C^{xx}+y^{xx_{a}}\mu_{a}E^{xxx}_{a}+(\Lambda_{a}\otimes_{a}1)\left[\der_{a}y^{xx_{a}}E_{a}^{xxx}+y^{x\sharp a}D_{a}^{xx}\right]\right)
\\&=-\der y^{x}C^{xx}+\der_{a}y^{x}\der_{\hat a}C^{xx}-\der_{\hat a}y^{xx_{a}}\mu_{a}E_{a}^{xxx}+y^{xx_{a}}(B_{a}^{xx}\der x+\mu_{a}F_{a}^{xxx})\\&+\mu_{a}(\Lambda_{a}\otimes_{a}1)[y^{x\sharp a}\der x\otimes_{a}\der x+ \der_{a}y^{xx_{a}}B_{a}^{xx}\otimes_{a}\der x-\der_{\hat a}y^{x\sharp a}D_{a}^{xx}+\der_{a}y^{xx_{a}}F_{a}^{xxx}
\\&-\der y^{xx_{a}}E_{a}^{xxx}]\\&=-\der y^{x}C^{xx}+\der_{a}y^{x}\der_{\hat a}C^{xx}-\der_{\hat a}y^{xx_{a}}\mu_{a}E_{a}^{xxx}+y^{xx_{a}}\mu_{a}F_{a}^{xxx}+(\int_{a}d_{a}y^{x}\int_{\hat a}dx)\der x\\&+(\Lambda_{a}\otimes_{a}1)\left[-\der_{\hat a}y^{x\sharp a}D_{a}^{xx}+\der_{a}y^{xx_{a}}F_{a}^{xxx}-\der y^{xx_{a}}E_{a}^{xxx}\right]
\end{split}
\end{equation*}
It is easy to see that  we have a similar equation if we replace $\dd x$ by $\dd\omega:=\dd_{1}x \dd_{2}x$. Finally we obtain 
\begin{equation*}
\begin{split}
r=&y^{\sharp}\der x+\der y^{x}C^{xx}+\der y^{\omega}C^{\omega x}
+\sum_{a=1,2}(\der_{\hat a}y^{xx_{a}}\mu_{a}E_{a}^{xxx}-y^{xx_{a}}\mu_{a}F_{a}^{xxx}+\der_{\hat a}y^{\omega x_{a}}\mu_{a}E_{a}^{x\omega x}
-y^{\omega x_{a}}\mu_{a}F_{a}^{x\omega x}\\&+(\Lambda_{a}\otimes_{a}1)[\der_{\hat a}y^{x\sharp a}D_{a}^{xx}-\der_{a}y^{xx_{a}}F_{a}^{xxx}
+\der y^{xx_{a}}E_{a}^{xxx}+\der_{\hat a}y^{\omega \sharp a}D_{a}^{\omega x}-\der_{a}y^{\omega x_{a}}F_{a}^{x\omega x}+\der y^{\omega x_{a}}E_{a}^{x\omega x}])
\end{split}
\end{equation*}
and this allow us to say that $r\in\CCC_{3,3}^{3\alpha,3\beta}$ which finishes the proof. 
\end{proof}

\begin{remark}
We observe that this definition of the two parameter integral is not consistent with the definition of the controlled sheet, indeed if $y\in\mathscr K_x^{\alpha,\beta}$ then the element $z\in \CCC_{1,1}$ defined by $z_{0t}=z_{s0}=0$ and $\der z=\iint y\dd x$ is not in general a controlled sheet.
\end{remark}

\subsection{Stability under mapping by regular functions}

\label{sec:stability}
 
In this section we show that $\varphi(x)\in \mathscr K^{\alpha,\beta}_{x}$ under more algebraic and geometric assumptions. To prove this result we will proceed by linear approximation the problem is that the terms which contain $\dd_1x\dd_2x$ does not approximate well to bypass this difficulty we will start by giving an alternative expression for the space $\mathscr K^{\alpha,\beta}_x$ 
\begin{hypothesis}
\label{hyp:algebric-assumption}
Let $\alpha , \beta >1/3$ and $a=1,2$.  Assume that there exist 
$$
G_{a}^{xx}\in\CC_{2,2}^{\alpha,\beta} ,\quad H_{a}^{xx}\in\CC_{2,2}^{\hat a\alpha,a\beta}, \quad I_{a}^{xx}\in\CC_{2,2}^{\hat a\alpha,a\beta} ,\quad J_{a}^{xx}\in\CC_{2,2}^{2\alpha,2\beta} ,\quad
$$
with the convention if $a=1$ then $\hat a=2$ and conversely. And we assume that this object satisfy the following relation:
\begin{enumerate}
\item $\der_{a}G_{a}^{xx}=\der_{a}I_{\hat a}^{xx}=0$
\item$\der_{a}H_{a}^{xx}=G_{a}^{xx}\der_{a}x$
\item$\der_{a} J_{a}^{xx}=I_{\hat a}^{xx}\der_{a}x$
\item $I_{a}^{xx}=B_{a}^{xx}+C^{xx}$
\item $A^{\omega}=1/2\der x^{2}-\iint xdx=G_{a}^{xx}-I_{a}^{xx}$
\item $B_{a}^{x\omega}=H_{a}^{xx}- J_{a}^{xx}$
\end{enumerate}
\end{hypothesis}
\begin{remark}
In regular case this last iterated integral are given by :
\begin{enumerate}
\item$(G_{1}^{xx})_{s_{1}s_{2}t_{1}t_{2}}=\int_{s_{1}}^{s_{2}}\der_{2}x_{st_{1}t_{2}}\dd_{s}x_{st_{2}}$
\item$(H_{1}^{xx})_{s_{1}s_{2}t_{1}t_{2}}=\int_{s_{1}}^{s_{2}}\der_{1}x_{s_{1}st_{1}}\der_{2}x_{st_{1}t_{2}}\dd_{s}x_{st_{2}}$
\item$(I_{2}^{xx})_{s_{1}s_{2}t_{1}t_{2}}=\iint_{(s_{1},t_{1})}^{(s_{2},t_{2})}\der_{2}x_{st_{1}t}\dd_{st}x_{st}$
\item$(J_{1}^{xx})_{s_{1}s_{2}t_{1}t_{2}}=\int_{s_{1}}^{s_{2}}\der_{1}x_{s_{1}st_{1}}\int_{t_{1}}^{t_{2}}\der_{2}x_{st_{1}t}\dd_{st}x_{st}$
\end{enumerate}
\end{remark}
Now under these assumption we give an alternative expression for the space $\mathscr K_{x}^{\alpha,\beta} $.
\begin{proposition}
Assume the Hypothesis~\ref{hyp:algebric-assumption} and~\ref{hyp:chen-2d} are true then  for $y\in\mathscr Q_{x}^{\alpha,\beta}$ we have : 
\begin{enumerate}
\item$\int_{a} y\der_{\hat a}x\dd_{a}x:=yG_{a}^{xx}+y^{xa}H_{a}^{xx}+\Lambda_{a}[y^{\sharp a}G_{a}^{xx}+\der_{a}y^{xa}H_{a}^{xx}]$
\item$\int_{a}y\int_{\hat a}\der_{\hat a}xdx:=yI_{\hat a}^{xx}+y^{xa} J_{a}^{xx}+\Lambda_{a}[y^{\sharp a}I_{\hat a}^{xx}+\der_{a}y^{xa}J_{a}^{xx}]$
\item$\int_{a}y\int_{\hat a}\dd\omega=\int_{a} y\der_{\hat a}xd_{a}x-\int_{a}y\int_{\hat a}\der_{\hat a}xdx$
\end{enumerate}
where 1 an 2 are well defined, moreover we can choose $G_a$, $H_a$ and $J_a$ such that the rough-integrals 1 and 2 coincide with their definition in the Riemann-Stieltjes case.
\end{proposition}
\begin{proof}
 We well only proof the first assertion (the proof of second assertion is similar). We assume that $x$ is smooth then we have :
$$
\left(\int_{1} y\der_{2}x\dd_{a}x\right)_{s_{1}s_{2}t_{1}t_{2}}:=\int_{s_{1}}^{s_{2}}y_{st_{1}}\der_{2}x_{st_{1}t_{2}}\dd_{s}x_{st}=y_{s_{1}t_{1}}(G_{1}^{xx})_{s_{1}s_{2}t_{1}t_{2}}+y^{x1}(H_{1}^{xx})_{s_{1}s_{2}t_{1}t_{2}}+r_{s_{1}s_{2}t_{1}t_{2}}
$$
where 
$$
r_{s_{1}s_{2}t_{1}t_{2}}:=\int_{s_{1}}^{s_{2}}y^{\sharp1}_{st_{1}t_{2}}\der_{2}x_{st_{1}t_{2}}\dd_{s}x_{st_{2}}
$$
Now is clear that $r\in\CC_{2,2}^{3\alpha,\beta}$ and :
$$
\der_{1}r=y^{\sharp1}G_{1}^{xx}+\der_{1}y^{x1}H_{1}^{xx}\in\CC_{3,2}^{3\alpha,\beta}
$$ 
and we get  
$$
r=\Lambda_{1}[y^{\sharp1}G_{1}^{xx}+\der_{1}y^{x1}H_{1}^{xx}].
$$
The proof of the last assertion is immediate consequence of the  assumption 5 and 6 of the hypothesis~\ref{hyp:algebric-assumption}. \end{proof}
This proposition allow us to say that $y\in \mathscr K_{x}^{\alpha,\beta}$ if and only if :
\begin{enumerate}
\item$y\in\mathscr Q_{x}^{\alpha,\beta}$
\item$\der y=-y^{x}\der x-y^{\omega}\left(1/2\der x^{2}-x\der x\right)+\sum_{a=1,2}\left(\int_{a}y^{x}\int_{\hat a}\dd x+\int_{a}y^{\omega}\der_{\hat a}x\dd_{a}x\right)+y^{\sharp}$
\end{enumerate}
where $
 y^{x},y^{\omega}\in\mathscr Q_{x}^{\alpha,\beta},\quad y^{\sharp}\in\CC_{2,2}^{\alpha,\beta}$
Now let us describe our strategy to prove that   $\varphi(x)\in \mathscr K_{x}^{\alpha,\beta}$. Let introduce the approximation:
$$
x^{1}_{st}=x_{0t}+s(x_{1t}-x_{0t})\quad,\quad x^{2}_{st}=x_{s0}+t(x_{s1}-x_{s0})
$$
$$
x^{12}_{st}=x_{00}+s(x_{10}-x_{00})+t(x_{01}-x_{00})+st\der x_{\square}
$$
where the $\square$ is the unit square. Now is clear if $\varphi\in C^{4}(\mathbb R,\mathbb R)$ then we have 
\begin{equation}
\label{eq:stab-first-exp}
\begin{split}
\der\varphi(x^{12})=&-\varphi'(x^{12})\der x-\varphi''(x^{12})(1/2\der(x^{12})^{2} -x^{12}\der x^{12})\\&+\sum_{a=1,2}\left(\int_{a}\varphi'(x^{12})\int_{\hat a}\dd x^{12}+\int_{a}\varphi''(x^{12})\der_{\hat a}x^{12}\dd_{a}x^{12}\right)+R^{12}
\end{split}
\end{equation}
where 
\begin{equation}
\label{rem-12}
\begin{split}
R^{12}&=\iint\varphi''(x^{12})\dd x^{12}\dd x^{12}+\iint\varphi'''(x^{12})\dd_1x^{12}\dd_2x^{12}\dd x^{12}+\iint\varphi'''(x^{12})\dd x^{12}\dd_1x^{12}\dd_2x^{12}
\\&+\iint\varphi^{(iv)}(x^{12})\dd_1x^{12}\dd_2x^{12}\dd_1x^{12}\dd_2x^{12}-\varphi''(x^{12})\iint\dd x^{12}\dd x^{12}-\sum_{a\in\{1,2\}}\int_a\dd_a\varphi''(x^{12})\int_{\hat a}\der_{\hat a}x^{12}\dd x^{12}
\end{split}
\end{equation}
Now our goal is to give similar formula for $x^{a}$ and to compare them with the same expansion for $x$. To do that we need to give a meaning for the boundary integral  appearing in the expansion of $x^a$ and for that we construct $B_{a}^{x^{\hat a}x^{\hat a}}$ , $G_{a}^{x^{\hat a}x^{\hat a}}$ and $H_{a}^{x^{\hat a}x^{\hat a}}$. But by formal computation we see that  
$$
(B_{2}^{x^{1}x^{1}})_{s_{1}s_{2}t_{1}t_{2}}=(s_{2}-s_{1})\left(\int_{t_{1}}^{t_{2}}\der_{2} x_{0t_{1}t}\dd_{t}\der_{1}x_{01t}+s_{1}\int_{t_{1}}^{t_{2}}\der x_{01t_{1}t}\dd_{t}\der_{1}x_{01t}\right)
$$
Then we define $B_2^{x^1x^1}$ by the following formula :
$$
(B_{2}^{x^{1}x^{1}})_{s_{1}s_{2}t_{1}t_{2}}:=(s_{2}-s_{1})\left(B_{2}^{xx}+s_{1}(K_{2}^{xx})_{01t_{1}t_{2}}\right)
$$
where
$$
(K_{2}^{xx})_{s_{1}s_{2}t_{1}t_{2}}:=\int_{t_{1}}^{t_{2}}\der x_{s_{1}s_{2}t_{1}t}\dd_{t}\der_{1}x_{s_{1}s_{2}t}
$$
$$ 
\der_{2}(K_{2}^{xx})_{s_{1}s_{2}t_{1}t_{2}t_{3}}=\der x_{s_{1}s_{2}t_{1}t_{2}}\der x_{s_{1}s_{2}t_{1}t_{3}}
$$
Similar computation allow us to define $(G_2^{x^1,x^1})$ and $H_2^{x^1x^1}$ in the following way :
$$
(G_{2}^{x^{1}x^{1}})_{s_{1}s_{2}t_{1}t_{2}}=(s_{2}-s_{1})((G_{2}^{xx})_{01t_{1}t_{2}}+(s_{2}-1)(L_{2}^{xx})_{01t_{1}t_{2}})
$$
$$
(H_{2}^{x^{1}x^{1}})_{s_{1}s_{2}t_{1}t_{2}}=(s_{2}-s_{1})(H_{2}^{xx})_{01t_{1}t_{2}}
+(s_{2}-1)(M_{2}^{xx})_{01t_{1}t_{2}}+s_{1}((N_{2}^{xx})_{01t_{1}t_{2}}+(s_{2}-1)(O_{2}^{xx})_{01t_{1}t_{2}}))
$$
where
$$
(L_{2}^{xx})_{s_{1}s_{2}t_{1}t_{2}}:=\int_{t_{1}}^{t_{2}}\der_{1}x_{s_{1}s_{2}t}\dd_{t}\der_{1}x_{s_{1}s_{2}t}
,\quad
(M_{2}^{xx})_{s_{1}s_{2}t_{1}t_{2}}:=\int_{t_{1}}^{t_{2}}\der x_{s_{1}s_{2}t_{1}t}\der_{1}x_{s_{1}s_{2}t}\dd_{t}x_{s_{2}t}
$$
and
$$
(O_{2}^{xxx})_{s_{1}s_{2}t_{1}t_{2}}:=\int_{t_{1}}^{t_{2}}\der x_{s_{1}s_{2}t_{1}t}\der_{1}x_{s_{1}s_{2}t}\dd_{t}\der_{1}x_{s_{1}s_{2}t}
$$
and this of course pushes us to give a more algebraic assumption on the sheet $x$:
\begin{hypothesis}
\label{hyp:objects-MLO}
Let $\alpha,\beta>1/3$,  $a=1,2$ and and assume the existence of :
$$
K_{a}^{xx},M_{a}^{xx}\in\CCC_{2,2}^{2\alpha,2\beta} ,\quad L_{a}^{xx}\in\CCC_{2,2}^{a\alpha,\hat a\beta},\quad O_{2}^{xxx}\in\CCC_{2,2}^{3\alpha,2\beta},\quad O_1^{xxx}\in\CCC_{2,2}^{2\alpha,3\beta}
$$
which satisfies the algebraic relation
\begin{enumerate}
\item$\der_{2}(K_{2}^{xx})_{s_{1}s_{2}t_{1}t_{2}t_{3}}=\der x_{s_{1}s_{2}t_{1}t_{2}}\der x_{s_{1}s_{2}t_{1}t_{3}}$
\item$\der_{2}L_{2}^{xx}=0$
\item$\der_{2}(M_{2}^{xx})_{s_{1}s_{2}t_{1}t_{2}t_{3}}=\der_2x_{s_{1}t_{1}t_{2}}(L_{2}^{xx})_{s_{1}s_{2}t_{2}t_{3}}$
\item$\der_{2}(N_{2}^{xx})_{s_{1}s_{2}t_{1}t_{2}t_{3}}=\der x_{s_{1}s_{2}t_{1}t_{2}}(G_{2}^{xx})_{s_{1}s_{2}t_{1}t_{2}}$
\item$\der_{2}(O_{2}^{xxx})_{s_{1}s_{2}t_{1}t_{2}}=\der x_{s_{1}s_{2}t_{1}t_{2}}(L_{2}^{xx})_{s_{1}s_{2}t_{2}t_{3}}$
\end{enumerate}
and same relation for $K_{1}^{xx}$, $M_{1}^{xx}$, $N_{1}^{xx}$ and $N_{1}^{xx}$.
\end{hypothesis}
\begin{remark}
Under this new hypothesis we have some fact :
\begin{enumerate}
\item$\der_{2} B^{x^{1}x^{1}}_{2}=\der x^{1}\der_{2}x^{2}$
\item$\der_{2} G^{x^{1}x^{1}}_{2}=0$
\item$\der_{2} H^{x^{1}x^{1}}_{2}=G^{x^{1}x^{1}}_{2}\der_{2}x^{2}$
\end{enumerate}
\end{remark}
This new hypothesis allows us to define
\begin{equation}
\label{eq:stab-1}
\small
R^1=\der\varphi(x^1)-\left(-\varphi'(x^{1})\der x^1-\varphi''(x^{1})(1/2\der(x^{1})^{2} -x^{1}\der x^{1})+\sum_{a=1,2}\left(\int_{a}\varphi'(x^{1})\int_{\hat a}\dd x^{1}+\int_{a}\varphi''(x^{1})\der_{\hat a}x^{1}\dd_{a}x^{1}\right)\right)
\end{equation}
and an analogue formula for $R^2$. Then we have the following relation between the remainder terms
\begin{proposition}\label{eq:rem}
Let $R^{1},R^{2}$ and $R^{12}$ given respectively by the equations~\eqref{eq:stab-1}and~\eqref{eq:stab-first-exp} and assume that hypothesis~\ref{hyp:objects-MLO} ,~\ref{hyp:algebric-assumption} and~\ref{hyp:chen-2d} are true then we have $R_{\square}=R^{1}_{\square}+R^{2}_{\square}-R^{12}_{\square}$ where 
$$
R=\varphi(x)-\left(-\varphi'(x)\der x-\varphi''(x)(1/2\der x^{2}-x\der x)+\sum_{a\in\{1,2\}}\int_a\varphi'(x)\int_{\hat a}\dd x+\int_a\varphi''(x)\der_{\hat a}x\dd_ax\right)
$$
\end{proposition}
\begin{proof}
The fact  that $x^1$and $x^2$ are respectively smooth in the first and second direction gives 
\begin{equation*}
\left(\int_{1}\varphi'(x^{1})\int_{2}\dd x^{1}\right)_{\square}=\left(\int_{1}\varphi'(x^{12})\int_{2}\dd x^{12}\right)_{\square}  ,\quad x^{l}_{ab}=x^{12}_{ab}
\end{equation*}
$$
\left(\int_2\varphi'(x^2)\int_1\dd x\right)_{\square}=\left(\int_2\varphi'(x^{12})\int_1\dd x^{12}\right)_{\square}
$$
For $(a,b)\in\{0,1\}^{2}$ and $l=1,2$. And of course similar equation for the boundary integrals given by $\int_1\varphi''(x^1)\der_2x^{1}\dd_2x^{1}$. On the other side we have by the definition of the functional $\Lambda_1$ that :
\begin{equation*}
\small
\begin{split}
\left(\int_{2}\varphi'(x)\int_{1}\dd x\right)_{\square}&=(1-\Lambda_{1}\der_1)(\varphi'(x)\der x+\varphi''(x)B_{2}^{xx})_{\square}=(1-\Lambda^{1d}\der^{1d})(\varphi'(x_{0.})\der x_{01..}+\varphi''(x_{0.})(B_{2}^{xx})_{01..})_{01}
\\&=(1-\Lambda^{1d}\der^{1d})(\varphi'(x_{0.}^{1})\der x_{01..}+\varphi''(x_{0.}^{1})(B_{2}^{x^{1}\tilde x^{1}})_{01..})_{01}
\\&=\left(\int_{2}\varphi'(x^{1})\int_{1}\dd x^{1}\right)_{\square}
\end{split}
\end{equation*}
and by similar argument we have also that 
\begin{equation*}
\left(\int_{2}\varphi''( x)\der_{1}x\dd_{2}x\right)_{\square}=\left(\int_{2}\varphi''(x^{1})\der_{1}x^{1}\dd_{2}\tilde x^{1}\right)_{\square}
\end{equation*}
Then putting these equation together we obtain the needed identity.
\end{proof}
Now to show that $R\in\CC_{2,2}^{2\alpha,2\beta}$ we have to gives a estimates for the remainder terms $R^{1}$ and $R^{2}$. At this point we give three technical lemma which help us to do this. 
\begin{lemma}\label{comp-shar-1}
Let $x\in\CCC_{2,2}$ and $\varphi\in C^{3}(\mathbb R)$, and we define $\nu_1(x)$ by :
\begin{equation*}
\begin{split}
\nu_1(x)=&-\varphi'(x_{00})\der x_{0101}-\varphi''(x_{00})\left(1/2\der x^{2}_{0101}-x_{00}\der x_{0101}\right)
\\&+\left(\int_{0}^{1}\varphi'(x_{00}+s\der_1x_{010})\dd s\right)\der x_{0101}
+\left(\int_{0}^{1}\varphi''(x_{00}+s\der_{1}x_{010})\dd s\right)\der_2x_{001}\der_1x_{011}
\\&+\left(\int_{0}^{1}\varphi''(x_{00}+s\der_1x_{010})s\dd s\right)\der x_{0101}\der_1x_{011}
\\&
-\left(\varphi'(x_{10})-\varphi'(x_{00})-\varphi''(x_{00})\der_1x_{010}\right)\der_{2}x_{101}
\end{split}
\end{equation*}
then the following inequality hold
$$
|\nu_1(x)|\lesssim\sup_{s\in[0,1]}|\varphi''(x_{s0})|(\der x_{0101})^{2}+\sup_{s\in[0,1]}|\varphi'''(x_{s0})||\der_1x_{010}\der_2x_{001}\der x_{0101}|
$$
\end{lemma}
\begin{proof}
We begin by remark that
\begin{equation*}
\begin{split}
(\int_{0}^{1}& \varphi''(x_{00}+s\der_{1}x_{010})s\dd s)\der x_{0101}\der_1x_{011}= 
\\ &
 \left(\int_{0}^{1}\varphi''(x_{00}+s\der_1x_{010})s\dd s\right)(\der x_{0101})^{2}
+\left(\int_{0}^{1}\varphi''(x_{00}+s\der_1x_{010}s\dd s\right)\der_1 x_{010})\der x_{0101}
\\ &
=\left(\int_{0}^{1}\varphi''(x_{00}+s\der_1x_{010})s\dd s\right)(\der x_{0101})^{2}+
\varphi'(x_{10})\der x_{0101}-\left(\int_{0}^{1}\varphi'(x_{00}+s\der_1x_{010})\dd s\right)\der x_{0101}
\end{split}
\end{equation*}
and
\begin{equation*}
\der x^{2}_{0101}=2x_{00}\der x_{0101}+2\der_1x_{010}\der_2x_{101}+2\der_2x_{001}\der x_{0101}+(\der x_{0101})^{2}
\end{equation*}
injecting these two equality in the definition of $\nu_1(x)$ gives :
\begin{equation*}
\begin{split}
\nu_1(x)&=-(\varphi'(x_{10})-\varphi'(x_{00}))\der_2x_{001}+(\int_{0}^{1}\varphi''(x_{00}+s\der_1x_{010})ds)\der_1x_{011}\der_2x_{001}-\varphi''(x_{00})\der_2x_{001}\der y_{0101}\\&
+(\int_{0}^{1}\varphi''(x_{00}+s\der_1x_{010})s\dd s)(\der x_{0101})^{2}-1/2\varphi''(x_{00})(\der y_{0101})^{2}
\\&=(\int_{0}^{1}\int_{0}^{1}\varphi'''(x_{00}+ss'\der_1x_{010})s\dd s'\dd s)\der_1x_{010}\der_2 x_{001}\der x_{0101}\\&
+(\int_{0}^{1}\varphi''(x_{00}+s\der_1x_{010})sds)(\der x_{0101})^{2}-1/2\varphi''(x_{00})(\der x_{0101})^{2}
\end{split}
\end{equation*}
Then the desired inequality is a simple consequence of this equality.
\end{proof}
\begin{lemma}
\label{stab lemma}
Let $y\in\CCC_{1,1}$ and $\alpha,\beta>1/3$ such that $\der_2 y\in\CCC_{1,2}^{\alpha,\beta}$ , $\der y\in\CCC^{\alpha,\beta}_{2,2}$ moreover we assume that $y$ is smooth in the first direction and that there exists $\int_2\int_2d_2yd_2y\in\CCC_{1,2}^{\alpha.2\beta}$,$\int_2\int \dd y\dd_2y\in\CCC_{2,2}^{\alpha,2\beta}$,$H_2^{yy}=\int_2\der_2y\der_1y\dd_2y\in\CCC_{2,2}^{1,2\beta} $,$ G_2^{yy}=\int_2\der_1y\dd_2y\in\CCC_{2,2}^{1,\beta}$ and $B_2^{yy}=\int_2\dd_2y\int_1dy\in\CCC_{2,2}^{1,2\beta}$ satisfying the  algebraic relation :
\begin{enumerate}
\item $\der_2G_2^{yy}=0$
\item $\der H_2^{yy}=G_2^{yy}\der_2y$
\item $\der_1\int_2\int_2\dd_2y\dd_2y=B_2^{yy}+\int_2\int \dd y\dd_2y$
\item $\der_2B_2^{yy}=\der_2y\der y$
\item $\der_2\int_2\dd_2y\dd_2y=\der_2y\der_2y$
\item $\der_2\int\int_2\dd y\dd_2y=\der y\der_2y$
\item $\der_1y\int_2\int_2\dd_2y\dd_2y=\der_1y\circ_1\int_2\int \dd y\dd_2y-\int_2\der_2y\circ_2\der y\dd_2y+H_2^{yy}$
\item$\der_1y\der_2y=G_2^{yy}-\int_2\int \dd y\dd_2y$.
\end{enumerate}
Then for $\varphi\in C^{5}(\mathbb R)$ the following equality holds :
\begin{equation*}
\der\varphi'(y)=\int_2\varphi'(y)\int_1\dd y+\int_2\varphi''(y)\der_1y\dd_2y+\varphi'(y)^{\sharp1}\der_2y+r_1(y)
\end{equation*}
where 
\begin{equation*}
\begin{split}
r_1(y)=&\varphi''(y)^{\sharp1}\int_2\int_2\dd_2y\dd_2y+\varphi'''(y)(\der_1y\circ_1\int_2\int dyd_2y)\\&
+\Lambda_2[(\der_2\varphi'(y)^{\sharp1}-\varphi''(y)^{\sharp1}\der_2y-\varphi'''(y)\der_1y\circ_1\der y)\der_2y+\der_2\varphi''(y)^{\sharp1}\int_2\int_2\dd_2y\dd_2y\\&
+\varphi'''(y)(\der y\circ_1\int_2\int \dd y\dd_2y)+\der_2\varphi'''(y)(\der_1y\circ_1\int_2\int\dd y\dd_2y]
\end{split}
\end{equation*}
\end{lemma}
\begin{proof}
By the one dimensional change of variable formula we have 
\begin{equation}
\der_2\varphi(y)=\int_2\varphi'(y)\dd_2y=(1-\Lambda_2\der_2)(\varphi(y)\der_2y+\varphi''(y)\int_2\int_2\dd_2y\dd_2y)
\end{equation}
Then if we apply $\der_1$ to this equation we get
\begin{equation}
\label{eq:exp-1-1}
\begin{split}
\der\varphi(y)&=(1-\Lambda_2\der_2)(\varphi'(y)\der y+\varphi''(y)B_2^{yy}+\der_1\varphi'(y)\der_2y+\der_1\varphi''(y)\int_2\int_2\dd_2y\dd_2y+\varphi''(y)\int_2\int \dd y\dd_2y)
\\&=\int_2\varphi'(y)\int_1\dd y+(1-\Lambda_2\der_2)(\der_1\varphi'(y)\der_2y+\der_1\varphi''(y)\int_2\int_2\dd_2y\dd_2y+\varphi''(y)\int_2\int \dd y \dd_2y)
\end{split}
\end{equation}
Expanding the two terms $\der_1\varphi'(y)\der_2y $, $\der_1\varphi''(y)\int_2\int_2\dd_2y\dd_2y$ and using the algebraic assumption gives   
\begin{equation}
\label{eq:exp-2}
\begin{split}
\der_1\varphi''(y)\der_2y+\varphi''(y)\int\int_2\dd y\dd_2y+\der_1\varphi''(y)\int_2\dd y\dd_2y&=\varphi''(y)G_2^{yy}+\varphi'''(y)H^{yy}_2-\varphi'''(y)\int_2\der_2y\circ_2\der y\dd_2y
\\&+(\varphi'(y)^{\sharp1}\der_2y
+\varphi''(y)^{\sharp1}\int_2\int_2\dd_2y\dd_2y+\\&\varphi'''(y)(\der_1y\circ_1\int_2\int \dd y\dd_2y)
\end{split}
\end{equation}
Now if we combine the fact that $(1-\Lambda_2\der_2)\int_2\der_2y\circ_2\der y\dd_2y=0$ (ie: $\int_2\der_2y\circ_2\der y\dd_2y\in\CCC_{2,2}^{1,3\beta}$ , $\der_2\int_2\der_2y\circ_2\der y\dd_2y\in\CCC_{2,3}^{1,3\beta}$) with equation~\eqref{eq:exp-1-1} and~\eqref{eq:exp-2} we obtain that
\begin{equation*}
\begin{split}
\der\varphi(y)&=\int_2\varphi'(y)\int_1\dd y+\int_2\varphi''(y)\der_1y\dd_2y+(1-\Lambda_2\der_2)(\varphi'(y)^{\sharp1}\der_2y
\\&+\varphi''(y)^{\sharp1}\int_2\int_2\dd_2y\dd_2y+\varphi'''(y)(\der_1y\circ_1\int_2\int \dd y\dd_2y) .
\end{split}
\end{equation*}
To obtain the needed result it suffice to expand the last term of this equality.
\end{proof}
\begin{lemma}
Let $y\in\CCC_{1,1}$  satisfying the assumption of lemma~\ref{stab lemma} and $\varphi\in C^{5}(\mathbb R)$ then we have the formula :
\begin{equation}
{\footnotesize
\begin{split}
&(\der_2\varphi'(y)^{\sharp1}-\varphi''(y)^{\sharp1}\der_2y-\varphi'''(y)\der_1y\circ_1\der y)_{s_1s_2t_1t_2}=\int_{[0,1]^{2}}k\varphi'''(y_{s_1t_1}+\overline k\der_1y_{s_1s_2t_1})\dd k\dd k'(\der y_{s_1s_2t_1t_2})^{2}
\\&+\int_{[0,1]^{2}}k\varphi^{(iv)}(y_{s_1t_1}+kk'\der_1y_{s_1s_2t_1})(\der_2y_{s_1s_2t_1}+kk'\der y_{s_1s_2t_1t_2})\dd k\dd k'(2\der_1y_{s_1s_2t_1}\der y_{s_1s_2t_1t_2}+(\der y_{s_1s_2t_1t_2})^2)
\\&+\int_{[0,1]^{3}}\tilde k\varphi^{(v)}(y_{s_1t_1}+\overline k\der_1y_{s_1s_2t_1}+k''\der_2y_{s_1t_1t_2}+\overline kk''\der y_{s_1s_2t_1t_2})(\der_2y_{s_1s_2t_1}+kk'\der y_{s_1s_2t_1t_2})^2\dd k\dd k' \dd k''(\der_1y_{s_1s_2t_2})^{2}
\end{split}
}
\end{equation}
where $\overline k=kk'$ and $\tilde k=k(1-k'')$ in particular this give us that $r_1(y)\in\CCC_{2,2}^{2\alpha,\beta}$ .
\end{lemma}
\begin{proof}
By the usual Taylor formula we have that 
\begin{equation}\label{cont:sha}
\begin{split}
\der_2\varphi'(y)^{\sharp1}_{s_1s_2t_1t_1}&=\int_{[0,1]^2}k(\varphi'''(y_{s_1t_2}+\overline k\der_1y_{s_1s_2t_2})-
\varphi'''(y_{s_1t_1}+\overline k\der_1y_{s_1s_2t_1}))dkdk'(\der_1y_{s_1s_2t_2})^{2}
\\&+2\int_{[0,1]^{2}}k\varphi'''(y_{s_1t_1}+\overline k\der_1y_{s_1s_2t_1})\dd k\dd k'\der_1y_{s_1s_2t_1}\der y_{s_1s_2t_1t_2}
\\&+\int_{[0,1]^{2}}k\varphi'''(y_{s_1t_1}+\overline k\der_1y_{s_1s_2t_1})\dd k\dd k'(\der y_{s_1s_2t_1t_2})^{2}
\end{split}
\end{equation}
Let denoted by $a_{s_1s_2t_1t_2}$ the first term in the r.h.s of this equation. Then if we remark that  $(\der_1y_{s_1s_2t_2})^2=(\der_1y_{s_1s_2t_1})^2+2\der_1y_{s_1s_2t_1}\der y_{s_1s_2t_1t_1}+(\der y_{s_1s_2t_1t_1})^2$ and use Taylor formula once again we obtain
\begin{equation}\label{cont:sha1}
\begin{split}
&a_{s_1t_1s_2t_2}=\int_{[0,1]^{2}}k\varphi^{(iv)}(y_{s_1t_1}+kk'\der_1y_{s_1s_2t_1})\dd k\dd k'\der_2y_{s_1t_1t_2}(\der_1y_{s_1s_2t_1})^{2}
\\&+\int_{[0,1]^{2}}k^2k'\varphi^{(iv)}(y_{s_1t_1}+kk'\der_1y_{s_1s_2t_1})\dd k\dd k'(\der_1y_{s_1s_2t_1})^{2}\der y_{s_1s_2t_1t_2}
\\&+\int_{[0,1]^{3}}\tilde k\varphi^{(v)}(y_{s_1t_1}+\overline k\der_1y_{s_1s_2t_1}+k''\der_2y_{s_1t_1t_2}+\overline kk''\der y_{s_1s_2t_1t_2})(\der_2y_{s_1s_2t_1}+kk'\der y_{s_1s_2t_1t_2})^2\dd k\dd k'\dd k''
\\&\times(\der_1y_{s_1s_2t_2})^{2}
\\&+\int_{[0,1]^{2}}k\varphi^{(iv)}(y_{s_1t_1}+kk'\der_1y_{s_1s_2t_1})(\der_2y_{s_1s_2t_1}+kk'\der y_{s_1s_2t_1t_2})\dd k\dd k'(2\der_1y_{s_1s_2t_1}\der y_{s_1s_2t_1t_2}+(\der y_{s_1s_2t_1t_2})^2)
\end{split}
\end{equation}
The two last terms in the r.h.s of this equation lie in the space $\CCC_{2,2}^{2\alpha,2\beta}$ then we will focus on the two first denoted respectively by $a^1_{s_1s_2t_1t_2}$ and $a^2_{s_1s_2t_1t_2}$. By integration by part formula we get 
$$
\int_{0}^{1}k^2k'\varphi^{(iv)}(y_{s_1t_1}+kk'\der_1y_{s_1s_2t_1})\dd k\der_1y_{s_1s_2t_1}=\varphi'''( y_{s_1t_1}+k'\der_1y_{s_1s_2t_1})-2\int_0^1k\varphi'''(y_{s_1t_1}+kk'\der_1y_{s_1s_2t_1})\dd k .
$$
Multiplying this equation by $\der_1y_{s_1s_2t_1}\der y_{s_1s_2t_1t_2}$ and integrating over $k'$ give us
\begin{equation}\label{cont:sha2}
\begin{split}
a^2_{s_1s_2t_1t_2}&=\int_{0}^{1}\varphi'''(y_{s_1t_1}+k'\der_1y_{s_1s_2t_1})\dd k'\der_1y_{s_1s_2t_1}\der y_{s_1s_2t_1t_2}
\\&-2\int_{[0,1]^{2}}k\varphi'''(y_{s_1t_1}+\overline k\der_1y_{s_1s_2t_1})\dd k\dd k'\der_1y_{s_1s_2t_1}\der y_{s_1s_2t_1t_2}
\end{split}
\end{equation}
on the other hand
\begin{equation}
\label{cont:sha3}
\begin{split}
a^1_{s_1s_2t_1t_2}&=\int_{[0,1]^2}\varphi^{(iv)}(y_{s_1t_1}+kk'\der_1y_{s_1s_2t_1})\dd k\dd k'(\der_1y_{s_1s_2t_1})^2\der_2y_{s_2t_1t_2}
\\&-\int_{[0,1]}\varphi'''(y_{s_1t_1}+k\der_1y_{s_1s_2t_1})\dd k\der_1y_{s_1s_2t_1}\der y_{s_1s_2t_1t_2}
\\&=(\varphi''(y)^{\sharp1}\der_2y)_{s_1s_2t_1t_2}-\int_{[0,1]}(\varphi'''(y_{s_1t_1}+k\der_1y_{s_1s_2t_1})-\varphi'''(y_{s_1t_1}))\dd k\der_1y_{s_1s_2t_1}\der y_{s_1s_2t_1t_2} .
\end{split}
\end{equation}
If we combine equations \eqref{cont:sha}~\eqref{cont:sha1}~\eqref{cont:sha2} and ~\eqref{cont:sha3} we obtain the needed result.
\end{proof}
Now what we have in mind is to take $y=x^1$ in this last two lemma but for this we need to construct $\int_2\int_2\dd_2x^1\dd_2x^1$ and $ \int_2\int \dd x^1\dd x^1$ then as usual we must add some algebraic conditions
\begin{hypothesis}
\label{final-alg-ass}
Let $\alpha,\beta>1/3$, $x\in\CCC_{1,1}$, $a=1,2$ and we assume the existence of
$P^{xx}_a=\int_a\int_ad_axd_ax\in\CCC_{1,2}^{*,2\beta}$ and $Q_a^{xx}=\int_a\int dxd_ax\in\CCC_{2,2}^{\alpha,2\beta}$ satisfying the algebraic relation 
\begin{enumerate}
\item$\der_aP_a^{xx}=\der_ax\der_ax$
\item$\der_{\hat a}P^{xx}_a=B_a^{xx}+Q_a^{xx}$
\item$\der_2Q_a^{xx}=\der x\der_ax$
\item$\der_{\hat a}xP^{xx}_a=\der_{\hat a}xQ_a^{xx}+H_a^{xx}-\int_a\der_ax\circ_a\der xd_ax$
\item$\der_{\hat a}x\der_ax=G_a^{xx}-Q_a^{xx}$
\end{enumerate}
where $\int_a\der_ax\circ_a\der x\dd_ax=\Lambda_a((\der_ax\circ_a\der x)\der_ax+\der_axQ_a^{xx}+\der xP_a^{xx}-\der x\circ_{\hat a}Q_a^{xx})$ and $B_a^{xx}$, $H_a^{xx}$ are the iterated integrals given respectively in the hypothesis~\ref{hyp:chen-2d} and~\ref{hyp:objects-MLO}. 
\end{hypothesis}
With this hypothesis we define 
$$
\left(\int_2\int_2\dd_2x^1\dd_2x^1\right)_{s_1s_2t_1t_2}=(P_2^{xx})_{0t_1t_2}+s((B_2^{xx})_{01t_1t_2}+(Q_2^{xx})_{01t_1t_2}+
(s-1)(K_2^{xx})_{01t_1t_2})
$$
and
$$
\left(\int_2\int \dd x^1\dd_2x^1\right)_{s_1s_2t_1t_2}=(s_2-s_1)((P_2^{xx})_{01t_1t_2}+(s_2-1)(K_2^{xx})_{01t_1t_2})
$$
Now all the ingredients are ready to prove our main result
\begin{theorem}
\label{th:stab}
Let $x\in\CCC_{1,1}^{\alpha,\beta}$ and $\varphi\in C^{5}(\mathbb R)$ such that hypothesis~\ref{hyp:chen-2d},~\ref{hyp:algebric-assumption},~\ref{hyp:objects-MLO} and~\ref{final-alg-ass} are satisfied  then $\varphi(x)\in\mathscr K_x^{\alpha,\beta}$ .
\end{theorem}
\begin{proof}
Let $R^{1}$ the remainder terms given by~\ref{eq:stab-1}  then if we put $y=x^{1}$ in the lemma \ref{stab lemma} we obtain 
$$
R^1=\int_2\varphi'(x^1)\int_1\dd x^1+\int_2\varphi''(x^1)\der_1x^1d_2x^1+\varphi'(x^1)^{\sharp1}\der_2x^{1}+r_1(x^1)
$$
Now if we observe that:
$$
(\varphi'(x^1)\der_2x^1)_{\square}=\left(\int_1\varphi'(x^1)\int_2\dd x^1+\int_1\varphi''(x^1)\der_2x^1dx^1\right)_{\square}-\nu_1(x)
$$
where $\nu_1(x)$ is given in the lemma~\ref{comp-shar-1} we get that $R^1_{\square}=r_1(x^1)_{\square}-\nu_1(x)_{\square}$ and $R^2_{\square}=(r_2(x^2)-\nu_2(x))_{\square}$ then by the proposition \label{eq:rem} we obtain that $R_{\square}=r_1(x^1)_{\square}-\nu_1(x)_{\square}+r_2(x^2)_{\square}-\nu_2(x)-R^{12}_{\square}$ of course this relation give us the needed regularity of $R$ on the unite square but if we take $X_{st}=x_{s_1+s(s_2-s_1);t_1+t(t_2-t_1)}$ for $(s,t)\in[0,1]^{2}$ then is easy to see that $X$ satisfy all the algebraic assumption and that $R^{x}_{s_1s_2t_1t_2}=R^{X}_{\square}$ which give us the result for any rectangle. 
\end{proof}
Now to simplify the notation we introduce the following definition.
\begin{definition}
Let $x\in \CCC_{1,1}$ a sheet satisfying the Hypothesis ~\ref{hyp:chen-2d}~,\ref{hyp:algebric-assumption}~,\ref{hyp:objects-MLO}  and~\ref{final-alg-ass} then we denote by $\mathbb X$ the collection of all iterated integrals giving in these Hypothesis and we call it Rough-Sheet associated to $x$ and then we define $\mathscr H_{\alpha,\beta}$ the space which contain the rough sheet as the product of the H\"older space giving in these Hypothesis equipped with the product topology.
\end{definition}
Now we have the following lemma:
\begin{lemma}\label{lemma:h-p-r}
Let $\rho_{1},\rho_{2}\in(0,1)$,  $x^1,x^2$ two increments lying in $\CCC_{1,1}^{\rho_1,\rho_2}$, and $\varphi\in C^{3}(\mathbb R)$.  Then we have:
\begin{equation}
||\der_1\varphi(x^1)||_{\rho_1,0}\lesssim \left(\sup_{s,t\in[0,1]^2}|\varphi(x^1_{st})|\right)||\der x^1||_{\rho_1,0}
\end{equation}
and 
\begin{equation}\label{eq:bnd-f(x1)-f(x2)}
\cn_{\rho_1,\rho_2}(\varphi(x^1)-\varphi(x^2))\lesssim 
c_{x^{1},x^{2}} \,\cn_{\rho_1,\rho_2}(x^1-x^2) \, 
\lc 1+\cn_{\rho_1,\rho_2}(x^1)+\cn_{\rho_1,\rho_2}(x^2) \rc^2
\end{equation}
where we recall that $\cn_{\rho_{1},\rho_{2}}(.)$ has been defined at equation~\eqref{eq:norm-def}. In the relation above we have also set
 $$
 c_{x^{1},x^{2}}=\sum_{i=1}^{3}\sup_{(s,t)\in[0,1]^{2}}|\varphi^{(i)}(x^1_{st})|+\sup_{(s,t)\in[0,1]^{2}}|\varphi^{(i)}(x^2_{st})|
 $$
 \end{lemma}

Using the concrete expression of the remainder term obtained previously, the continuity of the sewing map and this lemma we get easily the following continuity theorem. 
\begin{theorem}\label{th:cont-int}
Let $x\in\CCC_{1,1}^{\alpha,\beta}$ and $\overline x\in\CCC_{1,1}^{\alpha,\beta}$satisfying the hypothesis ~\eqref{hyp:chen-2d}~,\eqref{hyp:algebric-assumption}~,\eqref{hyp:objects-MLO}  and~\eqref{final-alg-ass} and $\varphi\in C^{8}(\mathbb R)$ then there exist a polynomial  function $K\in C([0,+\infty[,[0,+\infty[)$ such that 
$$
||\varphi(x)^{\sharp}-\varphi(\overline x)^{\sharp}||_{2\alpha,2\beta}\lesssim_{\alpha,\beta}CK(||\mathbb X||_{\mathscr H_{\alpha,\beta}}+||\overline{\mathbb X}||_{\mathscr H_{\alpha,\beta}})||\mathbb X-\overline{\mathbb X}||_{\mathscr H_{\alpha,\beta}}
$$
and then
\begin{equation*}
\begin{split}
\left|\left|\iint\varphi(x)\dd x-\iint\varphi(\overline x)\dd \overline x\right|\right|_{\alpha,\beta}+\left|\left|\iint\varphi(x)\dd \omega(x)-\iint\varphi(\overline x)\dd \omega(\overline x)\right|\right|_{\alpha,\beta}&\lesssim_{\alpha,\beta}CK(||\mathbb X||_{\mathscr H_{\alpha,\beta}}+||\overline{\mathbb X}||_{\mathscr H_{\alpha,\beta}})
\\&\times||\mathbb X-\overline{\mathbb X}||_{\mathscr H_{\alpha,\beta}}
\end{split}
\end{equation*}
where $\dd\omega(x):=\dd_1x\dd_2x$, $\omega(\overline x):=\dd_1\overline x\dd_2\overline x$ and $C=\sum_{k\in\{1,…,8\}}||\varphi^{(k)}||_{\infty,M}$ and $M=||x||_{\infty}+||\overline x||_{\infty}$
\end{theorem}

\section{Enhancement of the fractional Brownian Sheet and Stratonovich formula}
\label{sec:construction}
Let $(\Omega,\mathscr F,\mathbb P)$ a probability space, in this section we construct the rough-sheet associated to the fractional Brownian sheet $x$. Before staring with probabilistic computation let us recall the definition of such process.
\begin{definition}
The process $(x_{st})_{(s,t)\in[0,1]^2}$ defined on the probability space $(\Omega,\mathscr F,\mathbb P)$ is called fractional Brownian sheet with hurst parameter $\alpha,\beta\in[0,1]$   if $x$ is a Gaussian process with covariance function 
\begin{equation}\label{eq:def-cov-fBs}
R_{s_1s_2t_1t_2}=1/4(|s_1|^{\alpha}+|s_2|^{\alpha}+|s_2-s_1|^{\alpha})(|t_1|^\beta+|t_2|^{\beta}+|t_2-t_1|^\beta).
\end{equation}
\end{definition}
With this definition we recall the following harmonisable representation for the fractional Brownian sheet  
\begin{equation}\label{eq:harmonic-rep}
x_{st}\stackrel{law}{=}\int_{\mathbb R^2}\frac{e^{is\xi}-1}{|\xi|^{\alpha+1/2}}\frac{e^{it\eta}-1}{|\eta|^{\beta+1/2}}\hat W(\dd \xi,\dd \eta)
\end{equation}
where $\hat W$ is the Fourier transform of the white noise $W$ (see~\cite{ST94}). Let us now state some extension of Garsia-Rodemich-Rumsey Lemma (see~\cite{grr}) which will be useful to estimate the Hö\"older norm of random fields.
\begin{lemma}
\label{lemma:grr}
For $p >1$ and $\alpha,\beta\in(\frac{1}{3},\frac{1}{2}]$ there exist two  non negative constants $C_{1}=C_{1}(\alpha,\beta,p)$ and $C_{2}=C_{2}(\alpha,p)$ such as for every $y\in\CCC_{2}$ and $R\in\CC_{2}$   we have :
$$
 ||\delta y ||_{\alpha,\beta} \leq C_{1}U^{2}_{\alpha+\frac{2}{p},\beta+\frac{2}{p},p}(\der y)
$$
and
$$
||R||_{\alpha}\leq C_{2}(U^{1}_{\alpha+\frac{2}{p},p}(R)+||\der^{1d}R||_{\alpha})
$$
where
$$
U^{n}_{\alpha_{1},\alpha_{2},…,\alpha_{n},p}(V):=\left(\int_{[0,1]^{2n}}\frac{|V_{s_{1}^{1}s_{2}^{1}...s_{1}^{n}s_{2}^{n}}|^{p}}{\Pi_{i=1}^{i=n}|s_{2}^{i}-s_{1}^{i}|^{\alpha_{i}}}\dd s_{1}^{1}\dd s_{2}^{1}...\dd s_{1}^{n}\dd s_{2}^{n}\right)^{\frac{1}{p}}
$$
for $V\in\CC_{2}^{\otimes n}$
\end{lemma}
\begin{proof}
 Let $(s_{1},s_{2},t_{1} t_{2}) \in [0,1]^{4} $ then by  the Garsia-Rodemich-Rumsey  Lemma we have that  :
$$
|\delta y_{(s_{1},s_{2}),(t_{1},t_{2})}|^{p}=|\delta_{1}y_{(s_{1},s_{2})t_{2}}-\delta_{1} y_{(s_{1},s_{2})t_{1}}|^{p}\lesssim_{\alpha,\beta,p} |t_{2}-t_{2}|^{\beta p}\iint_{[0,1]^{2}}\frac{|\delta_{1}y_{(s_{1},s_{2})v_{2}}-y_{(s_{1},s_{2})v_{1}}|^{p}}{|v_{2}-v_{1}|^{\beta p+2}}\dd v_{1}\dd v_{2}
$$
Now we remark that  :
$$
|\delta y_{(s_{1},s_{2}),(t_{1},t_{2})}|=|\delta_{1}y_{(s_{1},s_{2})t_{2}}-\delta_{1}y_{(s_{1},s_{2})t_{1}}| =|\delta_{2}y_{s_{2}(t_{1},t_{2})}-\delta_{2} y_{s_{1}(t_{1},t_{2})}|
$$
Then if we apply the Garsia-Rodemich-Rumsey again we obtain:
$$
|\delta y_{(s_{1},s_{2}),(t_{1},t_{2})}|^{p}\lesssim_{\alpha,\beta,p} |s_{1}-s_{2}|^{\alpha p }|t_{2}-t_{2}|^{\beta p}\iiiint_{[0,1]^{4}}\frac{|\delta y_{(u_{1},u_{2})(v_{1},v_{2})}|^{p}}{|u_{1}-u_{2}|^{\alpha p + 2}|v_{2}-v_{2}|^{\beta p+2}}\dd u_{1}\dd u_{2}\dd v_{1}dv_{2} 
$$
For the proof of second inequality we refer the reader to~\cite{[Gubinelli-2004]}. 
\end{proof}
Now our strategy is to regularize $x$ in the following way:
$\forall N\in \mathbb N$ we put
\begin{equation}\label{reg-sheet}
x_{st}^{N} :=K_{\alpha,\beta}\iint_{\{|\xi|,|\eta|\leq N\}}\frac{e^{is\xi} -1}{|\xi|^{\alpha+1/2}}\frac{e^{it\eta}-1}{|\eta|^{\beta +1/2}}\hat W(\dd \xi,\dd \eta)
\end{equation}
so we are able to define
$$
\partial _{1}\partial_{2} x^{N}_{st} :=K_{\alpha,\beta}\iint_{\{|\xi| ,|\eta|\leq N\} }\frac{i\xi e^{is\xi}}{|\xi|^{\alpha + 1/2}}\frac{i\eta e^{is\eta}}{|\eta|^{\beta + 1/2}}\hat W(\dd \xi,\dd \eta)
$$
And this allows us to construct rough sheet associated to $x^{N}$ denoted in the following by $\mathbb X^N$
Now we will set out the main theorem of this section which will allow us to say that the fractional Brownian sheet can be enhanced in a rough sheet.
\begin{theorem}
\label{conv-th}
Let $x^{N}$ the process given by the equation ~\eqref{reg-sheet} and $\mathbb X^{N}$ the associated rough sheet then there exists a random variable $\mathbb X\in \mathscr H_{h,h'}$ such that $\mathbb X^{N}$ converges to $\mathbb X$ in $L^{p}(\Omega,\mathscr H_{h,h'})$  for all $(h,h',p)\in(\frac{1}{3},\alpha)\times(\frac{1}{3},\beta)\times[1,+\infty)$.
\end{theorem}
To prove the theorem~\ref{conv-th} we will need the following lemma
\begin{lemma}
\label{lemma:basic-bound-integral}
Let $\alpha>1/3$ and the function defined on $\mathbb R^{2}$ by :
$$
\mathscr Q(\xi,\eta):=\int _{0}^{1}\dd se^{is\xi}\int_{0}^{s}\dd ve^{iv\eta}
$$
then $\mathscr Q$ satisfies 
\begin{enumerate}
\item $\mathscr Q(\xi,-\xi)=\frac{(1-cos(\xi))+i(\xi-sin(\xi))}{|\xi|^{2}}.$
\item $\mathscr Q(\xi_{1},\xi_{2})\lesssim\frac{1}{|x_{i}|}  .$
\item $\mathscr Q(\xi_{1},\xi_{2})\lesssim \frac{1}{|x_{1}||x_{2}|}+\frac{1}{|x_{1}+x_{2}||x_{i}|}.$
\item $\mathscr Q(\xi_{1},\xi_{2})\lesssim1.$
\item $\iint_{\mathbb R^{2}}\frac{|\mathscr Q(\xi,\eta)|^{2}}{|\xi|^{2\alpha-1}|\eta|^{2\alpha-1}}\dd \xi\dd \eta<+\infty $ 
\end{enumerate}
where $i\in\{1,2\}$.
\end{lemma}
\begin{proof}
The properties 1,2,3,4 are easy to establish by a direct computation only the prove of the last assertion claim a bit more work. Indeed we begin by decomposing the plane in three region $D,U$ and $V$ given by  :
\begin{enumerate}
\item $D=\left\{(\xi,\eta)\in\mathbb R^{2} ; |\xi+\eta|\leq \frac{\min(|\xi|,|\eta|)}{2}\right\}$
\item $U=\left\{(\xi,\eta)\in\mathbb R^{2} ;|\xi+\eta|\geq\frac{\max(|\xi|,|\eta|)}{2}\right\}$
\item $V=(D\cup V)^{c}$
\end{enumerate}
Now when $(\xi,\eta)\in D$ we have that $2/3|\xi|\leq |\eta|\leq 3/2|\xi|$ which leads us by the third property given in Lemma to obtain the following bound :
$$
|\mathscr Q(\xi,\eta)|\lesssim\frac{1}{|\xi||\xi+\eta|}
$$
and then
\begin{equation*}
\begin{split}
\iint_{D\cap\{|\xi+\eta|\geq1\}} \frac{|\mathscr Q(\xi,\eta)|^{2}}{|\xi|^{2\alpha-1}|\eta|^{2\alpha-1}}\dd \xi\dd \eta\lesssim & \iint_{|\{\xi+\eta|\geq1 ;|\xi|\geq2|\xi+\eta|\}}\frac{1}{|\xi|^{4\alpha}|\xi+\eta|^{2}}\dd \xi\dd \eta  \\&\lesssim \int_{1}^{+\infty}\dd v \frac{1}{v^{2}}\int_{2v}^{+\infty}\dd u\frac{1}{u^{4\alpha}}<+\infty
\end{split}
\end{equation*}
Now if  $|\xi+\eta|\leq1$ we can estimate the integrand in the following way :
\begin{equation*}
\frac{|\mathscr Q(\xi,\eta)|^{2}}{|\xi|^{2\alpha-1}|\eta|^{2\alpha-1}}\lesssim\frac{1}{|\xi|^{4\alpha-2\gamma}}
\end{equation*}
where $\gamma\in[0,1]$ and then we get :
$$
\iint_{D\cap\{|\xi+\eta|\leq 1\}} \frac{|\mathscr Q(\xi,\eta)|^{2}}{|\xi|^{2\alpha-1}|\eta|^{2\alpha-1}}\dd \xi\dd \eta\lesssim \int_{0}^{1}\dd v\int_{2v}^{+\infty}\dd u\frac{1}{u^{4\alpha-2\gamma}}<+\infty
$$
as soon as $\gamma \in (2\alpha-1, 2\alpha-1/2)$ which shows that the integral is finite on $ D $.
Now on $U$ and $V$ we have that $|\xi|,|\eta|\lesssim |\xi+\eta|$ and hence we can estimate $\mathscr Q$ by:
$$
|\mathscr Q(\xi,\eta)|\lesssim\frac{1}{|\xi||\eta|}
$$
then we obtain
$$
\iint_{(U\cup V)\cap\{|\xi|>1,|\eta|>1\}}\frac{|\mathscr Q(\xi,\eta)|^{2}}{|\xi|^{2\alpha-1}|\eta|^{2\alpha-1}}\dd \xi\dd \eta\lesssim\iint_{[1,+\infty)^{2}}|\xi\eta|^{-4\alpha-1}\dd \xi\dd \eta<+\infty
$$
 In the region of  $U$ and  $V$ where $|\xi|,|\eta|\leq1$ we bound the integrand in the following manner.
$$
\frac{|\mathscr Q(\xi,\eta)|^{2}}{|\xi|^{2\alpha-1}|\eta|^{2\alpha-1}}\lesssim|\xi\eta|^{1-2\alpha}
$$
then
$$
\iint_{(U\cup V)\cap\{|\xi|\leq1,|\eta|\leq1\}}\frac{|\mathscr Q(\xi,\eta)|^{2}}{|\xi|^{2\alpha-1}|\eta|^{2\alpha-1}}\dd \xi\dd \eta\lesssim(\int_{\{|\xi|\leq1\}}|\xi|^{1-2\alpha}\dd \xi)^{2}
$$
The same bound for the region $(U\cup V)\cap\{|y|\leq1,|x|\geq1\}$ combined with the fact that $\mathscr Q(\xi,\eta)\lesssim \min(|\xi|^{-1},|\eta|^{-1})$ gives 
$$
\frac{|\mathscr Q(\xi,\eta)|^{2}}{|\xi|^{2\alpha-1}|\eta|^{2\alpha-1}}\lesssim|\xi|^{-1-2\alpha}|\eta|^{-2\alpha+1}
$$
This shows that our kernel is integrable on $(U\cup V)\cap\{|\xi|\leq1,|\eta|\geq1\}$ and by symmetry we obtain the integrability in the remaining area which completes the proof.
\end{proof}

\subsection{Proof of theorem~\eqref{conv-th}}
\begin{proof}
We will decompose the proof of the theorem in two step. In a first step we give the bound for the rough sheet $\mathbb X^{N}$ in $L^{2}(\Omega)$ for fixed parameters and in the second step we will use a variant of  Garsia-Rodemich-Rumsey inequality to prove that the sequence $(\mathbb X^{N})_N$ is a Cauchy sequence in
$\mathbb L^{p}(\Omega,\mathscr H_{h,h'})$.
\paragraph{Step 1: Estimation.}
Let  $A^{NMx}:=A^{Nx}-A^{Mx}=\der (x^{N}-x^{M})$ for $M\leq N$ and similar notation for all other terms of the rough sheet. Now is not difficult to see that 
\begin{equation*}
\begin{split}
\expect [|A^{NMx}_{(s_{1},s_{2})(t_{1},t_{2})}|^{2}] &=\int_{\{||(\xi,\eta)||_{\infty}\in[M,N]\}}\frac{|e^{is_{2}\xi}-e^{is_{1}\xi}|^{2}}{|\xi|^{1+2\alpha}}\frac{|e^{it_{2}\eta}-e^{it_{1}\eta}|^{2}}{|\eta|^{1+2\beta}}\dd \xi\dd \eta
\\&\lesssim(s_2-s_1)^{2\alpha}(t_{2}-t_{1})^{2\beta}(I_{A}^{M})_{(s_1,s_2)(t_1,t_2)}
\end{split}
\end{equation*}
where
\begin{equation*}
(I_{A}^{M})_{(s_{1},s_{2})(t_{1},t_{2})}=\int_{\{||((t_{2}-t_{1})x,(s_{2}-s_{1})y)||_{\infty}\geq M(s_{2}-s_{1})(t_{2}-t_{1})\}}|x|^{-1-2\alpha}|e^{ix}-1|^{2}|y|^{-1-2\beta}|e^{iy}-1|^{2}\dd x\dd y
\end{equation*}
And let us remark that for $(\alpha,\beta)\in(1/3,1/2]$ we have : 
\begin{equation*}
\int_{\mathbb R^{2}}|\xi|^{-1-2\alpha}|e^{i\xi}-1|^{2}|\eta|^{-1-2\beta}|e^{i\eta}-1|\dd \xi\dd \eta<+\infty
\end{equation*}
these imply 
$$
\sup_{M\in\mathbb N,(s_{1},s_{2},t_{1},t_{2})\in[0,1]^{4}}|(I_{A}^{M})_{(s_{1},s_{2})(t_{1},t_{2})}|<+\infty
$$ 
and 
$$
\lim_{M\to\infty}(I_{A}^{M})_{(s_{1},s_{2})(t_{1},t_{2})}=0  
$$
for $s_{1}\ne s_{2}$ and $t_{1}\ne t_{2}$. Now in what follows we prove similar  bound  for  the other component of the rough sheet. By Wick theorem
\begin{equation}\label{eq:1-first-Wick-sum}
\begin{split}
\mathbb E[|C^{Nxx}_{(s_{1},s_{2})(t_{1},t_{2})}|^{2}] &=|\int_{\mathbb R^{2}}( K^{NM}_{(s_{1},s_{2})(t_{1},t_{2})})(\xi,-\xi)\dd \xi|^{2}+\int_{\mathbb R^{4}}|( K^{NM}_{(s_{1},s_{2})(t_{1},t_{2})})(\xi,\eta)|^{2}\dd \xi\dd \eta
\\&+|\int_{\mathbb R^{4}}( K^{NM}_{(s_{1},s_{2})(t_{1},t_{2})})(\xi,\eta)\overline {( K^{NM}_{(s_{1},s_{2})(t_{1},t_{2})})}(\eta,\xi)\dd \xi\dd \eta|
\end{split}
\end{equation}
where
\begin{equation*}
( K^{NM}_{(s_{1},s_{2})(t_{1},t_{2})})(\xi,\eta):=\frac{i\xi i\eta}{|\xi|^{\alpha+1/2}|\eta|^{\alpha+1/2}}\mathscr Q_{s_{1}s_{2}}(\xi,\eta)\frac{i\xi'i\eta'}{|\xi'|^{\beta+1/2}|\eta'|^{\beta+1/2}}\mathscr Q_{t_{1}t_{2}}(\xi',\eta')\chi_{\{||(\xi,\xi',\eta,\eta')||_{\infty}\in[M,N]\}}
\end{equation*}
for $\xi:=(\xi,\xi')\in\mathbb R^{2}$ et $\eta=(\eta,\eta')\in\mathbb R^{2}$ and $\mathscr Q_{s_{1}s_{2}}(\xi,\eta):=e^{is_{1}(\xi+\eta)}(s_{2}-s_{1})^{2}\mathscr Q((s_{2}-s_{1})\xi,(s_{2}-s_{1})\eta)$ with $\mathscr Q$ is the function defined in the Lemma~\eqref{lemma:basic-bound-integral}, which gives us by change of variable formula that
\begin{equation*}
\int_{\mathbb R^{4}}|( K^{NM}_{(s_{1},s_{2})(t_{1},t_{2})})(\xi,\eta)|^{2}\dd\xi \dd\eta\lesssim (s_{2}-s_{1})^{4\alpha}(t_{2}-t_{1})^{4\beta}(I^{1,M}_{C})_{(s_{1},s_{2})(t_{1},t_{2})}
\end{equation*}
where
\begin{equation*}
(I^{1,M}_{C})_{(s_{1},s_{2})(t_{1},t_{2})}:=\int_{\mathbb R^{4}} \chi_{\{||(s_{2}-s_{1})\xi,(s_{2}-s_{1})\eta,(t_{2}-t_{1})\xi',(t_{2}-t_{1})\eta')||_{\infty}\geq M\}}\frac{|\mathscr Q(\xi,\eta)|^{2}}{|\xi\eta|^{2\alpha-1}}\frac{|\mathscr Q(\xi',\eta')|^{2}}{|\xi'\eta'|^{2\beta-1}}\dd \xi\dd \eta
\end{equation*}
Now the Lemma~\ref{lemma:basic-bound-integral} gives :
$$
\int_{\mathbb R^{4}}\frac{|\mathscr Q(\xi,\eta)|^{2}}{|\xi\eta|^{2\alpha-1}}\frac{|\mathscr Q(\xi',\eta')|^{2}}{|\xi'\eta'|^{2\beta-1}}\dd \xi\dd \eta<+\infty
$$
and then 
\begin{equation*}
\sup_{M\in\mathbb N,(s_1,s_2,t_1,t_2)\in[0,1]^{4}}(I^{1,M}_{C})_{(s_{1},s_{2})(t_{1},t_{2})}<+\infty
\end{equation*}
and
\begin{equation*}
\lim_{M\to\infty}(I^{1,M}_{C})_{(s_{1},s_{2})(t_{1},t_{2})}=0
\end{equation*}
for $s_2\ne s_1$ and $t_2\ne t_1$ and by Cauchy-Schwartz inequality we have :
$$
|\int_{\mathbb R^{4}}( K^{NM}_{(s_{1},s_{2})(t_{1},t_{2})})(\xi,\eta)\overline {( K^{NM}_{(s_{1},s_{2})(t_{1},t_{2})})}(\eta,\xi)\dd \xi\dd \eta|\leq\int_{\mathbb R^{4}}|( K^{NM}_{(s_{1},s_{2})(t_{1},t_{2})})(\xi,\eta)|^{2}\dd \xi\dd \eta
$$
Then is remind to bound the first term appearing in the sum of the equation~\eqref{eq:1-first-Wick-sum}, indeed 
$$
|\int_{\mathbb R^{2}}( K^{NM}_{(s_{1},s_{2})(t_{1},t_{2})})(\xi,-\xi)\dd \xi|^{2}\lesssim(s_2-s_1)^{2\alpha}(t_2-t_1)^{2\beta}(I_C^{2,M})_{(s_1,s_2)(t_1,t_2)}
$$
where 
$$
(I_C^{2,M})_{(s_1,s_2)(t_1,t_2)}=\int_{\mathbb R^{2}}\chi_{\{||(s_{2}-s_{1})\xi,(t_{2}-t_{1})\xi'||_{\infty}\geq M\}}|\xi|^{-1-2\alpha}|\xi'|^{1-2\beta}(1-\cos(\xi))(1-\cos(\xi'))\dd \xi\dd \xi'
$$
then is no difficult to see that $I_C^{2,M}$ satisfies the same property that $I_C^{1,M}$.Now if we put 
$I_C^{M}:=I_C^{2,M}+I_C^{1,M}$ is easy to see that :
$$
\mathbb E[|C^{NMxx}_{(s_{1},s_{2})(t_{1},t_{2})}|^{2}]\lesssim(s_2-s_1)^{2\alpha}(t_2-t_1)^{2\beta}(I_C^{M})_{(s_1,s_2)(t_1,t_2)}
$$
when $(I_C^{M})$ satisfy :
$$
\sup_{M\in\mathbb N,(s_1,s_2,t_1,t_2)\in[0,1]^{4}}(I^{M}_{C})_{(s_{1},s_{2})(t_{1},t_{2})}<+\infty
$$
$$
\lim_{M\to\infty}(I^{M}_{C})_{(s_{1},s_{2})(t_{1},t_{2})}=0
$$
for $s_2\ne s_1$ and $t_2\ne t_1$. All the other term of $\mathbb X^N$ can be estimate by the same argument and  satisfy the same type of bound to see these let us
treat a more complex term of the rough sheet.
$$
\mathbb E[|(E_1^{NMxxx})_{(s_1,s_2,s_3,s_3)(t_1,t_2)}|^{2}]=\sum_{l=1}^{15}\mathscr I_{l}
$$
where $(\mathscr I_l)_l$ is the different Wick contraction given by:
\begin{enumerate}
\item $\mathscr I_{1}=\int_{\mathbb R^{6}}|( G^{NM}_{(s_{1},s_{2},s_{3},s_{4})(t_{1},t_{2})})(a,b,c)|^{2}\dd a\dd b\dd c$
\item $\mathscr I_{2}=\int_{\mathbb R^{2}}|\int_{\mathbb R^{2}}( G^{NM}_{(s_{1},s_{2},s_{3},s_{4})(t_{1},t_{2})})(a,b,-a)\dd a|^{2}\dd b$
\item $\mathscr I_{3}=\int_{\mathbb R^{2}}|\int_{\mathbb R^{2}}( G^{NM}_{(s_{1},s_{2},s_{3},s_{4})(t_{1},t_{2})})(a,-a,b)\dd a|^{2}\dd b$
\item $\mathscr I_{4}=\int_{\mathbb R^{2}}|\int_{\mathbb R^{2}}( G^{NM}_{(s_{1},s_{2},s_{3},s_{4})(t_{1},t_{2})})(a,b,-b)\dd b|^{2}\dd a$
\item $\mathscr I_{5}=\int_{\mathbb R^{6}}( G^{NM}_{(s_{1},s_{2},s_{3},s_{4})(t_{1},t_{2})})(a,b,-a)\overline {( G^{NM}_{(s_{1},s_{2},s_{3},s_{4})(t_{1},t_{2})})}(b,c,-c) \dd a\dd b\dd c$
\item $\mathscr I_{6}=\int_{\mathbb R^{6}}( G^{NM}_{(s_{1},s_{2},s_{3},s_{4})(t_{1},t_{2})})(a,b,-a)\overline {( G^{NM}_{(s_{1},s_{2},s_{3},s_{4})(t_{1},t_{2})})}(c,-c,b) \dd a\dd b\dd c$
\item $\mathscr I_{7}=\int_{\mathbb R^{6}}( G^{NM}_{(s_{1},s_{2},s_{3},s_{4})(t_{1},t_{2})})(a,-a,b)\overline {( G^{NM}_{(s_{1},s_{2},s_{3},s_{4})(t_{1},t_{2})})}(a,c,-c) \dd a\dd b\dd c$
\item $\mathscr I_{8}=\overline{\mathscr I_{5}}$, $\mathscr I_{9}=\overline{\mathscr I_{6}}$,  $\mathscr I_{10}=\overline{\mathscr I_{7}}$
\item $\mathscr I_{k}=\int_{\mathbb R^{6}}( G^{NM}_{(s_{1},s_{2},s_{3},s_{4})(t_{1},t_{2})})(a,b,c)\overline {( G^{NM}_{(s_{1},s_{2},s_{3},s_{4})(t_{1},t_{2})})}(\sigma_{k}(a),\sigma_{k}(b),\sigma_{k}) \dd a\dd b\dd c$
\end{enumerate}
where $\sigma_{k}$ , $k\in\{11,12,13,14,15\}$ is a permutation of  three elements  
different of identity and for $a=(a,a') , b=(b,b'),c=(c,c')\in \mathbb R^{2}$ the kernel is given by :
\begin{equation*}
\begin{split}
( G^{NM}_{(s_{1},s_{2},s_{3},s_{4})(t_{1},t_{2})})(a,b,c):=&\left(\frac{iaib}{|a|^{\alpha+1/2}|b|^{|\alpha+1/2}}\mathscr Q_{s_{1}s_{2}}(a,b)\right)\left(\frac{ia'ic'}{|b'|^{\beta+1/2}|z'|^{\beta+1/2}}\mathscr Q_{t_{1}t_{2}}(b',c')\right)
\\&\times\left(\frac{e^{is_{4}c}-e^{is_{3}c}}{|c|^{\alpha+1/2}}\right)\left(\frac{e^{it_{1}a'}-1}{|a'|^{\beta+1/2}}\right)\chi_{\{||(a,a',b,b',c,c')||_\infty\in[M,N]\}}
\end{split}
\end{equation*}
The study of theses integral is confined to study the first four term in fact if $k\in\{11,12,13,14,15\}$ then  by Cauchy-Schwartz we have $|\mathscr I_{k}|\leq \mathscr I_{1}$ , $|\mathscr I_{5}|^{2}\leq \mathscr I_{2}\mathscr I_{4}$ , $|\mathscr I_{6}|^{2}\leq \mathscr I_{2}\mathscr I_{3}$ and
 $|\mathscr I_{7}|^{2}\leq \mathscr I_{3}\mathscr I_{4}$. By variable change formula we have that 
\begin{equation*}
|\mathscr I_1|\lesssim(s_2-s_1)^{4\alpha}(s_4-s_3)^{2\alpha}(t_2-t_1)^{4\beta}t_1^{2\beta}(I_{E_1}^{M})_{(s_1,s_2,s_3,s_4)(t_1,t_2)}
\end{equation*}
where
$$
(I_{E_1}^{M})_{(s_1,s_2,s_3,s_4)(t_1,t_2)}:=\int_{\mathbb R^6}\chi_{\mathscr D^{M}_{s_1s_2s_3s_4t_1t_2}}\frac{|\mathscr Q(a,b)|^{2}}{|ab|^{2\alpha-1}}\frac{|\mathscr Q(b',c')|^{2}}{|a'b'|^{2\beta-1}}\frac{|e^{ic}-1|^{2}}{|c|^{2\alpha+1}}\frac{|e^{ic}-1|^{2}}{|c|^{2\beta+1}}\dd a\dd a'\dd b\dd b'\dd c\dd c'
$$
where $\mathscr X$ is the Indicator function and $\mathscr D^{M}_{s_1s_2s_3s_4t_1t_2}$ is decreasing to the empty set when M tend to infinity for $s_1\ne s_2$,$t_1\ne t_2$ and $s_3\ne s_4$
these statement with the fact that :
$$
\int_{\mathbb R^{6}}\frac{|\mathscr Q(a,b)|^{2}}{|ab|^{2\alpha-1}}\frac{|\mathscr Q(b',c')|^{2}}{|a'b'|^{2\beta-1}}\frac{|e^{ic}-1|^{2}}{|c|^{2\alpha-1}}
\frac{|e^{iz}-1|^{2}}{|z|^{2\beta-1}}\dd a\dd a'\dd b\dd b'\dd c\dd c'<+\infty
$$
give us that 
$$
\sup_{M\in \mathbb N,(s_1,s_2,s_3,s_4,t_1,t_2)\in[0,1]^{6}}(I_{E_1}^{M})_{(s_1,s_2,s_3,s_4)(t_1,t_2)}<+\infty
$$
and 
$$
\lim_{M\to+\infty}(I_{E_1}^{M})_{(s_1,s_2,s_3,s_4)(t_1,t_2)}=0
$$
Now by Cauchy-Schwartz we have :
\begin{equation*}
\begin{split}
\mathscr I_2\lesssim&\left(\int_{\mathbb R^4}|( K^{NM}_{(s_1,s_2)(t_1,t_2)}\right)(a,b)|^2\dd a\dd b)\int_{\mathbb R}|a|^{-1-2\alpha}|e^{-is_4a}-e^{-is_3a}|^2\dd a\\&
\left(\int_{\mathbb R}|a'|^{-1-2\beta}|e^{it_1a'}-1|^2\dd a'\right)
\end{split}
\end{equation*}
all these integrals have already studied and then we obtain the good estimate for $\mathscr I_2$. The estimation of $\mathscr I_3$ and $\mathscr I_4$ is obtained by the same technique.

\paragraph{Step 2: Convergence of $\mathbb X^{N}$.}
 We prove some Garsia-Rodemich-Rumsey inequalities for our objects and then we obtain
the convergence  of the sheet $X^{N}$. To give an idea we will first show the convergence of our first term in fact by the Lemma~\ref{lemma:grr} and the Gaussian hypercontractivity we have that :
\begin{equation*}
\begin{split}
\mathbb E[||A^{NMx}||^{p}_{h,h'}]&\lesssim_{h,h',p,\alpha,\beta}\iint_{[0,1]^{4}}\frac{\mathbb E[|A^{NMx}_{s_1s_2t_1t_2}|^{p}]}{|s_2-s_1|^{hp+2}|t_2-t_1|^{h'p+2}}\dd s_1\dd s_2\dd t_1\dd t_2
\\&\lesssim\iint_{[0,1]^{4}}|s_2-s_1|^{(\alpha-h)p-2}|t_2-t_1|^{(\beta-h')p-2}((I_{A}^{M})_{(s_1,s_2)(t_1,t_2)})^{\frac{p}{2}}\dd s_1\dd s_2\dd t_1\dd t_2
\end{split}
\end{equation*}
where $h<\alpha$ and $h'<\beta$. Then if $p$ is large enough these last integral go to zero when $M$ go to 
infinity by dominate convergence and this give us the convergence of $A^{NMx}$. Now we will use the same argument for the other terms and for this we must establish the following estimate :
\begin{lemma}\label{extend:G-RR}
let $z$ and $y$ two smooth sheets, $(h,h')\in(0,\infty)^{2}$ and $p>1$ then the following inequalities :
\begin{equation*}
\begin{split}
||B_1^{zz}-B_1^{yy}||^{p}_{2h,h'}&\lesssim_{h,h',p}(U^{2}_{2h,h',p}(B_1^{zz}-B_1^{yy}))^{p}+(U^{3}_{2h,h',h'}(D_2^{zz}-D_2^{yy}))^{p}+||\der_1z\der z-\der_1y\der y||^{p}_{2h,h'}
\\&+||\der z-\der y||^{p}_{h+\frac{2}{p},h'+\frac{2}{p}}(||\der z||_{h+\frac{2}{p},h'+\frac{2}{p}}+||\der y||_{h+\frac{2}{p},h'+\frac{2}{p}})^{p}
\end{split}
\end{equation*}
\begin{equation*}
\begin{split}
||C^{zz}-C^{yy}||^{p}_{2h,2h'}&\lesssim_{h,h',p}(U^{2}_{2h,2h',p}(C^{zz}-C^{yy}))^{p}+(U_{2h+\frac{2}{p},h',h',p}^{3}(D_2^{zz}-D_2^{yy}))^{p}
\\&+(U_{h,h,2h'+\frac{2}{p},p}^{3}(D_1^{zz}-D_1^{yy}))^{p}+(||\der z-\der y||_{h+\frac{2}{p},h'+\frac{2}{p}}(||\der z||_{h+\frac{2}{p},h'+\frac{2}{p}}+||\der y||_{h+\frac{2}{p},h'+\frac{2}{p}}))^{p}
\end{split}
\end{equation*}
$$
||D_1^{zz}-D_{1}^{yy}||^{p}_{2h,2h'}\lesssim_{h,h',p}(U^{3}_{h,h,2h'}(D_1^{zz}-D_1^{yy})^{p}+(||\der z-\der y||_{h+\frac{1}{p},h'}(||\der z||_{h+\frac{1}{p},h'}+||\der y||_{h+\frac{1}{p},h'}))^{p}
$$
hold.
\end{lemma}
\begin{proof}
Let us prove the three first inequalities because  all the others are obtained by the same techniques.  Lemma~\ref{lemma:grr} applied in the first direction and using the fact that $\der_1 B_1^{zz}=\der_1z\der z$ we obtain that  
\begin{equation}
\begin{split}
|(B_1^{zz}-B_1^{yy})_{s_1s_2t_1t_2}|^{p}\lesssim_{h,h',p}&|s_2-s_1|^{2hp}(\int_{[0,1]^{2}}\frac{|(B_1^{zz}-B_1^{yy})_{u_1u_2t_1t_2}|^{p}}{|u_2-u_1|^{2hp+2}}\dd u_1\dd u_2
\\&+||(\der_1z\der z-\der_1y\der y)_{.t_1t_2}||^{p}_{2h})
\end{split}
\end{equation}
Now we have to deal with the other direction. The second term appearing in the right side of this inequality can be estimated by 
$|t_2-t_1|^{h'p}||\der_1x\der x-\der_1y\der y||^{p}_{2h,h'}$ for the first term we need to apply Lemma~\ref{lemma:grr} once again 
\begin{equation}
\begin{split}
\int_{[0,1]^{2}}\frac{|(B_1^{zz}-B_1^{yy})_{u_1u_2t_1t_2}|^{p}}{|u_2-u_1|^{2hp+2}}\dd u_1\dd u_2&\lesssim_{h,h',p}
|t_2-t_1|^{h'p}((U^{2}_{2h,h',p}(B^{zz}_1-B_1^{yy}))^{p}\\&+\int_{[0,1]^{2}}\frac{||\der_2(B^{zz}_1-B^{yy}_1)_{u_1u_2.}||^{p}_{h'}}{|u_2-u_1|^{hp+2}}\dd u_1\dd u_2)
\end{split}
\end{equation}
then it suffices to note that
\begin{equation}\label{equation:bound-GRR1}
\begin{split}
||\der_2(B^{zz}_1-B^{yy}_1)_{u_1u_2.}||^{p}_{h'}&=||(\mu_2D_2^{zz}-\mu_2D_2^{yy})_{u_1u_2.}||^{p}_{h'}
\\&\leq||(D_2^{zz}-D_2^{yy})_{u_1u_2.}||^{p}_{\mathscr C_{2}^{h'}\otimes\mathscr C_{2}^{h'}}
\\&\lesssim_{h,h',p}\int_{[0,1]^{4}}\frac{|(D_2^{zz}-D_2^{yy})_{u_1u_2v_1v_2v_3v_4}|^{p}}{|v_4-v_3|^{h'p+2}|v_2-v_1|^{h'p+2}}\dd v_1\dd v_2\dd v_3\dd v_4
\\&\lesssim_{h,h',p}|u_2-u_1|^{hp+2}((U^{3}_{h+\frac{2}{p},h',h',p}(D_2^{zz}-D_2^{yy}))^{p}\\&+\int_{[0,1]^{4}}\frac{||(\der z\otimes_2\der z-\der y\otimes_2\der y)_{.v_1v_2v_3v_4}||_{h+2/p}^{p}}{|v_4-v_3|^{h'p+2}|v_2-v_1|^{h'p+2}}\dd v_1\dd v_2\dd v_3\dd v_4).
\end{split}
\end{equation}
Now the last term in the right side of this inequality can be bounded by $||\der z-\der y||^{p}_{h+\frac{2}{p},h'+\frac{2}{p}}(||\der z||_{h+\frac{2}{p},h'+\frac{2}{p}}+||\der y||_{h+\frac{2}{p},h'+\frac{2}{p}})^{p}$
then putting  all these bound together gives 
\begin{equation*}
\begin{split}
||B_1^{zz}-B_1^{yy}||^{p}_{2h,h'}&\lesssim_{h,h',p}(U^{2}_{2h,h',p}(B_1^{zz}-B_1^{yy}))^{p}+(U^{3}_{2h,h',h'}(D_2^{zz}-D_2^{yy}))^{p}+||\der_1z\der z-\der_1y\der y||^{p}_{2h,h'}
\\&+||\der z-\der y||^{p}_{h+\frac{2}{p},h'+\frac{2}{p}}(||\der z||_{h+\frac{2}{p},h'+\frac{2}{p}}+||\der y||_{h+\frac{2}{p},h'+\frac{2}{p}})^{p}
\end{split}
\end{equation*}
then we have prove the first inequality. Now by the Lemma~\ref{lemma:grr} once again we have that 
$$
|C^{zz}_{s_1s_2t_1t_2}-C^{yy}_{s_1s_2t_1t_2}|\lesssim_{h,h',p}|s_2-s_1|^{2hp}((U_{2h,p}^{1}(C^{zz}_{.t_1t_2}-C^{yy}_{.t_1t_2})^{p}+||\der_{1}(C^{zz}-C^{yy})_{.t_1t_2}||_{2h}^{p})
$$
with $\der_1C_{zz}=\mu_1D_1^{zz}$ then by same argument has before we have 
\begin{equation}
\begin{split}
||\der_{1}(C^{zz}-C^{yy})_{.t_1t_2}||_{2h}^{p}&\lesssim_{h,h',p}|t_2-t_1|^{2h'p}(((U^{3}_{h,h,2h',p}(D_1^{zz}-D_1^{yy}))^{p}\\&
+||\der z-\der y||^{p}_{h+\frac{2}{p},h'+\frac{2}{p}}(||\der z||_{h+\frac{2}{p},h'+\frac{2}{p}}+||\der y||_{h+\frac{2}{p},h'+\frac{2}{p}})^{p})
\end{split}
\end{equation}
and 
\begin{equation}
\begin{split}
(U_{2h,p}^{1}(C^{zz}_{.t_1t_2}-C^{yy}_{.t_1t_2}))^{p}&\lesssim_{h,h',p}|t_2-t_1|^{2h'p}((U^{2}_{2h,2h',p}(C^{zz}-C^{yy}))^{p}\\&
+\int_{[0,1]^{2}}\frac{||\der_2(C^{zz}-C^{yy})_{u_1u_2.}||_{2h'}^{p}}{|u_2-u_1|^{hp+2}}\dd u_1\dd u_2)
\\&\lesssim_{h,h',p}|t_2-t_1|^{2h'p}((U^{2}_{2h,2h',p}(C^{zz}-C^{yy}))^{p}+(U^{3}_{2h+\frac{2}{p},h',h',p}(D_2^{zz}-D_2^{zz}))^{p}
\\&+||\der z-\der y||_{h+\frac{2}{p},h'+\frac{2}{p}}^{p}(||\der z||_{h+\frac{2}{p},h'+\frac{2}{p}}+||\der z||_{h+\frac{2}{p},h'+\frac{2}{p}})^{p})
\end{split}
\end{equation}
 Putting these two bound together  we obtain the second inequality. This concludes the proof of the Lemma
\end{proof}
 To obtain convergence of all the term of the rough sheet it suffices to take $x=X^{N}$ and $y=X^{M}$ in the lemma~\eqref{extend:G-RR}. Let us give an example :
\begin{equation*}
\begin{split}
&\mathbb E[||C^{NMxx}||^{p}_{2h,2h'}] \lesssim_{h,h',p}
\int_{[0,1]^{4}}|u_2-u_1|^{2(\alpha-h)p-2}|v_2-v_1|^{2(\beta-h')p-2}((I_C^{M})_{(u_1,u_2)(v_1,v_2)})^{\frac{p}{2}}du_1du_2dv_1dv_2
\\&+\int_{[0,1]^{6}}|u_2-u_1|^{2(\alpha-h)p-4}|v_2-v_1|^{(\beta-h')p-2}|v_4-v_3|^{(\beta-h')p-2}((I_{D_2}^{M})_{(u_1,u_2)(v_1,v_2,v_3,v_4)})^{\frac{p}{2}}\dd u_1\dd u_2\dd v_1\dd v_2\dd v_3\dd v_4
\\&+\int_{[0,1]^{6}}|v_2-v_1|^{2(\beta-h')p-4}|u_2-u_1|^{(\alpha-h')p-2}|u_4-u_3|^{(\alpha-h')p-2}((I_{D_1}^{M})_{(u_1,u_2,u_3,u_4)(v_1,v_2)})^{\frac{p}{2}}\dd v_1\dd v_2\dd u_1\dd u_2\dd u_3\dd u_4
\\&+\mathbb E[||A^{NMx}||_{h+2/p,h'+2/p}^{2p}]^{1/2}\mathbb E[(||\der x^{N}||_{h+2/p,h'+2/p}+||\der x^{M}||_{h+2/p,h'+2/p})^{2p}]^{1/2}
\end{split}
\end{equation*}
where we used that our term is in the second chaos of $\hat W$ and the Gaussian hypercontractivity (see for example~\cite{janson}). Then if  $p$ is large enough the three first terms of right side go to zero when $M\to\infty$ by dominate convergence and the last term was already studied this allow us to conclude that $C^{Nxx}$ is a Cauchy sequence in $L^p(\Omega,C\mathscr C_{2,2}^{2h,2h'})$.
The other terms can be treated similarly.
\end{proof}
 This construction and theorem~\eqref{th:integral} allows us to define the two integrals 
 $$
 \iint f'(x)\dd x,\quad\iint f''(x)\dd_1 x\dd_2x
 $$
for a function $f\in C^{10}(\mathbb R)$. Our goal is then to use the continuity result~\eqref{th:cont-int} to obtain the Stratonovich change of variable formula, for that we need an another assumption on the function $f$ which allow us to control the constant $C$ appearing in~\eqref{th:cont-int}. 
\begin{definition}\label{growuthcond}
Let $k\in\mathbb N$, we will say that a function $f\in C^k(\mathbb R)$ satisfies the growth
condition (GC) if there exist positive constants $c$ and $\lambda$
such that
\begin{equation}\label{eq-gc-1}
\lambda<\frac1{4\,\max_{s,t\in[0,1]}\lp R_{s}^{1}R_{t}^{2}\rp},
\quad\mbox{and}\quad
\max_{l=0,. . .,k}|f^{(l)}(\xi)|\le c\,e^{\lambda\,|\xi|^2}\,\ \text{for all }\xi\in\R.
\end{equation}
\end{definition}
And a preliminary result ensures that $x^{N}$ satisfy some uniform exponential integrability. 
\begin{proposition}\label{proposition:conve-reg}
 There exist  $\lambda>0$  such that 
\begin{equation}\label{eq:exp-moments-xn}
\sup_{N\in\mathbb N}\E\left[e^{\lambda\sup_{(s,t)\in[0,1]^{2}}|x^N_{s;t}|^2}\right]<+\infty.
\end{equation}
\end{proposition}
\begin{proof}
See~\cite{Itoplane}.
\end{proof}
Now is not difficult to obtain the following result 
\begin{theorem}\label{theorem:Stratonovich formula-rough-case}
Let $f\in C^{10}(\mathbb R)$ function satisfying the (GC) with a small parameter $\lambda>0$ then 
\begin{equation}\label{eq:Star-form}
\der f(x)=\iint f'\left(x\right)\dd x+\iint f''\left(x\right)\dd_1x\dd_2x
\end{equation}
moreover the following convergence holds 
\begin{equation}\label{equation:app-int}
\iint f'(x^N)\dd x^n\to^{N\to+\infty}\iint f'\left(x\right)\dd x,\quad \iint f''(x^N)\dd_1x^N\dd_2x^N\to^{N\to+\infty}\iint f''(x)\dd_1 x\dd_2x
\end{equation}
in $L^p(\Omega,\mathscr C_{2,2}^{\alpha-\eps,\beta-\eps})$ for $\eps>0$ and $0<\lambda<\lambda(p)$.
\end{theorem}
\begin{proof}
Due to  Theorem~\eqref{th:cont-int} we have that 
$$
\left|\left|\iint f'(x^N)\dd x^N-\iint f'(x)\dd x\right|\right|_{\alpha-\eps,\beta-\eps}\lesssim C^NK(||\mathbb X^N||+||\mathbb X||)||\mathbb X^N-\mathbb X||_{\mathscr H_{\alpha-\eps,\beta-\eps}}
$$
Then using the H\"older inequality we can see that the needed convergence is only due to the fact that 
$$
\E[||\mathbb X^N-\mathbb X||^a_{\alpha-\eps,\beta-\eps}]\to^{N\to+\infty}0,\quad  \sup_N\mathbb E[K(||\mathbb X^N||+||\mathbb X||)^b]<+\infty,\quad \sup_N\mathbb E[(C^N)^c]<\infty
$$
for some $a,b,c>0$, Now the two first affirmation are given by the theorem~\eqref{conv-th} and for the third it suffices to recall that 
$$
C=\sum_{i=1}^{10}\sup_{|\xi|\leq||x||_{\infty}+||x^N||_{\infty}}|f^{(i)}(\xi)|\lesssim e^{\lambda ||x||^2_{\infty}}e^{\lambda||x^N||^2}
$$
where we have used that $f$ satisfy the (GC) and then using H\"older inequality, Fernique's Theorem and ~\eqref{eq:exp-moments-xn} we can see that  $\sup_N\E[(C^N)^c]<+\infty$ which gives the convergence \eqref{equation:app-int}. Now to obtain the formula ~\eqref{eq:Star-form} it suffices to use the fact that 
$$
\mathscr N_{\alpha-\eps,\beta-\eps}(f(x^N)-f(x))\to^{N\to+\infty}0
$$ 
in $L^p(\Omega)$ for $p>1$ due to the Lemma~\eqref{lemma:h-p-r} and the fact that $\sup_{N}\E[(C^N)^p]<+\infty$, and this end the proof.
\end{proof}

\subsection{ The Brownian case}
In~\cite{CW}  the authors give a definition a la It\^o for the multidimensional integral in the case of the Brownian sheet, the aim of this section is to compare their notion of integration with the integral that we defined. More precisely we will show that the two concepts coincide when the rough sheet is understood a la It\^o. In all this subsection $x$ is a Brownian sheet  and 
$\mathbb X^{It\hat o}$ the rough sheet be associated where all the iterated integrals are understood in the sense given in~\cite{CW}.

Of course using the same argument as in the lemma~\ref{extend:G-RR} we can see that our objects satisfy the regularity expected. For example 
\begin{equation}\label{Ito-E-I}
\begin{split}
\expect[||B^{It\hat o xx}||^p_{2h,h'}]&\lesssim_{h,h',p}\int_{[0,1]^4}\frac{\expect[|(B_1^{It\hat oxx})_{u_1u_2v_1v_2}|^p]}{|u_2-u_1|^{2hp+2}|v_2-v_1|^{h'p+2}}\dd u_1d\dd u_2\dd v_1\dd v_2
\\&+\int_{[0,1]^6}\frac{\expect[|(D_2^{It\hat oxx})_{u_1u_2v_1v_2v_3v_4}|^p]}{|u_2-u_1|^{2hp+2}|v_4-v_3|^{h'p+2}|v_2-v_1|^{h'p+2}}\dd u_1\dd u_2\dd u_3\dd u_4\dd v_1\dd v_2
\\&+\expect[||\der_1x\der x||^p_{2h,h'}]+\expect[||\der x||_{h+2/p,h'+2/p}^{2p}]
\end{split}
\end{equation}
but by a simple computation we have that 
$$
\expect[|(B_1^{It\hat oxx})_{u_1u_2v_1v_2}|^p]=c^1_pu_1^{p/2}(u_2-u_1)^{p}(v_2-v_1)^{p/2}
$$
and
$$
\expect[|(D_2^{It\hat oxx})_{u_1u_2v_1v_2v_3v_4}|^p]=c^2_p(u_2-u_1)^p(v_2-v_1)^{p/2}(v_4-v_3)^{p/2}
$$
then if $h,h'<1/2$ and $p$ large enough then the r.h.s of the equation \label{Ito-E-I} is finite and this gives that $B_1^{It\hat oXX}\in\CCC_{2,2}^{2h,h'}$. To compare the two definition of integration we will show first that the two definition of boundary integral coincide.
\begin{proposition}\label{prop:boundry}
Let $\varphi\in C^{2}(\mathbb R)$ satisfying the (GC) then the integral 
$$
I^{It\hat o,b1}_{s_1s_2t_1t_2}=\int_{s_1}^{s_2}\varphi(x_{st_1})\int_{t_1}^{t_2}\hat \dd_{st}x_{st}
$$
with $\hat \dd$ is the It\^o differential admit a continuos version which coincide with 
$$
I^{Rough,b1}_{s_1s_2t_1t_2}=\int_{s_1}^{s_2}\varphi(x_{st_1})\int_{t_1}^{t_2}\dd_{st}x_{st}
$$
where $I^{rough,X,b1}$ are given by the Proposition~\ref{boundry}.
\end{proposition}
\begin{proof}
Let $\pi=(s_i)_i$ a dissection of the interval $[s_1,s_2]$. Now by definition we have that 
$$
I^{It\hat ox,b_1}_{s_1s_2t_1t_2}=\mathbb P-\lim_{|\pi|\to0}\sum_{i}\varphi(x_{s_it_1})\der x_{s_is_{i+1}t_1t_2}
$$
and
$$
I^{Rough,x,b_1}_{s_1s_2t_1t_2}=a.s-\lim_{|\pi|\to0}\sum_{i}\varphi(x_{s_it_1})\der x_{s_is_{i+1}t_1t_2}+\varphi'(x_{s_it_1})(B^{It\hat o,xx}_1)_{s_is_{i+1}t_1t_2}
$$
but 
\begin{equation*}
\begin{split}
\expect [|\sum_{i}\varphi'(x_{s_it_1})(B^{It\hat oxx}_1)_{s_is_{i+1}t_1t_2}|^2]&=\sum_{i}\expect[\varphi'(x_{s_it_1})^2]\expect[|(B^{It\hat o,xx}_1)_{s_is_{i+1}t_1t_2}|^2]
\\&=1/2t_1(t_2-t_1)\sum_{i}\expect[\varphi'(x_{s_it_1})^2](s_{i+1}-s_{i})^2
\\&\leq t_1(t_2-t_1)(s_2-s_1)|\pi|\sup_{s\in[0,1]^{2}}\expect[\varphi'(x_{st_1})^2].
\end{split}
\end{equation*}
Then it suffices to remark that the term appearing in the r.h.s vanish when $|\pi|$ go to zero and this finishes the proof.
\end{proof}
\begin{proposition}\label{l2-conv}
Let $\varphi\in C^{3}(\mathbb R)$ satisfying the (GC) and $\Pi=\{(s_i,t_j)\}_{ij}$ a dissection of the rectangle 
$[s_1,s_2]\times[t_1,t_2]$ then we have :
$$
 L^2(\Omega)-\lim_{|\Pi|\to0}\sum_{i,j}\varphi(x_{s_it_j})C^{It\hat o,xx}_{s_is_{i+1}t_jt_{j+1}}=0
$$
$$
 L^2(\Omega)-\lim_{|\Pi|\to0}\sum_{i,j}\varphi(x_{s_it_j})C^{It\hat o,\omega x}_{s_is_{i+1}t_jt_{j+1}}=0
$$
and
$$
 L^2(\Omega)-\lim_{|\Pi|\to0}\sum_{i,j}J^{Rough,xx,b_a}_{s_is_{i+1}t_jt_{j+1}}=0
$$
$$
L^2(\Omega)-\lim_{|\Pi|\to0}\sum_{i,j}J^{Rough,\omega x,b_a}_{s_is_{i+1}t_jt_{j+1}}=0
$$
where $J^{Rough,xx,b_a}=\int_a\varphi'(x)\dd_ax\int_{\hat a}\dd x\dd x$ and $J^{Rough,\omega x,b_a}=\int_a\varphi'(x)\dd_ax\int_{\hat a}\dd\omega \dd x$ are given by the Proposition~\ref{boundry} 
\end{proposition}
\begin{proof}
We will prove only the first and third statement, for the two other we have exactly the same proof. Now definition we have that $C^{Ito,XX}_{s_1s_2t_1t_2}=\iint_{(s_1,t_1)}^{(s_2,t_2)}\iint_{(s_1,t_1)}^{(s,t)}\hat \dd_{uv}x_{uv}\hat \dd_{st}x_{st}$ and then by independence of the increment of the Brownian sheet we have that 
\begin{equation*}
\begin{split}
\expect[|\sum_{i,j}\varphi(x_{s_it_j})C^{It\hat o,xx}_{s_is_{i+1}t_jt_{j+1}}|^2]&=\sum_{i,j}1/4\expect[\varphi(x_{s_it_j})^2](s_{i+1}-s_i)^2(t_{j+1}-t_j)^2
\\&\lesssim |\Pi|\sup_{(s,t)\in[0,1]^2}\expect[\varphi(x_{st})^2]
\end{split}
\end{equation*}
This give us the first convergence. Now we will focus on the third convergence which require more work. In fact by definition 
\begin{equation*}
\begin{split}
J^{Rough,xx,b_1}_{s_is_{i+1}t_jt_{j+1}}&=\Lambda_1[\der_1\varphi(x)C^{It\hat o,xx}+\varphi'(X)\mu_1E_1^{It\hat o,xxx}
\\&+\mu_1(\Lambda_1\otimes_11)(\varphi(x)^{\sharp1}D_1^{It\hat o,xx}+\der_1\varphi'(x)E_1^{It\hat o,xx})]_{s_is_{i+1}t_jt_{j+1}}
\\&=\Lambda^{1d}[\der_1\varphi(x)_{.t_j}C^{It\hat o,xx}_{.t_jt_{j+1}}+
\varphi'(x)_{.t_j}(\mu_1E_1^{It\hat o,xxx})_{.t_jt_{j+1}}
\\&+\mu_1(\Lambda_1\otimes_11)(\varphi(X)^{\sharp1}_{.t_j}(D_1^{It\hat  o,xx})_{.t_jt_{j+1}}+\der_1\varphi'(X)_{.t_j}(E_1^{It\hat o,xx})_{.t_jt_{j+1}})]_{s_is_{i+1}}
\end{split}
\end{equation*}
and then we have the bound 
\begin{equation}
\begin{split}
|J^{Rough,xx,b_1}_{s_is_{i+1}t_jt_{j+1}}|&\lesssim_h(s_{i+1}-s_i)^{3h}(||\der_1\varphi(X)_{.t_j}||_{h}||C^{It\hat o,xx}_{.t_jt_{j+1}}||_{2h}+(\sup_{(s,t)\in[0,1]^2}|\varphi'(x_{st})|)||(E_1^{It\hat o,xxx})_{.t_jt_{j+1}}||_{\CC_{2}^{2h}\otimes\CC_{2}^{h}}
\\&+||\varphi(x)^{\sharp1}_{.t_{j}}||_{2h}||(D_1^{It\hat o,xx})_{.t_jt_{j+1}}||_{\CC^h\otimes\CC^h}
+||\der_1\varphi'(x)_{.t_j}||_{h}||(E_1^{It\hat o,xxx})_{.t_jt_{j+1}}||_{\CC^{2h}\otimes\CC^h})
\end{split}
\end{equation}
By independence of increments in the second direction, this gives:
\begin{equation}
\expect[|\sum_{i,j}J^{Rough,xx,b_a}_{s_is_{i+1}t_jt_{j+1}}|^{2}]\lesssim_h|\sum_i(s_{i+1}-s_{i})^{3h}|^2(a_1+a_2+a_3+a_4)
\end{equation}
where 
\begin{enumerate}
\item $a_1=\sum_j\expect[||\der_1\varphi(x)_{.t_j}||_{h}^2]\expect[||C^{It\hat o,xx}_{.t_jt_{j+1}}||_{2h}^2]$
\item $a_2=\expect[\sup_{(s,t)\in[0,1]^2}|\varphi'(x_{st})|^2]\sum_j\expect[||(E_1^{It\hat o,xxx})_{.t_jt_{j+1}}||^2_{\CC_{2}^{2h}\otimes\CC_{2}^{h}}]$
\item $a_3=\sum_j\expect[||\varphi(x)^{\sharp1}_{.t_{j}}||_{2h}^2]\expect[||(D_1^{It\hat o,xx})_{.t_jt_{j+1}}||^2_{\CC^h\otimes\CC^h}]$
\item $a_4=\sum_j\expect[||\der_1\varphi'(x)_{.t_j}||_{h}^2]\expect[||(E_1^{It\hat o,xxx})_{.t_jt_{j+1}}||^2_{\CC^{2h}\otimes\CC^h}]$
\end{enumerate}
To obtain our convergence it suffices to show that all these terms are bounded. A simple computation gives : 
$$
\expect[||\der_1\varphi(x)_{.t_j}||^2]\lesssim\expect[\sup_{(s,t)\in[0,1]^2}|\varphi'(x_{st})|^4]^{1/2}\expect[||\der_1x||_{h}^4]^{1/2}
$$
where the r.h.s is finite. Now the Lemma~\ref{lemma:grr} gives 
\begin{equation*}
||C^{Ito,xx}_{.t_jt_{j+1}}||_{2h}^{p}\lesssim_{h,p}\int_{[0,1]^2}\frac{|C^{It\hat o,xx}_{u_1u_2t_jt_{j+1}}|^p}{|u_2-u_1|^{2hp+2}}\dd u_1\dd u_2+\int_{[0,1]^4}\frac{|(D_1^{It\hat o,xx})_{u_1u_2u_3u_4t_jt_{j+1}}|^p}{|u_4-u_3|^{hp+2}|u_2-u_1|^{hp+2}}\dd u_1\dd u_2\dd u_3\dd u_4
\end{equation*}
then taking the expectation in this last equality and using Jensen inequality  we obtain that 
$$
\expect[||C^{Ito,XX}_{.t_jt_{j+1}}||_{2h}^2]\lesssim_{h}c_p(t_{j+1}-t_j)^2<\infty
$$
where $c_p<+\infty $ for $p$ large enough and then 
$$
a_1\lesssim_{h,p}\expect[\sup_{(s,t)\in[0,1]^2}|\varphi'(x_{st})|^4]^{1/2}\expect[||\der_1x||_{h}^4]^{1/2}\sum_j(t_{j+1}-t_j)^{2}
$$  
the terms $a_2,a_3$ and $a_4$ can be treated similarly and then we obtain the wanted convergence of the boundary integral .
\end{proof}
\begin{corollary}
Let $\varphi\in C^{5}(\mathbb R)$ satisfying the (GC) then the integral
$$
\mathscr J^{It\hat o}_{s_1s_2t_1t_2}=\int_{s_1}^{s_2}\int_{t_1}^{t_2}\varphi(x_{st})\hat \dd_{st}x_{st}
$$
where $\hat d$ is a Ito differential admit a continuous version with coincide with
$$
\mathscr J^{Rough}_{s_1s_2t_1t_2}=\int_{s_1}^{s_2}\int_{t_1}^{t_2}\varphi(x_{st}) \dd_{st}x_{st}
$$
where this last integral are given by Theorem~\ref{th:integral}
\end{corollary}
\begin{proof}
let $\Pi=\{(s_i,t_j)\}_{ij}$ a dissection of the rectangle 
$[s_1,s_2]\times[t_1,t_2]$ then by definition we have that 
\begin{equation}
\begin{split}
\mathscr J^{Rough}_{s_1s_2t_1t_2}&=\sum_{ij}(-\varphi(x_{s_it_j})\der x_{s_is_{i+1}t_jt_{j+1}}+\varphi'(x_{s_it_j})C^{It\hat o,xx}_{s_is_{i+1}t_jt_{j+1}}+\varphi'(x_{s_it_j})C^{It\hat o,\omega x}_{s_is_{i+1}t_jt_{j+1}}
\\&+\sum_{a=1,2}(I^{Rough,x,b_a}_{s_is_{i+1}t_jt_{j+1}}+J^{Rough,xx,b_a}_{s_is_{i+1}t_jt_{j+1}}
+J^{Rough,\omega x,b_a}_{s_is_{i+1}t_jt_{j+1}})+r^{\flat}_{s_is_{i+1}t_jt_{j+1}})
\end{split}
\end{equation}
where $r^{\flat}\in\CCC_{2,2}^{1+,1+}$ this fact combined with the proposition~\eqref{l2-conv}
and~\eqref{prop:boundry} give us 
$$
\mathscr J^{Rough}_{s_1s_2t_1t_2}=\mathbb P-\lim_{|\Pi|\to0}\sum_{ij}-\varphi(x_{s_it_j})\der x_{s_is_{i+1}t_jt_{j+1}}+I^{It\hat o,x,b_1}_{s_is_{i+1}t_jt_{j+1}}+I^{It\hat o,x,b_2}_{s_is_{i+1}t_jt_{j+1}}
$$
Now this last converge to the Ito integral in fact 
$$
\expect[|\mathscr J^{It\hat o}_{s_1s_2t_1t_2}-\sum_{ij}I^{It\hat o,x,b_1}_{s_is_{i+1}t_jt_{j+1}}|^2]\lesssim\expect[\sup_{(s,t)\in[0,1]^2}|\varphi(x_{st})|^4]^{1/2}\expect[\sup_{|t-t'|+|s-s'|\leq|\Pi|}|x_{st}-x_{s't'}|^4]^{1/2}
$$
which the r.h.s vanish when the mesh of the partition go to zero then we have the result.
\end{proof}


\begin{bibdiv}
\begin{biblist}

\bib{bony}{article}{
author={J.-M. Bony}, 
title= {Calcul symbolique et propagation des singularités
pour les équations aux dérivées partielles
non linéaires}, 
journal={Ann. Sci. École Norm. Sup.}, 
number={14}, 
year={1981},
pages={209–246},
}







\bib{CW}{article}{
   author={Cairoli, R.},
   author={Walsh, John B.},
   title={Stochastic integrals in the plane},
   journal={Acta Math.},
   volume={134},
   date={1975},
   pages={111--183},
   issn={0001-5962},
   review={\MR{0420845 (54 \#8857)}},
}

\bib{Itoplane}{unpublished}{
title={SKOROHOD AND STRATONOVICH INTEGRATION IN THE PLANE}
author={K.Chouk}
author={S.Tindel}
url= {http://arxiv.org/abs/1309.6604},
}
\bib{grr}{article}{
author={A. M. Garsia}, 
author={E. Rodemich}, 
author={H. Rumsey}
title={A real variable lemma and the continuity
of paths of some Gaussian processes}, 
journal={Indiana Univ. Math. J. 20}
year={1970},
number={6}, 
pages={565–578},
}



	


\bib{samyandme}{article}{
   author={Gubinelli, Massimiliano},
   author={Tindel, Samy},
   title={Rough evolution equations},
   journal={Ann. Probab.},
   volume={38},
   date={2010},
   number={1},
   pages={1--75},
   issn={0091-1798},
   review={\MR{2599193}},
   doi={10.1214/08-AOP437},
}
		

\bib{Ramification}{article}{
   author={Gubinelli, Massimiliano},
   title={Ramification of rough paths},
   journal={J. Differential Equations},
   volume={248},
   date={2010},
   number={4},
   pages={693--721},
   issn={0022-0396},
   review={\MR{2578445}},
   doi={10.1016/j.jde.2009.11.015},
}
		

\bib{glt}{article}{
   author={Gubinelli, Massimiliano},
   author={Lejay, Antoine},
   author={Tindel, Samy},
   title={Young integrals and SPDEs},
   journal={Potential Anal.},
   volume={25},
   date={2006},
   number={4},
   pages={307--326},
   issn={0926-2601},
   review={\MR{2255351 (2007k:60182)}},
   doi={10.1007/s11118-006-9013-5},
}
		

\bib{[Gubinelli-2004]}{article}{
   author={Gubinelli, M.},
   title={Controlling rough paths},
   journal={J. Funct. Anal.},
   volume={216},
   date={2004},
   number={1},
   pages={86--140},
   issn={0022-1236},
   review={\MR{2091358 (2005k:60169)}},
   doi={10.1016/j.jfa.2004.01.002},
}
		
\bib{Para}{unpublished}{
	title = {Paracontrolled distributions and singular PDEs},
	url = {http://arxiv.org/abs/1210.2684},
	journal = {{arXiv} preprint {arXiv:1210.2684}},
	author = {Gubinelli, M.},
	author={Imkeller, P.},
	author={Perkowski, N.},
	year = {2012},
}


\bib{[LionsStFlour]}{book}{
   author={Lyons, Terry J.},
   author={Caruana, Michael},
   author={L{\'e}vy, Thierry},
   title={Differential equations driven by rough paths},
   series={Lecture Notes in Mathematics},
   volume={1908},
   note={Lectures from the 34th Summer School on Probability Theory held in
   Saint-Flour, July 6--24, 2004;
   With an introduction concerning the Summer School by Jean Picard},
   publisher={Springer},
   place={Berlin},
   date={2007},
   pages={xviii+109},
   isbn={978-3-540-71284-8},
   isbn={3-540-71284-4},
   review={\MR{2314753 (2009c:60156)}},
}
		

\bib{LyonsBook}{book}{
   author={Lyons, Terry},
   author={Qian, Zhongmin},
   title={System control and rough paths},
   series={Oxford Mathematical Monographs},
   note={Oxford Science Publications},
   publisher={Oxford University Press},
   place={Oxford},
   date={2002},
   pages={x+216},
   isbn={0-19-850648-1},
   review={\MR{2036784 (2005f:93001)}},
   doi={10.1093/acprof:oso/9780198506485.001.0001},
}
	
\bib{Lyons1998}{article}{
   author={Lyons, Terry J.},
   title={Differential equations driven by rough signals},
   journal={Rev. Mat. Iberoamericana},
   volume={14},
   date={1998},
   number={2},
   pages={215--310},
   issn={0213-2230},
   review={\MR{1654527 (2000c:60089)}},
}
	
\bib{Friz}{book}{
   author={Friz, Peter K.},
   author={Victoir, Nicolas B.},
   title={Multidimensional stochastic processes as rough paths},
   series={Cambridge Studies in Advanced Mathematics},
   volume={120},
   note={Theory and applications},
   publisher={Cambridge University Press},
   place={Cambridge},
   date={2010},
   pages={xiv+656},
   isbn={978-0-521-87607-0},
   review={\MR{2604669}},
}



\bib{hajek_stochastic_1982}{article}{
	title = {Stochastic equations of hyperbolic type and a two-parameter Stratonovich calculus},
	volume = {10},
	issn = {0091-1798},
	url = {http://www.ams.org/mathscinet-getitem?mr=647516},
	number = {2},
	journal = {The Annals of Probability},
	author = {Hajek, Bruce},
	year = {1982},
	pages = {451?463},
}

\bib{Twisted}{article}{
   author={Norris, J. R.},
   title={Twisted sheets},
   journal={J. Funct. Anal.},
   volume={132},
   date={1995},
   number={2},
   pages={273--334},
   issn={0022-1236},
   review={\MR{1347353 (96f:60094)}},
   doi={10.1006/jfan.1995.1107},
}
\bib{Towghi}{article}{
   author={Towghi, Nasser},
   title={Multidimensional extension of L. C. Young's inequality},
   journal={JIPAM. J. Inequal. Pure Appl. Math.},
   volume={3},
   date={2002},
   number={2},
   pages={Article 22, 13 pp. (electronic)},
   issn={1443-5756},
   review={\MR{1906391 (2003c:26035)}},
}
	
 \bib{janson}{book}{,
	title = {Gaussian Hilbert Spaces},
	isbn = {9780521561280},
	language = {en},
	publisher = {Cambridge University Press},
	author = {Janson, Svante},
	year = {1997},
	keywords = {Mathematics / Calculus, Mathematics / Mathematical Analysis, Mathematics / Probability \& Statistics / General}
}


\bib{hairer_theory_2013}{article}{
	title = {A theory of regularity structures},
	url = {http://arxiv.org/abs/1303.5113},
	urldate = {2013-06-10},
	journal = {Inventiones mathematicae},
	author = {Hairer, Martin},
	year = {2013}
}


\bib{Wave}{article}{
   author={Quer-Sardanyons, Llu{\'{\i}}s},
   author={Tindel, Samy},
   title={The 1-d stochastic wave equation driven by a fractional Brownian
   sheet},
   journal={Stochastic Process. Appl.},
   volume={117},
   date={2007},
   number={10},
   pages={1448--1472},
   issn={0304-4149},
   review={\MR{2353035 (2008j:60152)}},
   doi={10.1016/j.spa.2007.01.009},
}
\bib{TVp1}{article}{
	title = {It\^o Formula and Local Time for the Fractional Brownian Sheet},
	volume = {8},
	issn = {1083-6489},
	url = {http://ejp.ejpecp.org/article/view/155},
	doi = {10.1214/EJP.v8-155}
	number = {0},
	journal = {Electronic Journal of Probability}
	author = {Tudor, Ciprian A.}
	author =  {Viens, Frederi G}
	year = {2003}
}

\bib{TVp2}{article}{
	title = {It\^o formula for the two-parameter fractional Brownian motion using the extended divergence operator}
	volume = {78}
	issn = {1744-2508}
	url = {http://www.tandfonline.com/doi/abs/10.1080/17442500601014912},
	doi = {10.1080/17442500601014912},
	number = {6}
	journal = {Stochastics An International Journal of Probability and Stochastic Processes},
	author = {Tudor, Ciprian A.}
	author =  {Viens, Frederi G.}
	year = {2006},
	pages = {443--462}
}


\bib{Young}{article}{
   author={Young, L. C.},
   title={An inequality of the H\"older type, connected with Stieltjes
   integration},
   journal={Acta Math.},
   volume={67},
   date={1936},
   number={1},
   pages={251--282},
   issn={0001-5962},
   review={\MR{1555421}},
   doi={10.1007/BF02401743},
}
\bib{ST94}{article}{
author={Samorodnitsky, G}
author={Taqqu, M} 
title={Stable non-Gaussian random
processes.} 
journal={Chapman and Hall}
}



\end{biblist}
 \end{bibdiv}
\end{document}